\documentclass[11pt]{amsart}
\usepackage{amsmath,amssymb,amsthm,graphicx}
\usepackage[left=1in,right=1in,top=1in, bottom=1in]{geometry}
\usepackage{cite}
\usepackage[british]{babel}
\usepackage{hyperref} 
\usepackage{enumitem}
\usepackage{mathtools}

\hypersetup{
    colorlinks=false,
    pdfborder={0 0 0},
}
\theoremstyle{plain}
\newtheorem{thm}{Theorem}[section]
\newtheorem{lem}[thm]{Lemma}
\newtheorem{prop}[thm]{Proposition}
\newtheorem{cor}[thm]{Corollary}

\newtheorem*{assumption*}{Assumption}

  \newtheorem{example}[thm]{Example}

\theoremstyle{remark}
\newtheorem{rem}[thm]{Remark}

\numberwithin{equation}{section}

\newcommand{\mathbbm}[1]{\mathbf{#1}}
\newcommand{\ind}{{\mathbbm{1}}}
\newcommand{\1}[1]{{\ind\mkern -1.5mu}{\{#1\}}}
\DeclareMathOperator{\Hessian}{Hess}
\DeclareMathOperator{\E}{\mathbb E}
\renewcommand{\P}{\mathbb P}

\renewcommand{\epsilon}{\varepsilon}

\newcommand{\TV}{\mathrm{TV}}

\newcommand{\eps}{\varepsilon}

\newcommand{\ud}{{\mathrm d}}

\newcommand{\R}{{\mathbb R}}

\newcommand{\N}{{\mathbb N}}

\newcommand{\RP}{{\mathbb R}_+}

\newcommand{\cB}{{\mathcal B}}
\newcommand{\cC}{{\mathcal C}}

\newcommand{\cF}{{\mathcal F}}

\newcommand{\cP}{{\mathcal P}}

\newcommand{\cT}{{\mathcal T}}
\newcommand{\cW}{{\mathcal W}}
\newcommand{\cX}{{\mathcal X}}

\newcommand{\Id}{\text{Id}}

\makeatletter
\def\namedlabel#1#2{\begingroup  
    #2%
    \def\@currentlabel{#2}%
    \phantomsection\label{#1}\endgroup
}
\makeatother

\newlist{myenumi}{enumerate}{10}
\setlist[myenumi]{leftmargin=0pt, labelindent=\parindent, listparindent=\parindent, labelwidth=0pt, itemindent=!, itemsep=1pt, parsep=4pt}

\newlist{thmenumi}{enumerate}{10}
\setlist[thmenumi]{leftmargin=0pt, labelindent=\parindent, listparindent=\parindent, labelwidth=0pt, itemindent=!}

\begin{document}

\title[$\mathbf L$-drift conditions and ergodic CLTs]
{Ergodic Central limit theorem for subgeometric Markov chains}
%with applications to unadjusted sampling algorithms}
%:
%necessary conditions for central limit theorems and applications to unadjusted algorithms}
\author{Miha Bre\v{s}ar}
\address{School of Data Science, The Chinese University of Hong Kong, Shenzhen, China} 
\email{mihabresar@cuhk.edu.cn}
\author{Aleksandar Mijatovi\'c}
\address{Department of Statistics, University of Warwick, UK}
\email{a.mijatovic@warwick.ac.uk}
\author{Gareth Roberts}
\address{Department of Statistics, University of Warwick, UK}
\email{gareth.o.roberts@warwick.ac.uk}
\date{\today}

\subjclass[2020]{
Primary  60J05%  	Discrete-time Markov processes on general state spaces
,  60F05%  	Central limit and other weak theorem
; Secondary 
%  60J20% Applications of Markov chains and discrete-time Markov processes on general state spaces 
  60J22%  	Computational methods in Markov chains [See also 65C40]
, 65C40%  	Numerical analysis or methods applied to Markov chains
%, 65C05% Monte Carlo methods
}

\keywords{Discrete-time Markov processes on general state spaces,
central limit theorem for ergodic averages,
necessary conditions for the central limit theorem,
$\mathbf L$-drift conditions,
return times,
lower bounds on convergence rates,
total variation,
Wasserstein distance,
unadjusted Langevin algorithm,
Euler schemes with infinite-variance increments}

\begin{abstract}
We establish verifiable necessary conditions for central limit
theorems (CLTs) of ergodic averages of Markov chains on general
state spaces. Our criteria are based on discrete-time
$\mathbf{L}$-drift conditions that can be verified through
estimates for the one-step transition kernel of the original
chain. They yield an integral criterion for failure of the CLT
for a class of observables that includes unbounded functions
and indicators of tail events.

The main technical step transfers lower bounds for excursions
of the original chain to regenerative excursions of its Nummelin
split chain. The same drift conditions yield lower bounds on
the tails of return times and modulated excursion functionals,
on invariant tails, and on rates of convergence in
$f$-variation and $L^p$-Wasserstein distance. These estimates
apply without an explicit formula for the invariant distribution.

Applications to unadjusted Langevin algorithms show preservation
of polynomial target tail indices at every fixed step size.
For Euler schemes with additive infinite-variance noise, we
establish restrictions on invariant tails and transience under
 superlinear drifts. Schemes with linear drift,
multiplicative noise and logarithmic innovation tails exhibit
sharp polynomial convergence exponents and exact CLT thresholds,
despite the absence of positive Euclidean moments.
\end{abstract}

\maketitle

\section{Introduction}
\label{sec:introduction}
The central limit theorem (CLT) for ergodic averages of Markov chains
plays a fundamental role in probability theory and its applications.
There are numerous approaches to establishing the CLT in this setting,
ranging from martingale methods and the Poisson
equation~\cite{MR834478,MR2981426} to classical drift
conditions~\cite{MR2446322}. A complementary problem is to identify
necessary conditions that can be checked directly from the transition
kernel of the chain. Building on the ideas behind the \textbf{L}-drift
conditions introduced in~\cite{brešar2024subexponential} for
continuous-time processes, we develop such criteria in discrete time.
We apply them to unadjusted Langevin-type sampling algorithms and to
Euler schemes driven by heavy-tailed L\'evy noise. The sampling
applications are motivated by the central role of finite asymptotic
variance in the error analysis of ergodic
estimators~\cite{MR3235677,MR4749108}. The Euler schemes also illustrate
how excursion behaviour links invariant tails, convergence rates and
the validity of the CLT.

\subsection{Lyapunov drift and subgeometric ergodicity}

Foster--Lyapunov theory connects the local dynamics of a Markov chain
to its recurrence, equilibrium and fluctuations. For chains on general
state spaces, a negative drift of an unbounded function $V$ outside
a small set yields bounds on return times, stationary moments and
convergence to equilibrium~\cite{meynadntweedie,MarkovChains}.
Subgeometric drift conditions quantify these connections when the
restoring force weakens in the tails, giving polynomial and more
general convergence rates~\cite{MR1285459,MR1890063,MR2071426},
as well as sufficient conditions for the CLT~\cite{MR2446322}.

A central difficulty in this regime is to determine whether a
Lyapunov estimate captures the actual behaviour of the chain. Indeed for geometrically ergodic Markov chains, bounds on convergence rates are often extremely conservative, for example see \cite{meyn1994computable, MR1340509}.
%A slow upper bound may reflect long excursions into the tails or a suboptimal choice of $V$. 
Remarkably, in the far more delicate sub-geometrically ergodic case, it has long been conjectured that upper bounds are far more realistic, for example see \cite{Roberts07}. However up until now, 
there has been very little generic theory for lower bounds to complement the extensive upper bound theory (although see \cite{Aspandiiarov96} for criteria for the infiniteness of passage-time
moments in the null-recurrent case and \cite{MR1996270} for necessary moment conditions for polynomial ergodicity). Such a theory is essential for identifying sharp rates of
convergence for subgeometric Markov chains.

Our $\mathbf{L}$-drift conditions develop this perspective by
controlling the drift of the bounded reciprocal $1/V$. An upper
bound on its positive drift limits how quickly the chain can return
from large values of $V$; combined with a lower bound on excursion
heights, it yields lower bounds on return-time tails. This provides
a common framework for necessary conditions for the ergodic CLT
and lower bounds on invariant tails and convergence rates.
Comparison with classical upper bounds then allows the quality
of a Lyapunov function to be assessed: matching estimates show
that it captures the excursions responsible for slow convergence
and the breakdown of the CLT.

\subsection{Main contributions}

Our main contribution, Theorem~\ref{thm:CLT}, identifies a readily
verifiable obstruction to the CLT for the ergodic average $S_n(g)$
in~\eqref{eq:ergodic_average}. The classical excursion characterisation
of the finiteness of the asymptotic variance is expressed through
regeneration times of the Nummelin split chain~\cite{Chen99,MR776608}.
Our \textbf{L}-drift conditions in
Assumption~\nameref{sub_drift_conditions} are formulated entirely in
terms of the transition kernel of the original chain. The main
technical step is to transfer lower bounds for excursions of the
original chain to excursions of its split chain.
Theorem~\ref{thm:CLT} gives an integral criterion under which the CLT
fails when the centred function $g-\pi(g)$ is bounded below in the
tails by a positive non-decreasing function of a Lyapunov function $V$.
This applies to unbounded functions as well as bounded ones, including
indicators of tail events.

Our excursion estimates yield three further results.
Theorem~\ref{thm:modulated_moments} gives explicit lower bounds on the
tails of return times and modulated moments of the chain.
Theorem~\ref{thm:invariant} gives lower bounds on the tails of its
invariant distribution, including when that distribution is not known
explicitly. This is typically the situation for unadjusted sampling
algorithms and Euler schemes with state-dependent noise.
Finally, Theorem~\ref{thm:f_rate} converts invariant-tail estimates into
lower bounds on convergence in $f$-variation and $L^p$-Wasserstein
distance. In the polynomially ergodic examples considered below, such
as the unadjusted Langevin algorithm~\cite{Durmus16}, these bounds
identify the same polynomial exponents as classical upper-bound drift
conditions, up to arbitrarily small power losses.

\subsubsection{Applications to Langevin-type sampling algorithms}

We first study unadjusted Langevin-type algorithms for sampling from
heavy-tailed target distributions. Such algorithms are generally
biased: for a fixed step size $h>0$, the invariant distribution
$\pi_h$ of the chain typically differs from the intended target $\pi$.
Since $\pi_h$ is usually not known explicitly, two distinct questions
arise: how rapidly the chain converges to $\pi_h$, and how the tails
of $\pi_h$ compare with those of $\pi$. Verifying the
$\mathbf{L}$-drift conditions through a common one-step analysis allows
our theory to address both questions.

For the unadjusted Langevin algorithm (ULA), obtained by discretising
a Langevin diffusion with a polynomial-tailed target,
Theorem~\ref{thm:ULA} gives necessary conditions for the CLT, identifies
the polynomial convergence exponents, and determines the tail index
of $\pi_h$. Under the assumptions of that theorem, $\pi_h$ has the
same polynomial tail index as $\pi$ for every fixed $h>0$.
Preservation of the tail index nevertheless coexists with slow
returns from the tails and failure of the CLT for suitable ergodic
averages. Moreover, the polynomial exponents of convergence in
$f$-variation and $\cW_p$-Wasserstein distance agree with those of the
underlying Langevin diffusion~\cite[Sec.~3.1]{brešar2024subexponential}.

The slow convergence of ULA motivates replacing Gaussian increments
by increments of infinite variance. Such schemes arise from
L\'evy-driven SDEs that can be exponentially ergodic while having
heavy-tailed invariant distributions~\cite{Majka17,Majka21,MR4580902}.
However, Theorem~\ref{thm:levy_driven_discretisation} identifies an
obstruction for their explicit Euler discretisations with additive
noise. If the innovations have polynomial tail index
$\alpha\in(1,2)$, every invariant distribution of the chain has tails
at least as heavy as a constant multiple of $r^{-\alpha}$ and hence
has infinite moments of all orders $p\ge\alpha$, including the second
moment. If the drift satisfies $|b(x)|/|x|\to\infty$ as
$|x|\to\infty$, the chain instead escapes to infinity almost surely
from every starting point. Both conclusions hold for every fixed
step size $h>0$.

For the class of stable-driven samplers discussed in
Section~\ref{subsec:discretisation_levy_algorithm}, these results
distinguish two regimes. Under the stated assumptions on the target
and its associated drift, a polynomial target tail index $v>\alpha$
leads to transience of the Euler chain, whereas $0<v<\alpha$ leads
to polynomial lower bounds on convergence to the chain's invariant
distribution. Thus changing the innovation law does not by itself
guarantee faster convergence or preservation of the target tails.
The transience result is a heavy-tailed counterpart
of~\cite[Thm.~3.2]{MR1440273}, which shows that the Euler--Maruyama
discretisation of a Langevin diffusion targeting a distribution with
lighter than Gaussian tails is transient. In both settings, reducing
the fixed step size does not remove the obstruction.

\subsubsection{Linear drift, multiplicative noise and logarithmic tails}

Section~\ref{subsec:log-tails_SDE_chain} extends the applications beyond
sampling algorithms with a prescribed target and beyond innovations
with a finite first absolute moment. We study multidimensional Euler
schemes with a contracting linear drift and a bounded, uniformly
non-degenerate, state-dependent dispersion coefficient. The driving
L\'evy increments have tails of order $(\log r)^{-\kappa}$, with
$\kappa>1$, and possess no positive Euclidean moments.
Theorem~\ref{thm:log-euler} proves positive Harris recurrence,
identifies the invariant logarithmic tail index as $\kappa-1$, and
gives matching polynomial convergence
exponents in weighted variation, up to arbitrarily small power
losses. Thus even a contracting linear update can coexist with slow
convergence when sufficiently heavy innovations produce long
excursions.

In this example the theory also gives an exact CLT threshold.
For locally bounded observables equal to $V^s$ outside a ball, where
$V(x)=1+\log(1+|x|)$ and $0<2s<\kappa-1$, the CLT holds precisely
when $2s<\kappa-2$ (equality belongs to the failure regime).
For indicators of the complement of a ball, the CLT holds precisely
when $\kappa>2$.
Theorem~\ref{thm:log-euler} further distinguishes the metrics in which
convergence can be measured: the Euclidean Wasserstein distances
between finite-time and invariant laws are infinite, whereas
logarithmic ground metrics yield finite distances and matching
polynomial convergence exponents in their admissible moment range.
These conclusions remain available with state-dependent noise,
without an explicit formula for the invariant distribution.

\subsection{Structure of the paper}
Section~\ref{sec:main_results} introduces the discrete-time
$\mathbf{L}$-drift conditions and states the general results on
CLT failure, return times, invariant tails and convergence rates.
It also explains how to verify these conditions in practice and
places the results in the context of the related literature.
Section~\ref{sec:examples} applies the theory to the unadjusted
Langevin algorithm, sampling chains with increments of infinite
variance, and Euler schemes for L\'evy-driven SDEs with linear
drift, multiplicative noise and logarithmic innovation tails.
Section~\ref{sec:main_proofs} proves the main general results,
while Section~\ref{sec:return_times_convergence} develops the
technical return-time and convergence estimates used in those
proofs. Finally, Section~\ref{sec:examples_proofs} contains the
proofs of the results in Section~\ref{sec:examples}.

\section{\texorpdfstring{$\mathbf L$}{L}-drift conditions, stability of Markov chains and the ergodic CLT}
\label{sec:main_results}
\subsection{General assumptions and notation}\label{subsec:General_Assumptions}
Let $X=(X_n)_{n\in\N}$ be a discrete-time Markov chain on a locally compact topological space with a Borel $\sigma$-algebra $(\cX,\cB(\cX))$ and transition kernel $P$. We additionally assume that $(\cX,\cB(\cX))$ is a standard Borel space. Unless otherwise stated, the following \textit{standard assumptions} hold throughout the paper. The chain $X$ is aperiodic,  irreducible and positive Harris recurrent~\cite[Sec.~6.1.2]{meynadntweedie}. Let $P^n$, for $n\in\N\coloneqq\{0,1,\ldots\}$, denote the $n$-step transition kernel for the Markov chain $X$, that is, 
\begin{equation*}
    P^n(x,A)=\P_x(X_n\in A),\qquad x\in\cX,~~ A\in \cB(\cX),
\end{equation*}
where $\P_x$ is the conditional distribution of the chain $X$, given $X_0=x$. The corresponding expectation operator will be denoted $\E_x$. We assume that $X$ is a \textit{Feller} chain~\cite[Ch.~6]{meynadntweedie}, i.e., for every $n\in\N$ and open set $O\subset \cX$, the function $x\mapsto P^n(x,O)$ is lower semicontinuous. For any $x\in\cX$ and measurable function $f:\cX\to\R$, we denote $Pf(x)\coloneqq \int_\cX f(y)P(x,\ud y)$. For a signed measure $\mu$ on $\cB(\cX)$, we write $\mu(f)\coloneqq \int_\cX f(y)\mu(\ud y)$. We assume that $X$ is an ergodic chain with an invariant distribution $\pi$, i.e., for any $x\in\cX$ we have $\| P^n(x,\cdot)-\pi\|_{\TV}\to 0$ as $n\to\infty$, where $\|\mu\|_{\TV} \coloneqq \sup\{\vert\int_\cX g(x)\mu(\ud x)\vert:g:\cX\to \R \text{ measurable, } |g|\leq 1\}$.

Following~\cite{brešar2024subexponential}, we introduce the \textbf{L}-drift conditions in discrete time.

\begin{assumption*}[\textbf{L($V$,$\varphi$,$\Psi$)}]
\namedlabel{sub_drift_conditions}
Let $V:\cX\to[1,\infty)$ be a continuous function, such that, for all $x\in \cX$, $\limsup_{n\to\infty}V(X_n)=\infty$ holds $\P_x$-a.s. There exists $\ell_0\in[1,\infty)$ such that~\ref{sub_drift_conditions(i)} and~\ref{sub_drift_conditions(ii)} hold.\\
\namedlabel{sub_drift_conditions(i)}{\textbf{(i)}}
Let $\varphi:(0,1] \to\RP$ 
be a continuous function,
such that  $r\mapsto r\varphi(1/r)$ is decreasing\footnote{A non-increasing function may have intervals of constancy, while a decreasing function does not (similarly for non-decreasing and increasing). In particular, since $r\mapsto r\varphi(1/r)$ is decreasing, $\varphi$ must be increasing.} on $[1,\infty)$, the limit
$\lim_{r\to\infty}r\varphi(1/r)=0$ holds
and for some $b\in\RP\coloneqq[0,\infty)$ we have
\begin{equation*}
\label{eq:submartingale_drift}
P(1/V)-1/ V\leq \varphi(1/V) +b \mathbbm{1}\{V\leq\ell_0\}.
\end{equation*}
\namedlabel{sub_drift_conditions(ii)}{\textbf{(ii)}}
Let $\Psi:[1,\infty)\to[1,\infty)$ be a differentiable, increasing, submultiplicative\footnote{A function $\Psi:[1,\infty) \to[1,\infty)$ is
 \textit{submultiplicative} if it satisfies
$\Psi(r_1+r_2)\leq C\Psi(r_1)\Psi(r_2)$
for some constant $C\in(0,\infty)$ and all $r_1,r_2\in[1,\infty)$~\cite[Def.~25.2]{MR3185174}.}\label{footnote:submultiplicative} function satisfying the following: for any $\ell\in(\ell_0,\infty)$ and $m\in(0,\infty)$, there exists a 
constant $C_{\ell,m}\in(0,1)$ such that 
\begin{equation}
\label{eq:assumption_heavy_tail_bound}
\P_x(T^{(r)}<S_{(\ell)}) \geq 
%\1{V(x)\geq r} + 
C_{\ell,m}/\Psi(r) \quad \text{for all $r\in(\ell+m,\infty)$ and $x\in\{\ell+m\leq V\}$,}
\end{equation}
where $T^{(r)}\coloneqq\inf\{n\in\N :V(X_n)>r\}$ and $S_{(\ell)} \coloneqq \inf\{n\in\N: V(X_n)<\ell\}$.
\end{assumption*}

\begin{rem}
\nameref{sub_drift_conditions} describes two complementary aspects of
excursions away from $\{V<\ell\}$.

\smallskip\noindent
\emph{Excursion durations $\varphi$.}
Since $1/V$ increases as the chain returns from the tail, its drift
measures the tendency to move inwards. Condition~\ref{sub_drift_conditions(i)}
says that subtracting the accumulated drift allowance
$\sum_{k=0}^{n-1}\varphi(1/V(X_k))$ from $1/V(X_n)$ produces a
supermartingale when stopped at $S_{(\ell)}$. Thus, a larger $\varphi$
allows faster returns, while a smaller $\varphi$ gives stronger lower
bounds on return-time tails: once the chain reaches a high level,
it has a substantial chance of remaining away for a long time.

\smallskip\noindent
\emph{Excursion heights $\Psi$.}
Condition~\ref{sub_drift_conditions(ii)} ensures that an excursion
starting with $V(x)\geq\ell+m$ reaches $\{V>r\}$ before returning below $\ell $ with probability at least $C_{\ell,m}/\Psi(r)$. Such an exit bound typically
follows from a submartingale condition on $\Psi\circ V(X)$, together
with control of the overshoot above $r$. Thus, $\Psi$ controls how
rare high excursions can be, while $\varphi$ controls how long they
can last; together they yield lower bounds on invariant tails and
rates of convergence to equilibrium. Section~\ref{subsec:applying}
explains how to choose and verify both functions in practice,
including how the exit bound follows from a submartingale argument
or directly from the probability of a single large jump.
\end{rem}

A Borel measurable function  $g:\cX\to \R$ is in $L^2(\pi)$ if \begin{equation}
\label{eq:CLT_necessary}
 \quad \int_{\cX} g(x)^2 \pi(\ud x)<\infty.
\end{equation}
Denote the ergodic average by
\begin{equation}
    \label{eq:ergodic_average}
    S_{n}(g) \coloneqq \frac{1}{n}\sum_{k=0}^{n-1} g(X_k), \quad n\geq 1.
\end{equation}
The central limit theorem  for the ergodic average $S_n(g)$ holds if there exists $\sigma_g^2\in(0,\infty)$ such that, for every starting state $X_0=x\in\cX$, we have
\begin{equation}
\label{eq:CLT}
\sqrt{n}(S_n(g)-\pi(g))\rightarrow N(0,\sigma_g^2)\quad \text{in distribution as $n\to\infty$,}
\end{equation}
where $N(0,\sigma_g^2)$ denotes centred Gaussian law with variance $\sigma_g^2\in(0,\infty)$. 
A real function $f$ (defined on a neighbourhood of infinity in $\RP$) is integrable near infinity if:
$$f\in L^{1}_{\mathrm{loc}}(\infty)\stackrel{\mathrm{def}}{\iff}
\text{there exists } l\in\RP,\text{ such that }\int_{l}^\infty |f(u)|\ud u<\infty.$$

\subsection{Necessary conditions for the validity of the CLT of an ergodic average, invariant tails and lower bounds on convergence rates}
\label{subsec:Results}
The necessary condition for the CLT in Theorem~\ref{thm:CLT} is a consequence of   our \textbf{L}-drift conditions~\nameref{sub_drift_conditions}. 
Let $h:\mathbb R_+\to\mathbb R_+$ be measurable,
positive, continuous and non-decreasing on a
neighbourhood of infinity
%Let $h:\RP\to\RP$ be a measurable function, non-decreasing outside of some compact set
and let $\varphi$ be the function in Assumption~\nameref{sub_drift_conditions}. 
For any constant $q\in(0,1)$, define 
\begin{equation}
\label{eq:G_h}
G_h(r)\coloneqq q(1-q)\frac{h(r)}{r\varphi(1/r)},\quad r\in[1,\infty),
\end{equation}
 which is increasing on a complement of a compact set in $\RP$.

%Theorem~\ref{thm:CLT} below provides a condition for $g$, satisfying~\eqref{eq:CLT_necessary}, under which the CLT fails. 

\begin{thm}
\label{thm:CLT}
Let Assumption~\nameref{sub_drift_conditions} hold. Assume additionally
that $V$ has compact sublevel sets and that every compact set
$K\in\cB(\cX)$ satisfying $\pi(K)>0$ is small in the sense
of~\eqref{eq:small}. Let $h:\RP\to\RP$ be a measurable function which
is continuous, positive and non-decreasing on a neighbourhood of
infinity, and let $G_h^{-1}$ be the inverse of $G_h$ in~\eqref{eq:G_h}
on a neighbourhood of infinity.

Suppose that~\eqref{eq:CLT_necessary} holds for a locally bounded
function $g:\cX\to\R$ and that for some $\ell_0>0$ we have
\begin{equation}
\label{eq:CLT_centered_tail}
g(x)-\pi(g)\geq h\circ V(x)
\qquad\text{for all $x$ in $\{V\geq\ell_0\}$.}
\end{equation}
If
\begin{equation}
\label{eq:CLT_integral_test}
t\mapsto
1/\Psi\left(G_h^{-1}(\sqrt{t})/(1-q)^2\right)
\notin L^{1}_{\mathrm{loc}}(\infty)
\quad\text{for some $q\in(0,1)$,}
\end{equation}
then the CLT \textbf{fails} for the ergodic average $S_n(g)$
in~\eqref{eq:ergodic_average}.
\end{thm}

The proof of Theorem~\ref{thm:CLT} in Section~\ref{sec:main_proofs} requires verifying the negation of the classical  characterisation conditions for the ergodic CLT~\cite{Chen99}, given in terms of the excursions away from a small set for the Nummelin split chain of $X$. Via Theorem~\ref{thm:modulated_moments} below, our $\mathbf{L}$-drift conditions enable us to verify the negation of the characterisation of the CLT in terms of  the transition kernel $P$ of the chain $X$. In the Markov chain setting, sufficient condition for the ergodic CLT to hold are given in~\cite{MR2446322,MarkovChains}. In Section~\ref{sec:examples} below, we demonstrate that our necessary conditions in Theorem~\ref{thm:CLT} and these well-known sufficient conditions  are \textit{sharp} for the ULA chain targeting a density with polynomial tails.

Theorem~\ref{thm:modulated_moments} gives lower bounds on the tails of return times and modulated moments of the chain $X$. For any $k\in\N$ and any set $B\in \cB(\cX)$, define
 the first return time of $X$ to $B$
after $k$ steps as
\begin{equation}
\label{eq:tau_k}
\tau_B(k) \coloneqq \inf\{n\geq k: X_n\in B\}\qquad\text{(with convention $\inf\emptyset \coloneqq\infty$)}.
\end{equation}

\begin{thm}[Tails of return times \& modulated moments]\label{thm:modulated_moments}
Let Assumption~\nameref{sub_drift_conditions} hold with $\ell_0\in[1,\infty)$ and fix $q\in(0,1)$. 
For a measurable function 
$h:[1,\infty)\to[1,\infty)$, continuous and 
non-decreasing on $[\ell_0',\infty)\subseteq[\ell_0,\infty)$,
let  $G_h^{-1}:[G_h(\ell_0'),\infty)\to\RP$ be the inverse of $G_h$ given in~\eqref{eq:G_h}. 
Then  for every $\ell\in[\ell_0',\infty)$ and $B\in \cB(\cX)$, satisfying $B \subset \{V\leq \ell\}$, the following statement holds:
%\begin{myenumi}[label=(\alph*)]
%\item\label{thm:modulated_moments_a}
for every $m\in(0,\infty)$, the constant $C_{\ell+m/2,m/2}$ in~\nameref{sub_drift_conditions}\ref{sub_drift_conditions(ii)} and every $x\in\cX$ we have
$$
\P_x\left(\sum_{k=0}^{\tau_B(1)-1} h\circ V(X_k)\geq r\right)\geq \frac{q^2 C_{\ell+m/2,m/2} P(x,\{V\geq \ell + m\})}{\Psi(G_h^{-1}(r)/(1-q)^2)} \quad \text{for all $r\in(G_h(\ell+m),\infty)$.}
$$
% \item\label{thm:modulated_moments_b}
% Fix $h\equiv1$. For every $m\in(0,\infty)$, the constant $C_{\ell,m}$ in~\nameref{sub_drift_conditions}\ref{sub_drift_conditions(ii)} and every $x\in\cX$ we have
% $$
% \P_x(\tau_B(1)\geq t)\geq \frac{q C_{\ell,m}P(x,\{V\geq \ell + m\})}{\Psi(G_1^{-1}(t)/(1-q)^2)}\quad \text{for all $t\in(G_1(t+m) ,\infty)$.}
% $$
%\end{myenumi}
\end{thm}

Observe that Theorem~\ref{thm:modulated_moments} implies a lower bound on the tail of the return time $\tau_B(1)$: for
$h\equiv1$, every $m\in(0,\infty)$, the constant $C_{\ell+m/2,m/2}$ in~\nameref{sub_drift_conditions}\ref{sub_drift_conditions(ii)} and every $x\in\cX$ we obtain
$$
\P_x(\tau_B(1)\geq t)\geq \frac{q^2 C_{\ell+m/2,m/2}P(x,\{V\geq \ell + m\})}{\Psi(G_1^{-1}(t)/(1-q)^2)}\quad \text{for all $t\in(G_1(\ell+m) ,\infty)$.}
$$

The comparison of the target distribution with the invariant distribution of a biased sampling algorithm, such as ULA, requires lower bounds on the tail of the invariant distribution.
Theorem~\ref{thm:invariant} gives such lower bounds on the tails of the invariant measure $\pi$ of a general ergodic Markov chain $X$ under Assumption~\nameref{sub_drift_conditions}. The bounds are in terms  of the functions $V$, $\varphi$, $\Psi$, which can in applications be obtained directly from the one-step transition kernel $P$ of $X$.
For any $q\in(0,1)$, define the function
\begin{equation}
\label{eq:def_L_eps_q}
    L_q(r)\coloneqq r\varphi(1/r)\Psi(r/(1-q)^2)D(r),\quad\text{$r\in[1,\infty)$,}
\end{equation}
where $D:[1,\infty)\to[1,\infty)$ is eventually increasing and tends to infinity arbitrarily slowly (for example, $D(r)=\log\circ\ldots\circ\log(r)$
for all large $r$, where the number of compositions of logarithms is fixed but arbitrary).

\begin{thm}[Tails of the invariant measure]
\label{thm:invariant}
Let Assumption~\nameref{sub_drift_conditions} hold. Then for any $q\in(0,1)$,  there exists a constant $c_q\in(0,1)$ such that
\begin{equation}
\label{eq:main_result_invariant}
c_q/L_q(r)\leq \pi(\{V \geq r \})\qquad\text{for all $r\in[1,\infty)$.}
\end{equation}
 \end{thm}

The proof of Theorem~\ref{thm:invariant} is based on the well-know characterisation of the tails of the invariant measure in terms of excursions away from a small set, see e.g.~\cite[Ch.~11]{MarkovChains}, and
the lower bounds in Theorem~\ref{thm:modulated_moments} above. The details are in Section~\ref{subsec:lower_bound_invariant} below. General upper bounds on the tails of $\pi$ are well known and given in e.g.~\cite[Thm~16.1.12]{MarkovChains}. We show in Section~\ref{sec:examples} below that the rate of decay in the lower bounds of Theorem~\ref{thm:invariant} agrees with the rate in the upper bounds for the ULA chain  targeting densities with polynomial tails.

For any measurable function $f:\cX\to[1,\infty)$,
the \textit{$f$-variation} of a signed measure $\mu$ on $(\cX,\cB(\cX))$ is given by
 $\|\mu\|_{f} \coloneqq \sup\{\vert\int_\cX g\ud\mu\vert:g:\cX\to \R \text{ bounded and measurable, } |g|\leq f\}=\int_\cX f\ud |\mu|$.
 In the special case $f\equiv 1$ we obtain the \textit{total variation} $\|\mu\|_{\TV} =\|\mu\|_f$ of the measure $\mu$.

Denote by $\cP(\cX)$ the class of all  probability measures on $\cB(\cX)$. If, in addition, $\cX$ is a metric space equipped with the distance $d:\cX\times\cX\to\RP$, for $p\in[1,\infty)$, let $\cP_p(\cX)$  denote the space of probability measures $\mu\in\cP(\cX)$ satisfying $\int_{\cX}|d(x_0,x)|^p\mu(\ud x)<\infty$ for some $x_0\in\cX$. The $L^p$-\textit{Wasserstein distance} $\cW_p$ on $\cP_p(\cX)$  is defined by
\begin{equation}
\label{eq:def_Wassesstein_distance}
\cW_p(\mu_1,\mu_2) \coloneqq \inf_{\Pi \in \cC(\mu_1,\mu_2)}\left(\int_{\cX\times \cX}d(x,y)^p\Pi(\ud x,\ud y)\right)^{1/p},
\end{equation}
where $\cC(\mu_1,\mu_2)$ is the family of couplings of probability measures $\mu_1$ and $\mu_2$ in $\cP_p(\cX)$.

Under Assumption~\nameref{sub_drift_conditions},
Theorem~\ref{thm:f_rate} provides lower bounds on the rates of convergence of the chain $X$ to its invariant measure $\pi$ in $f$-variation and $L^p$-Wasserstein distance.

\begin{thm}[Lower bounds on convergence rates]
\label{thm:f_rate}
Let Assumption~\nameref{sub_drift_conditions} hold.
Pick $q\in(0,1)$ and let the constant $c_q\in(0,1)$ be such that inequality~\eqref{eq:main_result_invariant} holds with the function $L_q$ in~\eqref{eq:def_L_eps_q}.
Let $f_\star :[1,\infty)\to(0,\infty)$ be non-decreasing.
Assume a continuous function $a_g:[1,\infty)\to\RP$ satisfies
\begin{equation*}
%\label{eq:a_convergence_lower}
a_g(r) \leq  f_\star(g^{-1}(r))c_q/L_q(g^{-1}(r))\quad\text{for all $r\in[1,\infty)$,}
\end{equation*} 
where $g:[1,\infty)\to [1,\infty)$ 
is an increasing continuous function specified below and
$g^{-1}$ is its inverse. Suppose also that  $A_g(r)\coloneqq ra_g(r)$  is increasing, $\lim_{r\to\infty} A_g(r) = \infty$ and denote its inverse by $A_g^{-1}$.
Let a continuous function $h:[1,\infty)\to[1,\infty)$ be such that  the function $h/f_\star$ is increasing  and $\lim_{r\to\infty}h(r)/f_\star(r)=\infty$.
Let  $v:\cX\times\N\to[1,\infty)$ be increasing in the second argument and satisfy 
\begin{equation*}
%\label{eq:h_rate_codition}
P^n(h\circ V)(x)\leq v(x,n)\quad\text{for all $x\in\cX$ and $n\in\N$.}
\end{equation*}
%\item\label{assumption:f_convergence_b}
\begin{myenumi}[label=(\alph*)]
\item\label{assumption:f_convergence}
Set $g\coloneqq h/f_\star$ and assume $g(1) = 1$ and $f_\star\ge 1$.  
 Then the  lower bound on $f$-variation holds for $f=f_\star\circ V$:
\begin{equation}
\label{eq:main_result_rate_of_convergence}
\frac{1}{2}(a_g\circ A_g^{-1}\circ (2 v))(x,n)\leq \| P^n(x,\cdot)-\pi\|_{f} \quad \text{for all $x\in\cX$ and $n\in\N\setminus\{0\}$.}
\end{equation}
\item\label{assumption:Wasserstein_convergence}
Assume in addition that 
$\cX$ 
is a metric space, $V:\cX\to[1,\infty)$ is Lipschitz with Lipschitz constant $L_V\in(0,\infty)$ and, for some $p\in[1,\infty)$,  we have $\pi,P^n(x,\cdot)\in\cP_p(\cX)$ for all $n\in\N$ and $x\in \cX$. Set $f_\star(r) \coloneqq (r/2)^p$, $r>0$, and 
$ g(r)\coloneqq h(r/2)/f_\star(r/2)$, $r\in[2,\infty)$, and 
$ g(r)\coloneqq ( g(2)-1)(r-1)+1$,  
$r\in[1,2)$.
Then we obtain the following lower bound on the $L^p$-Wasserstein distance:
\begin{equation}
\label{eq:main_result_Wasserstein}
\frac{1}{2L_V}(a_{ g}\circ A_{ g}^{-1}\circ (2^{2p} v))(x,n)^{1/p}\leq \cW_p(P^n(x,\cdot),\pi) \quad \text{for all $x\in\cX$ and $n\in\N\setminus\{0\}$.}
\end{equation}
\end{myenumi}
\end{thm}

\begin{rem}
\noindent (i) Theorem~\ref{thm:f_rate}\ref{assumption:Wasserstein_convergence}  requires a Lipschitz Lyapunov function $V$. We stress that this is \textit{not} a restrictive assumption as it clearly works for all the examples in Section~\ref{sec:examples} below. More generally, the functions $V$ and $V^s$ (for some $s>0$), used as Lyapunov functions in the \textbf{L}-drift condition in~\nameref{sub_drift_conditions},
typically yield the same lower bounds. This is because the change in the growth rate of the Lyapunov function is naturally absorbed by appropriately modifying $\varphi$ and $\Psi$ (the key quantities in~\nameref{sub_drift_conditions} are the functions $\varphi\circ (1/V)$ and $\Psi\circ V$).
 
\noindent (ii) The integrability condition $\pi,P^n(x,\cdot)\in \cP_p$ for all $n\in\N$ and $x\in \cX$, in part~\ref{assumption:Wasserstein_convergence} of Theorem~\ref{thm:f_rate}, holds if the chain $X$ converges to $\pi$ in the  Wasserstein distance $\cW_p$. Since the convergence in the $L^p$-Wasserstein metric $\cW_p$ holds under appropriate upper bound drift conditions for the chain $X$ (see e.g.~\cite{MR4429313,Durmus16}), this assumption is satisfied in most (if not all) cases of interest.

\noindent (iii) The formulas for the function $g$ differ slightly in parts~\ref{assumption:f_convergence} and~\ref{assumption:Wasserstein_convergence} of Theorem~\ref{thm:f_rate}. The reason for the discrepancy is in the fact that both bounds rely on the comparison between the tails of the invariant measure $\pi$ and the marginal $P^n(x,\cdot)$ but in different metrics. The $f$-variation distance works with discontinuous functions, while the $\cW_p$-distance  uses Lipschitz functions, leading to different definitions of $g$.

\noindent (iv) Upper bounds on convergence rates in $f$-variation and the Wasserstein distance $\cW_p$ have been extensively studied, see, e.g.,~\cite{MarkovChains,Durmus16} and the references therein. Theorem~\ref{thm:ULA} shows that our results are sharp for the ULA chain in both metrics considered in Theorem~\ref{thm:f_rate}.
%\noindent (iv) Note that the lower bound on the $L^p$-Wasserstein metric $\cW_p$ in~\eqref{eq:main_result_Wasserstein} for $p>1$ is asymptotically larger than the corresponding lower bound on the $f$-variation in~\eqref{eq:main_result_rate_of_convergence}. This difference is illustrated in the context of sampling algorithms in Section~\ref{sec:examples} below.
\end{rem}

\begin{rem}
While the applications of the results of Subsection~\ref{subsec:Results} in this paper focus on Langevin-type chains (see Section~\ref{sec:examples} below), the theory developed here applies much more broadly. In particular, it may yield sharp convergence rates and conditions for the validity of the ergodic CLT  for most standard MCMC algorithms, including random-walk Metropolis, Metropolis adjusted Langevin algorithm, and Hamiltonian Monte Carlo. These prospective applications, beyond the scope of the present paper, allow a novel theoretical comparison of sampling algorithms. 
\end{rem}

\subsection{How to apply the results of Section~\ref{subsec:Results} in practice?}
\label{subsec:applying}
The assumptions of the theorems in Section~\ref{subsec:Results}, essentially consisting of the $\mathbf{L}$-drift conditions in~\nameref{sub_drift_conditions},
can be verified as easily as the classical Lyapunov drift conditions for upper bounds.   If the process $V(X)$ satisfies the upper bound drift conditions (see e.g.~\cite{MR2071426,MarkovChains}, also stated in~\eqref{eq:upper_drift} below), then the assumptions on the ergodicity of the chain $X$ in~\nameref{sub_drift_conditions} are met. Moreover, the non-confinement can be directly verified via the following elementary lemma, proved in Section~\ref{subsec:verifying} below.

\begin{lem}
\label{lem:non_confinement}
Beyond the standard assumptions on $X$ in Section~\ref{subsec:General_Assumptions}, assume that $V:\cX\to[1,\infty)$ is continuous with compact sublevel sets. If for every $u\geq 1$ and $x\in\{V\leq u\}$ there exists $n_x\in\N$ such that $ P^{n_x}(x,\{V> u\})>0$, 
the non-confinement property holds:
$\P_x(\limsup_{n\to\infty}V(X_n)=\infty)=1$
for every $x\in\cX$.
\end{lem}

The next task is to identify the scalar functions $\varphi$ and $\Psi$ and verify the conditions in~\nameref{sub_drift_conditions}.
Identifying $\varphi$ and verifying the inequality 
$P(1/V)- 1/V\leq \varphi(1/V)$ on a complement of a compact set
can be easily done directly from the one-step transition kernel $P$, since  the function $1/V$ is bounded and therefore integrable (see e.g., Section~\ref{subsec:ULA_proofs} below for the ULA chain).

\subsubsection{Identifying \texorpdfstring{$\Psi$}{Psi}} 
\label{subsubsec:Identifyin_Psi}
The search for the function $\Psi$ is in two steps. We first look for a good candidate and then, using it, verify the condition in~\nameref{sub_drift_conditions}\ref{sub_drift_conditions(ii)}.
Since the functions $\Psi$, $\varphi$ and the tail $r\mapsto \pi(\{V\geq r\})$ satisfy the inequality in~\eqref{eq:main_result_invariant}, a good candidate for the function $\Psi$ is such that the inequality in~\eqref{eq:main_result_invariant}  provides a sharp bound on the tail of $\pi$. In the case where $\pi$ is not given explicitly, a good candidate for its tail decay can be obtained from the upper-bound Lyapunov drift condition (see e.g.~\cite{MR2071426}, also stated in~\eqref{eq:upper_drift}). Once we find a suitable candidate for $\Psi$, we distinguish between the two cases for verifying  the condition in~\eqref{eq:assumption_heavy_tail_bound} of~\nameref{sub_drift_conditions}\ref{sub_drift_conditions(ii)}.

\noindent\textbf{Non-integrable case: single large jump phenomenon.} If $P(\Psi \circ V)(x)=\infty$ on some  complement of a compact set in $\cX$, the condition in~\eqref{eq:assumption_heavy_tail_bound} of~\nameref{sub_drift_conditions}\ref{sub_drift_conditions(ii)} can be verified directly from the one-step transition kernel $P$ by choosing $\Psi$ so that the following holds:
$$
\P_x(T^{(r)}<S_{(\ell)})\geq P(x,\{V> r\})\geq 1/\Psi(r)\qquad \text{for all $x\in\{V\geq\ell\}$.}
$$
Heuristically, in this case, the probability of a single large jump offers a
good approximation for the probability of exiting the ``interval''
$\{\ell\leq V\leq r\}$ through its upper boundary.  This mechanism appears in
the unadjusted schemes with infinite-variance increments studied in
Section~\ref{subsec:discretisation_levy_algorithm}.

\noindent \textbf{Integrable case: many small jumps are needed to visit the tails.} If $P(\Psi\circ V)<\infty$ on $\cX$ and 
\begin{equation}
\label{eq:Psi_integrable_case}
P(\Psi \circ V)(x)\geq \Psi \circ V(x)\quad\text{for all $x\in\cX$ on the complement of some compact set,}
\end{equation}
the lemma below applies.  
In this regime, the process exits through the upper boundary of the ``interval''
$\{\ell\leq V\leq r\}$, not by a single large jump, but by the cumulative effect of many smaller jumps driven by the positive ``drift’’ of the process $\Psi\circ  V(X)$ (note that  by~\eqref{eq:Psi_integrable_case} the process $\Psi\circ  V(X)$ is a submartingale on the complement of some compact set).  The ULA chain studied in Section~\ref{subsec:ULA} is an example where $P(\Psi\circ V)<\infty$ on $\cX$.
The proof of Lemma~\ref{lem:overshoot} can be found in Section~\ref{subsec:verifying} below.

\begin{lem}
\label{lem:overshoot}
Let $X$ be a Markov chain on $\cX$ and $V:\cX\to[1,\infty)$ such that $\limsup_{n\to\infty}V(X_n) = \infty$ $\P_x$-a.s.~for all $x\in \cX$. Fix  a non-decreasing function $\Psi:[1,\infty)\to[1,\infty)$.
%\begin{myenumi}
%\item[(a)]
Assume that for some $\ell_0>0$ the function $\Psi$ is increasing on $[\ell_0,\infty)$ and the following inequality holds
\begin{equation}
\label{eq:sub_mart_ass}
P(\Psi \circ V)(x)\geq\Psi \circ V(x)\quad\text{for all $x\in \{V\geq\ell_0\}$}.
\end{equation}
 Assume further that there exist  $C\in(0,\infty)$ and
 $r_0\in(\ell_0,\infty)$
 such that for all $r\in(r_0,\infty)$ we have
\begin{equation}
\label{eq:incremental_overshoot}
P(\1{V>r}\Psi\circ V)(x)\leq C\Psi(r)P(x,\{V>r\})\quad\text{for all $x\in\{V\leq r\}$.}
\end{equation}
%\item[(b)] 
Then, 
for any $\ell\in(\ell_0,\infty)$ and $m\in(0,\infty)$, there exists a 
constant $C_{\ell,m}\in(0,1)$ such that 
% for every $m>0$ there exists $C_m\in(0,1)$, such that for all $r\in(\max\{r_0,\ell_0\},\infty)$ and $\ell\in(\ell_0,r)$ we have 
$$
\P_x(T^{(r)}<S_{(\ell)})\geq C_{\ell,m}/\Psi(r) \quad \text{for all $r\in(\max\{r_0,\ell+m\},\infty)$ and $x\in \{V\geq \ell+m\}$.}
$$
%\end{myenumi}
\end{lem}

\begin{rem}
\noindent (i)    Observe that if $P(\Psi \circ V)<\Psi \circ V$ outside of some compact set, composing $\Psi$ with an increasing convex function (e.g., $\Psi^s$ for some $s>1$)  heuristically yields either the non-integrability  (for all $s>1$) or~\eqref{eq:Psi_integrable_case}. Indeed, if $s>1$ can be  chosen so that $(\Psi\circ V(X))^s$ is a martingale on the complement of a compact, then 
$
P((\Psi \circ V)^s)=(\Psi \circ V)^s
$
implying~\eqref{eq:Psi_integrable_case}.

If the function  $P(\Psi \circ V)-\Psi \circ V$ 
takes both positive and negative values on the complements of compact sets, this indicates that the Lyapunov function $V$ does not optimally transform the state space $\cX$, leading to suboptimal lower bounds.

\noindent (ii) The proof of Lemma~\ref{lem:overshoot} shows that assumption in~\eqref{eq:incremental_overshoot} alone implies the following:
for any $r\in(r_0,\infty)$ and $\ell\in(0,r_0)$, there exists $C>0$ such that the following inequality holds  
\begin{equation*}
\E_x[\Psi\circ V(X_{T^{(r)}})\mathbbm{1}\{S_{(\ell)}> T^{(r)}\}]\leq C\Psi(r)\P_x(S_{(\ell)}> T^{(r)})\quad\text{for all $x\in\{V\leq r\}$.}
\end{equation*}
\end{rem}

\subsubsection{Choosing \texorpdfstring{$h$}{h} in Theorem~\ref{thm:f_rate}}
A rule of thumb  for choosing the function $h$ in Theorem~\ref{thm:f_rate} is to look for the ``fastest growing''
function for which we can still obtain a sharp upper bound in $n$ on $P^n(h\circ V)$. A further useful heuristic, which yields sharp bounds for the polynomially ergodic applications in this paper, is as follows: let $W$ be a Lyapunov function that provides upper bounds via the upper-bound drift condition~\cite{MR2071426} (see also~\eqref{eq:upper_drift} below), and let $V$ be the Lyapunov function satisfying the \textbf{L}-drift condition~\nameref{sub_drift_conditions}. Choosing $h$ such that $h\circ V\propto W$ for the ULA chain in Section~\ref{sec:examples}
yields lower bounds that match the upper bounds obtained via the Lyapounov function $W$.\footnote{Here and throughout we denote
$f\propto g$ for $f,g:U\to(0,\infty)$, defined on a set $U$, if the function $f/g$ is constant on $U$.} In particular, observe that~\eqref{eq:upper_drift} implies  at most linear growth in $n$ of  the sequence $P^n(h\circ V)$.

We expect the above heuristic to hold for all polynomially ergodic chains. However, in cases of faster convergence (e.g., exponential or stretched exponential), it is likely necessary to choose a function 
$h$, such that $h\circ V$ grows faster than $W$. Examples of such Markov processes can be found in~\cite[Sections~3.1.1 and~3.1.2]{brešar2024subexponential}. In those cases, matching lower bounds on the rates of convergence are obtained by selecting $h$ such that $h\circ V/ W\to\infty$ on the complements of  compact sets, leading to a superlinear growth in $n$ of the sequence
$P^n(h\circ V)$. Obtaining sharp bounds on the growth of $P^n(h\circ V)$  would then require a discrete-time version of~\cite[Lem.~2.8]{brešar2024subexponential}. The detailed treatment of lower bounds for convergence rates faster than polynomial is beyond the scope of this paper and is left for future work.

\subsection{Related literature and discussion}
\label{sec:discussion}

\paragraph{\textbf{Central limit theorems for ergodic averages}}  Sufficient conditions for the validity of the CLT for the sample mean of Markov processes are central results in probability theory~\cite{MR834478,MR1782272}. 
The CLT for estimators in sampling algorithms is of fundamental importance in applications~\cite{meynadntweedie,MarkovChains}.
Many authors have investigated conditions under which the CLT holds in the context of MCMC (see~\cite{MR1890063,meynadntweedie,MR2068475,MarkovChains} and references therein). 
However, while there exist numerous representations of the asymptotic variance for ergodic estimators (see, e.g.,~\cite[Sec.~17.4]{meynadntweedie} and~\cite[Sec.~21]{MarkovChains}) in terms of the solution of the Poisson equation for the Markov chain or its related correlation function, there are still very few general results giving practical \textit{necessary} conditions, expressed through the one-step transition kernel, for the CLT to hold.  An important advance in this direction was made in~\cite{MR3843830}, which established sharp necessary and sufficient conditions for the independence sampler using analytical methods available for that chain. The main contribution of the present paper is a robust method that gives necessary conditions for a broad class of Markov chains; see Theorem~\ref{thm:CLT}.  Section~\ref{sec:examples} demonstrates that the condition yields a sharp CLT threshold for ULA, whose invariant law is not explicitly known.

\smallskip

\paragraph{\textbf{Upper bounds on the convergence rate of ergodic Markov chains}} Two standard ways of establishing upper bounds on the convergence rate to stationarity for ergodic Markov chains are:

\noindent\textit{(I) Drift conditions for upper bounds.} 
The drift conditions are based on the idea that upper bounds can be derived by successful coupling on compact sets. This approach has been widely used since the 1990s, see e.g.~\cite{meynadntweedie,MR1285459,MR1340509}. The modern formulation of drift conditions, which we now briefly recall, was first introduced in~\cite{MR1796485} (for the independence sampler) and later refined in~\cite{MR2071426}.

Let $W:\cX\to[1,\infty)$ and $B\subset \cX$ be a petite set  for the chain $X$. Let $\phi:[1,\infty)\to\RP$ be a concave, increasing function. Assume that for $b\in(0,\infty)$ the following upper-bound drift condition: 
\begin{equation}
\label{eq:upper_drift}
P W(x)\leq W(x) -\phi\circ W(x) + b \mathbbm{1}_{\{x\in B\}}\qquad\text{for all $x\in\cX$.}
\end{equation}
Then  the following upper bounds hold: $\pi(\phi \circ W)<\infty$ and $\|P^n(x,\cdot)-\pi\|_{\TV}\leq CW(x)/r_\phi(n)$, where $r_\phi\coloneqq \phi\circ H_\phi^{-1}$ with $H_\phi(u)=\int_1^u \ud s/\phi(s)$, $u\in[1,\infty)$  and some constant $C\in(0,\infty)$ that does not depend on $x\in\cX$ or $n\in\N$ (see  monograph~\cite{MarkovChains} for the state of the art results). Moreover, condition~\eqref{eq:upper_drift} also implies upper bounds on convergence rates in $f$-variation and the Wasserstein distance~\cite[Sec.~20]{MarkovChains}. 

Since this approach relies on transforming the state space and upper bounding the drift of the transformed process, key information may be lost, potentially resulting in poor upper bounds on convergence rates. Analogously, condition~\eqref{eq:upper_drift} provides sufficient conditions for a CLT of an ergodic average to hold~\cite{MR2446322}. However, these conditions may also be far from necessary for a poor choice of the Lyapunov function $W$. 
%This natural motivates the development of \textit{new drift conditions} for \textit{lower} bounds and a \textit{sharp transition} for the validity of the CLT.

\noindent\textit{(II) Functional inequalities.} An alternative approach to subexponential ergodicity of Markov chains involves the use of functional inequalities, particularly of the weak Poincar\'e inequalities. This method was initially employed in continuous time~\cite{MR1856277} and, more recently, in discrete time~\cite{MR4524509,MR5033571}. A key advantage of these techniques is that they yield non-asymptotic bounds, providing   dependence on the dimension in the constants and tuning parameters to maximize performance. To the best of our knowledge, these methods are generally not well-suited for providing verifiable necessary conditions for the CLT of ergodic averages. In this sense, the results presented in this paper complement the theory developed through functional inequalities.

\smallskip

\paragraph{\textbf{Lower bounds on the convergence rate of ergodic Markov chains}} The first general conditions for ergodicity in $f$-variation 
were given in the seminal paper~\cite{MR1285459}.
Further results on the failure of uniform and geometric ergodicity for MCMC samplers were established in~\cite{Mengersen96}. These results were later extended in~\cite{MR1996270} to derive necessary conditions for polynomial rates of convergence. However, all these works can only provide an arbitrarily sparse sequence in the lower bound. More precisely, for some $x\in\cX$ and $\alpha\in(0,\infty)$, they show the existence of a sequence $(n_k)_{n\in\N}$ in $\N$  such that $n_k\to\infty$ as $k\to\infty$ and   $n_k^{-\alpha}\leq \|P^{n_k}(x,\cdot)-\pi\|_{\TV}$ for all $k\in\N$. While this demonstrates the failure of geometric ergodicity, the proofs of these results do not control the gap $|n_{k+1}-n_k|$, making it impossible to conclude $C_0n^{-\alpha}\leq \|P^{n}(x,\cdot)-\pi\|_{\TV}$ for a constant $C_0>0$ and all $n\in\N$.

In continuous time, lower bounds on the convergence of Markov processes were studied in~\cite{MR2540073}. This approach compares the tail of the marginal distribution with that of the invariant measure to derive a lower bound on the convergence rate (in the case the tail decay of the invariant measure of the continuous-time process is explicitly known). However, in cases where the invariant measure is not explicitly available (e.g., unadjusted Langevin algorithms in Section~\ref{subsec:biased_algorithms} below) this method can only provide a lower bound for an arbitrarily sparse sequence (cf. previous paragraph). 
Our work similarly employs the tail comparison in~\cite{MR2540073} to derive lower bounds on convergence rates in $f$-variation and the Wasserstein distance. However, the main contribution in this paper lies in the application of the discrete-time 
\textbf{L}-drift conditions to  characterise the failure of the CLT for ergodic averages,  
which does not use the tail comparison convergence argument from~\cite{MR2540073}.
Finally, as in~\cite{brešar2024subexponential}, our result on the lower bounds of the tails of the invariant distribution (see Theorem~\ref{thm:invariant} above) make the tail comparison convergence argument in~\cite{MR2540073} applicable to ergodic Markov chains without an explicitly given invariant distributions.

\smallskip

\paragraph{\textbf{L-drift conditions in discrete time}} The main assumption in our results are the recent \textbf{L}-drift conditions~\nameref{sub_drift_conditions},  first introduced in continuous time~\cite{brešar2024subexponential}. \nameref{sub_drift_conditions} serves as a natural \textit{lower bound} analogue to the classical \textit{upper bound} drift condition (see, e.g.,~\cite{MR2071426}). A key aspect of our lower bound drift conditions is that they involve an estimation of  $P(1/V)$ for a given Lyapunov function $V:\cX\to[1,\infty)$, rather than relying solely on $P(V)$ used in the classical upper bound case~\cite{MR2071426}. Similar transformations were previously employed in~\cite{meynadntweedie,Menshikov21} to establish the transience of Markov chains. 
In addition to our main result Theorem~\ref{thm:CLT}, which provides a necessary condition for the CLT,
our work also yields  lower bounds on the tails of return times to compact sets, the tails of invariant distributions and the rates of convergence to stationarity. Moreover, since the estimates derived  from the \textbf{L}-drift conditions  are non-asymptotic (see Proposition~\ref{prop:lyapunov_return_times} below), it is conceivable that \nameref{sub_drift_conditions} could also yield lower bounds on the \textit{asymptotic variance} of the estimator (cf. non-asymptotic bounds on denoising diffusions in~\cite{brevsar2024non}). This direction is left for future work.

\section{Euler schemes: invariant tails, convergence and the CLT}
\label{sec:examples}
\label{subsec:biased_algorithms}

We apply the general theory to three classes of Euler schemes to examine how the drift and the innovation tails determine invariant tails, convergence rates and the validity of the CLT. Section~\ref{subsec:ULA} treats the unadjusted Langevin algorithm with Gaussian increments and polynomial-tailed targets: the chain
preserves the target tail index, while exhibiting slow convergence and failure of the CLT for suitable observables. Section~\ref{subsec:discretisation_levy_algorithm} considers additive increments with polynomial tails, establishing restrictions on invariant tails and transience under a superlinear drift condition.
The stable-driven samplers considered there are transient if  innovation has heavier tails than the target and  converge at a polynomial rate otherwise.
Finally, Section~\ref{subsec:log-tails_SDE_chain} studies linear drift with multiplicative noise and logarithmic innovation tails. This example identifies sharp CLT thresholds and polynomial convergence exponents even though the innovations and the invariant
law have no positive Euclidean moments. Throughout, the convergence results concern the invariant law $\pi_h$ of the chain at a fixed step size, which, in the sampling examples,  is distinguished from the prescribed target $\pi$.

%We now apply our general theory to biased sampling algorithms for targets with polynomial tails, modelled on discretisations of Langevin-type processes.  Section~\ref{subsec:ULA} considers Gaussian
%increments, proving (among other results) that the tail decay of the  stationary measure of the ULA chain is the same as that of the target. Section~\ref{subsec:discretisation_levy_algorithm} focuses on ULA-type chains with 
%increments of infinite variance, proposing a dichotomy that such algorithms are either divergent (if innovations have heavier tails than the target) or converge at most polynomially (otherwise). Section~\ref{subsec:log-tails_SDE_chain} applies our theory to Euler discretisations of L\'evy-driven SDEs with multiplicative noise and innovations with logarithmically decaying tails. 

Throughout this section, $|\cdot|$ and $\langle\cdot,\cdot\rangle$ denote the
Euclidean norm and inner product on $\R^d$. 

\subsection{Unadjusted Langevin algorithm}
\label{subsec:ULA} 

In the present section, we consider the  twice differentiable and strictly positive density $\pi:\R^d\to(0,\infty)$ satisfying
\begin{equation}
\label{eq:examples_pi}
 \pi(x)=\frac{\ell(|x|)}{|x|^{v+d}}
 \qquad\text{for all $|x|$ sufficiently large and some $v>0$,}
\end{equation}
where
$$
 0<\liminf_{r\to\infty}\ell(r)
 \leq \limsup_{r\to\infty}\ell(r)<\infty.
$$
This class includes the symmetric multivariate Student-$t$ distribution with
$v$ degrees of freedom, whose density is proportional to
$x\mapsto (1+|x|^2/v)^{-(v+d)/2}$.
The parameter $v$ in~\eqref{eq:examples_pi} is the polynomial tail index: in particular,
$\pi(\{|x|>r\})\asymp r^{-v}$ as $r\to\infty$.\footnote{Here and throughout, for functions $f,g:(0,\infty)\to(0,\infty$ we denote $f(r)\asymp g(f)$ as $r\to\infty$ if there exist constants $c,C>0$ satisfying $cg(r)\leq f(r) \leq C g(r)$ for all $r\in(0,\infty)$.}

The unadjusted Langevin algorithm is a time-discretisation of the Langevin diffusion
\begin{equation*}
\ud Y_t = \nabla \log \pi(Y_t)\ud t + \sqrt{2}\ud B_t,
\end{equation*}
where $B$ is a standard $d$-dimensional Brownian motion and $\nabla$ is the gradient of the function $\log\pi$.
The Euler-Maruyama discretisation scheme yields the ULA Markov chain with the transition law given by
\begin{equation}
\label{eq:ULA}
X_{k+1}^{(h)} = X_k^{(h)} + h \nabla \log \pi(X_k^{(h)}) + \sqrt{2h} Z_{k+1},\quad\text{for $k\in\N$,}
\end{equation}
where $(Z_k)_{k\in\N}$ is an i.i.d. sequence of standard  $d$-dimensional Gaussian random vectors (i.e., with zero mean and identity covariance matrix) and $h\in(0,\infty)$ is a step size. 
Equivalently, $X^{(h)}$ has the Gaussian transition kernel
\begin{equation}
\label{eq:ULA_kernel}
 Q_h(x,\cdot)=N\bigl(x+h\nabla\log\pi(x),2h\Id\bigr),
 \qquad x\in\R^d,
\end{equation}
where $\Id$ is the $d\times d$ identity matrix.  It is well known that this discretisation typically introduces bias: if the chain $X^{(h)}$ is
ergodic, its invariant distribution $\pi_h$ will generally differ from the
target~$\pi$.

\begin{thm}
\label{thm:ULA}
Let the target density $\pi$ in~\eqref{eq:examples_pi} be twice continuously differentiable with $\ell$ satisfying 
$r(\log \ell(r))'\to 0$ as $r\to\infty$. Then for any $h>0$, the ULA chain $X^{(h)}$ with Markov transition kernel $Q_h$ in~\eqref{eq:ULA_kernel} admits an invariant measure $\pi_h$ and the following statements hold.

\begin{myenumi}[label=(\Roman*)]
\item \label{thm:ULA_CLT}
\textup{\textbf{[CLT-necessary condition]}}
Fix a locally bounded measurable function $g:\R^d\to\R$.
\begin{enumerate}
\item[(Ia)] Suppose that $\pi_h(g^2)<\infty$ and that $g(x)\geq |x|^s$
for some $2s\in(0,v)$ and all $|x|$ sufficiently large.
If $2s\in(v-2,v)$, then the CLT for $S_n(g)$
\textbf{fails}.

\item[(Ib)] Suppose that $g$ is bounded and
$\liminf_{|x|\to\infty}\bigl(g(x)-\pi_h(g)\bigr)>0$.
If $v\in(0,2)$, then the CLT for $S_n(g)$
\textbf{fails}.
\end{enumerate}
\item  \label{thm:ULA_lower_bounds} \textup{\textbf{[Lower bounds]}} Fix $x\in\R^d$ and, for $p\in[0,v)$, define $f_p(y)\coloneqq \max\{1,|y|^p\}$, $y\in\R^d$.
\begin{enumerate}
\item[(IIa)] For any $p\in[0,v)$ and every $\eps>0$, there exists a constant $c_p\in(0,\infty)$ such that
$$
c_p n^{-\frac{v-p}{2}-\eps}\leq \|Q_h^n(x,\cdot)-\pi_h\|_{f_p}\quad\text{for all $n\in\N\setminus\{0\}$.}
$$
\item[(IIb)] For any $p\in[1,v)$ and every $\eps>0$, there exists a constant $c_{\cW,p}\in(0,\infty)$ such that
$$
c_{\cW,p} n^{-\frac{v-p}{2p}-\eps}\leq \cW_p(Q_h^n(x,\cdot),\pi_h)\quad\text{for all $n\in\N\setminus\{0\}$.}
$$
\end{enumerate}
\item \label{thm:ULA_invariant} 
\textup{\textbf{[Tails of invariant measure $\pi_h$]}}
For any $\eps>0$, there exist $c,C\in(0,\infty)$ satisfying  
$$
cr^{-v-\eps}\leq \pi_h(|\cdot|\geq r) \leq Cr^{-v+\eps}, \quad \text{for all $r\in[1,\infty)$.}
$$ 
Thus, for any $h>0$, the tail decay of the invariant distribution $\pi_h$ of $X^{(h)}$ matches asymptotically the tail decay of the target $\pi$ in~\eqref{eq:examples_pi}.
\end{myenumi}
\end{thm}

The proof of Theorem~\ref{thm:ULA} is in Section~\ref{subsec:ULA_proofs}. It consists of showing that the functions $V(x)=|x|\vee 1$ (here and throughout we denote $u\vee v\coloneqq \max\{u,v\}$ for any $u,v\in\R$), $\Psi(r) = r^{v+2+\eps}$, $\varphi(1/r) \propto 1/r^{3}$ satisfy the \textbf{L}-drift conditions~\nameref{sub_drift_conditions} and applying Theorem~\ref{thm:CLT} and  Theorem~\ref{thm:f_rate} (with $h(r)=r^{v+2-\eps}$) to get parts~\ref{thm:ULA_CLT} and~\ref{thm:ULA_lower_bounds} in Theorem~\ref{thm:ULA}, respectively.  Theorem~\ref{thm:invariant}, together with the classical upper bounds on the tail of the invariant distribution of an ergodic Markov chain (see e.g.~\cite[Thm~16.1.12]{MarkovChains}),
implies Theorem~\ref{thm:ULA}\ref{thm:ULA_invariant}.

By Theorem~\ref{thm:ULA}\ref{thm:ULA_CLT}, the ergodic CLT fails
for functions satisfying part~(Ia) and, when the polynomial
tail-decay exponent satisfies $v<2$, also for bounded functions
satisfying the tail condition in part~(Ib), including indicators of
tail sets. Theorem~\ref{thm:ULA}\ref{thm:ULA_lower_bounds} implies that the convergence rate of the ULA chain to $\pi_h$ cannot be faster than $n^{v/2}$, where $v>$ is the tail decay of the target $\pi$.  By Theorem~\ref{thm:ULA}\ref{thm:ULA_invariant}, the tail decay of $\pi_h$ matches that of the target $\pi$  for \textit{every} fixed $h>0$. Section~\ref{subsec:discretisation_levy_algorithm} shows that this preservation of the target tail exponent is not a generic feature of unadjusted algorithms: it can fail when the increments have infinite variance.

\begin{rem}
\label{rem:matching_bounds_ULA}
Under the assumptions of Theorem~\ref{thm:ULA}, we can obtain matching upper bounds on the convergence rate and cases where CLT holds by applying the drift conditions from~\cite{MarkovChains}.  By combining these results, we demonstrate that our lower bounds match the known upper bounds and that the condition under which the CLT fails is sharp.
\begin{myenumi}[label=(\Roman*)]
\item \textbf{[CLT characterisation]} 
Let $g:\R^d\to\RP$ be locally bounded and measurable, with
$g(x)=|x|^s$ for all sufficiently large $|x|$, where $0<2s<v$.
Then $g\in L^2(\pi_h)$ by
Theorem~\ref{thm:ULA}\ref{thm:ULA_invariant}.
The asymptotic variance associated with $S_n(g)$ is finite
when $2s<v-2$, by the drift condition in
Proposition~\ref{prop:drift_ULA} and
\cite[Thm.~5 and the polynomial discussion on
pp.~245--246]{MR2446322}.
When $v-2<2s<v$, the CLT fails by
Theorem~\ref{thm:ULA}\ref{thm:ULA_CLT}.

For bounded measurable $g:\R^d\to\R$, the same sufficient
criterion gives finite asymptotic variance whenever $v>2$.
If $v<2$ and
\begin{equation}
\label{eq:lower_bound_assm_CLT}
 \liminf_{|x|\to\infty}
 \bigl(g(x)-\pi_h(g)\bigr)>0,
\end{equation}
then the CLT fails by
Theorem~\ref{thm:ULA}\ref{thm:ULA_CLT}.
In particular, condition~\eqref{eq:lower_bound_assm_CLT}
holds for $g(x)=\mathbf{1}_{\{|x|>r\}}$, with any $r>0$,
since its left-hand side equals
$\pi_h(\{|x|\le r\})>0$.
These conclusions leave open the boundary case $2s=v-2$
for the unbounded functions above, and $v=2$ for bounded
functions satisfying~\eqref{eq:lower_bound_assm_CLT}.

\item \textbf{[Matching bounds]} For every $x\in\R^d$, $p\in[0,v)$ and $\eps>0$ there exist constants $c\in(0,\infty)$ (by Theorem~\ref{thm:ULA}\ref{thm:ULA_lower_bounds}) and $C\in(0,\infty)$ (by~\cite[Thm~2.8 and Eq.~(2.9)]{MR2071426}), such that
\begin{equation}
\label{eq:f_p_variation_sided}
c n^{-\frac{v-p}{2}-\eps}\leq \|Q_h^n(x,\cdot)-\pi_h\|_{f_p}\leq Cn^{-\frac{v-p}{2}+\eps} \quad\text{for all $n\in\N\setminus\{0\}$.}
\end{equation}
Moreover, assuming $v>p\geq1$, the optimality of the lower bound on the 
 $\cW_p$-Wasserstein distance 
 in Theorem~\ref{thm:ULA}\ref{thm:ULA_lower_bounds}
 follows from the upper bound in~\eqref{eq:f_p_variation_sided} and 
 the elementary residual-mass coupling inequality
$W_p(\mu,\nu)^p
\le2^{p-1}\int_{\R^d} |y|^p\,|\mu-\nu|(dy)
\le2^{p-1}\|\mu-\nu\|_{f_p}$.

The threshold $2s=v-2$ for the failure of the CLT for the function $g(x)=|x|^s$ for all sufficiently large $|x|$ and $0\leq2s<v$ coincides with the threshold for summability of
convergence in $f_p$-variation, with $p=2s$.
Indeed, by~\eqref{eq:f_p_variation_sided}, for every
$x\in\mathbb R^d$ and $p\in[0,v)\setminus\{v-2\}$ we have 
\begin{equation*}
 \sum_{n=1}^{\infty}
 \|Q_h^n(x,\cdot)-\pi_h\|_{f_p}<\infty
 \quad\Longleftrightarrow\quad
 p<v-2.
\end{equation*}
\end{myenumi}
\end{rem}

\begin{rem}
The results in Theorem~\ref{thm:ULA} are not restricted to target densities whose tails factor through the radial map $x\mapsto|x|$. Theorem~\ref{thm:ULA}  holds for example for satisfying $\pi(x)\asymp |x|^{-v-d}$ as $|x|\to\infty$  and an appropriate analogous assumption on the gradient of $\log \pi$ (e.g., $\nabla\log\pi(x)=-(v+d)\frac{x}{|x|^2}+o(|x|^{-1})$ uniformly as $|x|\to\infty$). Assumption~\eqref{eq:examples_pi} is imposed solely for ease of exposition, since ULA chains targeting this simpler class of densities already exhibit all the phenomena discussed in Section~\ref{sec:main_results}.
\end{rem}

A failure of the CLT does not necessarily imply that the corresponding asymptotic variance is infinite. However, in the case of the ULA chain in~\eqref{eq:ULA}, as a consequence of Theorem~\ref{thm:ULA}\ref{thm:ULA_invariant}, the asymptotic variance is indeed infinite  (see Section~\ref{subsec:ULA_proofs} below for the proof Corollary~\ref{cor:infinite_asymp_var}).

\begin{cor}
    \label{cor:infinite_asymp_var} Under the assumptions of Theorem~\ref{thm:ULA}\ref{thm:ULA_CLT}, if either $g(x)\geq |x|^s$ with $v-2<2s<2$ or bounded $g$  with $\liminf_{|x|\to\infty}\bigl(g(x)-\pi_h(g)\bigr)>0$ and $v<2$,  the asymptotic variance $\sigma_g^2$ is infinite.
\end{cor}

\subsection{Unadjusted algorithms with increments of infinite variance}
\label{subsec:discretisation_levy_algorithm}
The potential for rapid movement through the tails motivates unadjusted Langevin-type algorithms driven by heavy-tailed, rather than Gaussian, increments. This class includes Euler--Maruyama discretisations of L\'evy-driven SDEs, whose continuous-time dynamics may admit heavy-tailed invariant distributions while converging exponentially fast to equilibrium; see, for example,~\cite{Majka21,MR4580902} for theoretical results and~\cite{simsekli2020fractional,nguyen2019non} for applications in machine learning.

For simplicity of exposition, let
%In this paper we are interested in discretisation of $Y$, which yields 
the Markov chain $X=(X_n)_{n\in\N}$ have a transition law governed by the dynamics
\begin{equation}
\label{eq:alpha_stable_ULA}
X_{k+1} = X_k + hb(X_k) + Z^{(h)}_{k+1},\quad\text{for $k\in\N$,}
\end{equation}
for some fixed parameter $h>0$ and a continuous vector field $b:\R^d\to\R^d$, where 
$(Z_k^{(h)})_{k\in\N}$ is an 
 i.i.d. sequence in $\R^d$ with mean zero and density $p_{\alpha,h}$ satisfying
\begin{equation}
    \label{eq:p_alpha_assumption}
0<\liminf_{r\to\infty}r^\alpha \int_{\{|x|>r\}}p_{\alpha,h}(x)\ud x\leq \limsup_{r\to\infty}r^\alpha \int_{\{|x|>r\}}p_{\alpha,h}(x)\ud x<\infty,\quad\text{for some $\alpha\in(1,2)$.}
\end{equation}
We further assume that 
$p_{\alpha,h}$ converges weakly to a delta measure at the origin $0\in\R^d$ as $h\to0$.
We note that the chain $X$ in~\eqref{eq:alpha_stable_ULA} need not necessarily satisfy the standard assumptions of Section~\ref{subsec:General_Assumptions} above.

If $q_\alpha$ is an $\alpha$-stable density on $\R^d$ (possibly
non-isotropic), let $Z_k^{(h)}=h^{1/\alpha}Z_k$, where $Z_k$ has density
$q_\alpha$.  Equivalently,
$$
 p_{\alpha,h}(z)=h^{-d/\alpha}q_\alpha(h^{-1/\alpha}z).
$$
Then~\eqref{eq:alpha_stable_ULA} yields the standard Euler--Maruyama chain
$X$,
\begin{equation}
\label{eq:alpha_stable_ULA_disc}
X_{k+1} = X_k + hb(X_k) + h^{1/\alpha} Z_{k+1},\quad\text{for $k\in\N$,}
\end{equation}
where the sequence $(Z_k)_{k\in\N}$ is i.i.d. with law $q_\alpha$.
For an $\alpha$-stable L\'evy process $L$ with $L_1\sim q_\alpha$,
there exists a vector field 
$b_\pi:\R^d\to\R^d$, given in terms of $\pi$,  such that 
the process $Y$,  satisfying the following L\'evy-driven SDE 
\begin{equation}
\label{eq:Stable_SDE}
\ud Y_t = b_\pi(Y_t)\ud t + \ud L_t\quad \text{(for an isotropic $\alpha$-stable law $q_\alpha$ on $\R^d$),}
\end{equation}
converges to $\pi$ (see e.g.~\cite[Eq.~(4)]{Majka21} for the definition of $b_\pi$). For $b=b_\pi$, the chain $X$ defined in~\eqref{eq:alpha_stable_ULA_disc} above is the Euler-Maruyama discretisation of the continuous-time process $Y$.

%where  $(Z_k)_{k\in\N}$ is an i.i.d. sequence of centred $\alpha$-stable random vectors and $h>0$ is a step size.

\begin{thm}
\label{thm:levy_driven_discretisation}
Let the chain $X$ satisfy the recursion in~\eqref{eq:alpha_stable_ULA} with an arbitrary continuous vector field $b:\R^d\to\R^d$, some $h>0$ and $p_{\alpha,h}$ satisfying~\eqref{eq:p_alpha_assumption}. Then the following statements hold.
\begin{myenumi}[label=(\alph*)]
    \item \label{thm:levy_driven_discretisation:invariant} %Assume $|b(x)|\leq C(|x|+1)$ for all $x\in\R^d$ and some constant $C>0$. If $0<h<1/(2C)$ and 
    For every fixed $h>0$, if the chain $X$ in~\eqref{eq:alpha_stable_ULA}  admits an invariant distribution $\pi_h$, then there exists a constant $C_{h,\pi_h}\in(0,\infty)$ satisfying 
    $$\pi_h(|\cdot|>r)\geq C_{h,\pi_h} r^{-\alpha}\qquad\text{for all $r\in[1,\infty)$. }$$
    \item \label{thm:levy_driven_discretisation:transient} If $\liminf_{|x|\to\infty}|b(x)|/|x|=\infty$, then the chain $X=(X_n)_{n\in\N}$ is transient: $$\P_x\left(\lim_{n\to\infty}|X_n|=\infty\right)=1\qquad\text{for all $x\in\R^d$.}
    $$
\end{myenumi}
\end{thm}

\begin{rem}
\label{rem:discretisation_Levy_SDE}
\begin{myenumi}[label=(\roman*)]
 \item Note that if the vector field  $b$ in~\eqref{eq:alpha_stable_ULA} is Lipschitz, it satisfies the linear-growth condition $|b(x)|\leq C(1+|x|)$, $x\in\R^d$. In this case the corresponding chain $X$ will admit an invariant distribution if $b$ is inward pointing and unbounded~\cite{MR2071426}. 
 Important examples include discretisations of L\'evy-driven OU processes (extensively studied in queueing, see e.g.~\cite{MR3910024} and the references therein). 
 \item Theorem~\ref{thm:levy_driven_discretisation}\ref{thm:levy_driven_discretisation:invariant} is proved in Section~\ref{subsec:proof_for_levy_driven} by a direct argument using the stationarity  of the chain $X$ in~\eqref{eq:alpha_stable_ULA}: $\int_{\R^d}\P(x+hb(x)+Z^{(h)}_1\in\ud y)\pi_h(\ud x) =\pi_h(\ud y)$. If in addition the vector field $b$ satisfies the linear-growth condition, Theorem~\ref{thm:levy_driven_discretisation}\ref{thm:levy_driven_discretisation:invariant} can also be proved using \textbf{L}-drift conditions~\nameref{sub_drift_conditions} (with functions $V(x)=\log (|x| \vee \mathrm{e})$, $\Psi(r) = \exp(\alpha r)$ and $\varphi(1/r) \propto 1/r^2$) by applying Theorem~\ref{thm:invariant} from Section~\ref{sec:main_results}. The proof of Theorem~\ref{thm:levy_driven_discretisation}\ref{thm:levy_driven_discretisation:transient} relies on showing that the process $1/V(X)$ is a supermartingale, which implies transience by a version of the Foster-Lyapunov criterion  (see e.g.~\cite{Menshikov21,meynadntweedie}).
\item It has been shown that, in continuous time, replacing Langevin diffusion with a Lévy-driven SDE with an appropriate drift may lead to an exponentially ergodic continuous-time process converging to a heavy-tailed measure~\cite{Majka17,Majka21,MR4580902}. For example~\cite[Thm~1.1]{Majka21} states that a  L\'evy-driven SDE with an isotropic $\alpha$-stable driver converges exponentially fast to an invariant distribution with at least $\alpha$ moments. In particular, this includes  stationary distributions that possess the second moment.
\end{myenumi}
\end{rem}

Theorem~\ref{thm:levy_driven_discretisation} implies the following fact about the Euler-Maruyama discretisations of L\'evy-driven SDEs, leading to a dichotomy discussed below which, to the best of our knowledge, was previously unknown.

\begin{cor}
\label{cor:Levy_driven_euler_schemes}   
%The Euler-Maruyama chain in~\eqref{eq:alpha_stable_ULA_disc} for the Stable-driven SDE in~\eqref{eq:alpha_stable_ULA} (with $\alpha\in(1,2)$ and $b=b_\pi$) either converges to an invariant distribution with an infinite  $(\alpha+\eps)$-moment (for any arbitrarily small $\eps>0$) or is transient (i.e., does not converge at all).
For any \(h>0\), every invariant probability distribution of the  Euler chain with stable increments in~\eqref{eq:alpha_stable_ULA_disc} has infinite moments of all orders \(p\ge\alpha\). If, in addition, the drift $b_\pi$ is superlinear, \(\liminf_{|x|\to\infty}|b_\pi(x)|/|x|=\infty\), then \(|X_n|\to\infty\) almost surely from every starting point.
\end{cor}

\subsubsection{Two typical regimes for stable-driven samplers} The principle is as follows: the scheme in~\eqref{eq:alpha_stable_ULA} is  either divergent or, if it  does converge, the rate is at  most polynomial.
To see this, it is instructive to apply Corollary~\ref{cor:Levy_driven_euler_schemes}  to a class of  target with globally radial $\pi$ introduced in~\cite[Ass.~(A)]{Majka21}. By~\cite[Lem.~2.4]{Majka21} the potential  \(U=-\log\pi\) and drift $b_\pi$ satisfy the inequality  
$\langle x,b_\pi(x)\rangle\le-c\frac{e^{U(x)}|x|^2}{(1+|x|)^{d+\alpha}}$
outside a ball for some constant $c>0$. Moreover, if \(\pi(x)\asymp |x|^{-v-d}\) and \(v>\alpha\), then
\[
\frac{|b_\pi(x)|}{|x|}\ge c|x|^{v-\alpha}\to \infty\quad\text{as $|x|\to\infty$.}
\]
Thus Theorem~\ref{thm:levy_driven_discretisation}\ref{thm:levy_driven_discretisation:transient} implies transience of the Euler chain~\eqref{eq:alpha_stable_ULA_disc} for \textbf{every} fixed step size $h>0$ for this class of targets. 
%By Corollary~\ref{cor:Levy_driven_euler_schemes} and the class of examples in the previous paragraph, 
In the regime $v>\alpha$, the Euler-Maruyama discretisatons of L\'evy-driven SDEs  exhibit a theoretical obstruction: in contrast to the continuous-time limit in~\eqref{eq:Stable_SDE}, even for an arbitrarily small $h>0$, the scheme can only approximate densities  $\pi$ that do not possess a finite second moment (since the $\alpha$-stable driver requires $\alpha<2$). 

However, in the regime $0<v<\alpha$,  by~\cite[Lemmas~2.4~\&~2.5]{Majka21}, for large $|x|$ we have $\langle x,b_\pi(x)\rangle\leq |x|^{2-(\alpha-v)}$ and the drift is sublinear at infinity $|b_\pi(x)|\leq |x|^{1-(\alpha-v)}$. Hence, for any $\eps>0$ and starting point $x\in\R^d$, by Theorem~\ref{thm:f_rate} there exists a constant $c_{h,x,\eps}>0$ such that  
\begin{equation}  \label{eq:lower_Bound_estimate_stable_ULA}\|Q_h^n(x,\cdot)-\pi_h\|_{\textrm{TV}}\geq c_{h,x,\eps} n^{-\frac{v}{\alpha-v}-\eps}\qquad \forall n\in \N,   
\end{equation}
where $Q_h$ denotes the transition kernel of the chain in~\eqref{eq:alpha_stable_ULA_disc}. The proof of~\eqref{eq:lower_Bound_estimate_stable_ULA} is in Section~\ref{subsec:proof_for_levy_driven} below.
Note also that, unlike in the regime $\alpha<v$, by~\cite[Thm~2.2 and Lemmas~2.10 and~5.1]{brešar2024subexponential} in this case  the slow polynomial convergence also holds in continuous time for SDE~\eqref{eq:Stable_SDE}.

The obstruction given by  Theorem~\ref{thm:levy_driven_discretisation}\ref{thm:levy_driven_discretisation:transient} in the regime $\alpha<v$ is a heavy-tailed analogue to~\cite[Thm~3.2]{MR1440273}, which states that the Euler-Maruyama discretisation of a Langevin diffusion targeting a probability measure with lighter than Gaussian tails necessarily yields a transient chain. 
Theorem~\ref{thm:levy_driven_discretisation} thus shows that, to target distributions with tails lighter than those of the driving noise, a discretisation with heavy-tailed increments cannot rely on additive noise alone and should instead allow the noise magnitude to depend on the state.

In the present subsection, we encountered sub and super linear drifts with additive noise. The next section studies the case of an Euler scheme for a L\'evy-driven SDE with linear drift and multiplicative noise.

\subsection{Linear drift, multiplicative noise and logarithmic tails}
\label{subsec:log-tails_SDE_chain}

The preceding subsection treats additive noise with polynomial tails
and the restrictions imposed by the drift and innovation law on the
Euler scheme. We now consider a different regime: the deterministic
Euler update is a linear contraction, the dispersion coefficient may
depend on the state, and the innovations have logarithmic tails.
These innovations possess no positive Euclidean moments, so the
finite-mean assumption of
Section~\ref{subsec:discretisation_levy_algorithm} no longer applies.
The example shows how the same lower-drift theory determines
invariant tails, polynomial convergence exponents and exact CLT
thresholds in this regime. Inspired
by~\cite[Sec.~3.2]{brešar2024subexponential}, we introduce a
multidimensional L\'evy-driven SDE and then study its Euler
discretisation.

Consider a multidimensional 
L\'evy-driven SDE with possibly multiplicative noise,
\begin{align}
 dY_t&=-AY_t\,dt+\Sigma(Y_{t-})\,dL_t
       \label{eq:log-euler-sde}
\end{align}
where $A\in\mathbb R^{d\times d}$ satisfies
$\langle Ax,x\rangle\ge\lambda|x|^2$ for all $x\in\mathbb R^d$ and some $\lambda>0$. By a linear change of coordinates, coercivity can be reduced to requiring that every eigenvalue of $A$ has a positive real part, see Remark~\ref{rem:coercivity_assump} in
Section~\ref{subsec:log-tails_SDE_chain_proofs}.
We assume that  $\Sigma:\mathbb R^d\to\mathbb R^{d\times d}$ is globally Lipschitz satisfying the following inequalities for some  constants $0<\sigma_-\le\sigma_+<\infty$:
\begin{equation*}
 \sigma_-|z|\le |\Sigma(x)z|\le\sigma_+|z|,
 \qquad x,z\in\mathbb R^d,
\end{equation*}
and $L$ a $d$-dimensional L\'evy process. Fix $\kappa>1$ and assume
that, for every $t>0$, $L_t$ has a continuous strictly positive density
$p_t$ satisfying, for some $K_t<\infty$,
\begin{equation}\label{eq:log-euler-noise-density}
 p_t(z)\le K_t(1+|z|)^{-d},\qquad z\in\mathbb R^d,
\end{equation}
The density bound~\eqref{eq:log-euler-noise-density} is satisfied
by many standard L\'evy processes, including non-degenerate
isotropic stable processes.
%Requirement~\eqref{eq:log-euler-noise-density} is a generic assumption on the density of $L_t$.
Moreover, assume there exist $c_t,C_t,R_t\in(0,\infty)$ such that
\begin{equation}\label{eq:log-euler-noise}
 c_t(\log r)^{-\kappa}\le\mathbb P(|L_t|>r)
 \le C_t(\log r)^{-\kappa},\qquad \text{for all } r\ge R_t>1.
\end{equation}
%(Example~\ref{ex:log-euler-driver} in the appendix gives an explicit
%infinite-variation pure-jump L\'evy process with  properties~\eqref{eq:log-euler-noise-density} and~\eqref{eq:log-euler-noise}.)
The crucial condition in~\eqref{eq:log-euler-noise}, stipulating that tails of $L_t$ are logarithmic, holds  for every $t>0$,
whenever the L\'evy measure $\nu$ of $L$ satisfies an analogous requirement:  for some constants
$0<c_\nu\le C_\nu<\infty$ and $R_\nu>1$,
\[
 c_\nu(\log r)^{-\kappa}
 \le \nu\bigl(\{z\in\mathbb R^d:|z|>r\}\bigr)
 \le C_\nu(\log r)^{-\kappa},
 \qquad r\ge R_\nu
\]
(see the L\'evy--It\^o decomposition and comparison with
subexponential compound-Poisson tails
\cite[Thm~19.2 and Rem.~25.14]{MR3185174}).

The Lipschitz assumption on $\Sigma$ implies existence of a unique non-explosive strong solution
of~\eqref{eq:log-euler-sde}.
 For a concrete state-dependent example one may take
$A=\mu I_d$, $\Sigma(x)=(1+\delta\tanh x_1)I_d$,
$\mu>0$, $0<\delta<1$,
with $0<h<1/\mu$ and the driver $L$ in Example~\ref{ex:log-euler-driver} satisfying~\eqref{eq:log-euler-noise-density} and~\eqref{eq:log-euler-noise} (note that $\sigma_-=1-\delta$, $\sigma_+=1+\delta$ and
$a_h=b_h=1-h\mu$).

The Euler-scheme chain is given by the recursion, respectively,
\begin{align}
 X^{(h)}_{n+1}&=B_hX^{(h)}_n+
                  \Sigma(X^{(h)}_n)Z^{(h)}_{n+1},
 \qquad B_h=I_d-hA,
       \label{eq:log-euler-chain}
\end{align}
where the $Z^{(h)}_n$ are iid with the law of $L_h$. 
Assume there exist constants $0<b_h\le a_h<1$ such that
\begin{equation}\label{eq:log-euler-contraction}
 b_h|x|\le |B_hx|\le a_h|x|,\qquad x\in\mathbb R^d.
\end{equation}
This holds for example if we restrict the step size by
$0<h<\min\left\{1/\|A\|,2\lambda/\|A\|^2\right\}$, see the first paragraph of Section~\ref{subsec:log-tails_SDE_chain_proofs}.
Constants in what follows are allowed to depend on $h$.

For any $q\geq0$ and $\rho\in(0,1]$, define the $f$-variation function $f_q:\R^d\to[1,\infty)$ and an appropriate complete metric $d_\rho$ inducing the Euclidean topology on $\R^d$ as follows
\begin{equation}\label{eq:log-euler-V-metric}
 V(x)\coloneqq 1+\log(1+|x|),\qquad f_q(x)\coloneqq V(x)^q,
 \qquad d_\rho(x,y)\coloneqq\bigl[\log(1+|x-y|)\bigr]^\rho.
\end{equation}
Write $\cW_{p,\rho}$ for the $p$-Wasserstein distance defined in~\eqref{eq:def_Wassesstein_distance} with  metric
$d_\rho$ on $\R^d$. The distinction from Euclidean Wasserstein distance is essential
in this example as neither the marginals of the chain $X^{(h)}$ nor its stationary probability measure possess any polynomial moments.

The following result is proved in Section~\ref{subsec:log-tails_SDE_chain_proofs} below.

\begin{thm}
\label{thm:log-euler}
Under the assumptions above on $A$ and $\Sigma$, the noise conditions
\eqref{eq:log-euler-noise-density}--\eqref{eq:log-euler-noise}, and
\eqref{eq:log-euler-contraction}, the chain~\eqref{eq:log-euler-chain}, with
transition kernel $Q_h$, is aperiodic and positive Harris recurrent and
has a unique invariant probability $\pi_h$.
The following statements hold.
\begin{enumerate}
\item[\textup{(I)}] \textup{\textbf{[Invariant tails and moments]}} 
For every $\eps>0$, there exist $c_h,C_h\in(0,\infty)$ such that
\begin{equation*}
 c_h (\log r)^{-(\kappa-1)-\eps}\le\pi_h(|\cdot|>r)
       \le C_h (\log r)^{-(\kappa-1)},\qquad r\ge2.
\end{equation*}
Equivalently, $\frac{\log\pi_h(|\cdot|>r)}{\log \log r}\to-(\kappa-1)$ as $r\to\infty$. Moreover, for $q>0$,
\begin{equation}\label{eq:log-euler-log-moments}
 \pi_h(V^q)<\infty\quad\Longleftrightarrow\quad q<\kappa-1.
\end{equation}
In particular, $\pi_h$ has no finite positive Euclidean moment.

\item[\textup{(II)}] 
\textup{\textbf{[Convergence in $f$-variation]}}
For every $x\in\mathbb R^d$, $q\in[0,\kappa-1)$ and
$\varepsilon>0$, there exist $c,C\in(0,\infty)$ such that
\begin{equation}\label{eq:log-euler-f-rates}
 c n^{-(\kappa-1-q)-\eps}
 \le \|Q_h^n(x,\cdot)-\pi_h\|_{f_q}
 \le C n^{-(\kappa-1-q)+\varepsilon},\qquad n\ge1.
\end{equation}
Thus the total-variation exponent is $\kappa-1$, and the
$f_q$-variation exponent is $\kappa-1-q$.

\item[\textup{(III)}] \textup{\textbf{[Wasserstein convergence]}} 
For $0<\rho\le1$, $p\ge1$ and $\rho p<\kappa-1$, both
$Q_h^n(x,\cdot)$ and $\pi_h$ have finite $p$-moments with respect to the metric 
$d_\rho$ in~\eqref{eq:log-euler-V-metric}. For every $x\in\mathbb R^d$ and $\varepsilon>0$,
there exist $c,C\in(0,\infty)$ such that
\begin{equation}\label{eq:log-euler-W-rates}
 c n^{-(\kappa-1-\rho p)/p-\eps}
 \le \cW_{p,\rho}(Q_h^n(x,\cdot),\pi_h)
 \le C n^{-(\kappa-1-\rho p)/p+\varepsilon},\qquad n\ge1.
\end{equation}
If $\rho p\ge\kappa-1$, this Wasserstein distance is infinite,
even when defined as an extended transport cost.
For the Euclidean metric in definition~\eqref{eq:def_Wassesstein_distance} one has, for every $p\ge1$,
\begin{equation}\label{eq:log-euler-Euclidean-infinite}
 \cW_p(Q_h^n(x,\cdot),\pi_h)=\infty,
 \qquad x\in\mathbb R^d,\ n\ge1,
\end{equation}
again interpreting $\cW_p$ as an extended transport cost.
\item[\textup{(IV)}] 
\textup{\textbf{[CLT threshold]}}
Let $g:\mathbb R^d\to\mathbb R$ be locally bounded and measurable.
If $g\in L^2(\pi_h)$ and $g(x)\ge cV(x)^s$ outside a ball for some
$c>0$ and $s>0$, then the CLT fails whenever
$2s\ge\kappa-2$.
If $g$ is bounded and
 $\liminf_{|x|\to\infty}\bigl(g(x)-\pi_h(g)\bigr)>0$,
then the CLT fails for $1<\kappa\le2$.
Conversely, if $|g(x)|\le CV(x)^s$ for some $s\ge0$ and
$2s<\kappa-2$, then the stationary asymptotic variance is finite,
and the CLT holds from every starting state whenever $g$ is not
$\pi_h$-almost everywhere constant.
\end{enumerate}
\end{thm}

\begin{rem}
If $s>0$ and $g$ is locally bounded and measurable, with
$g=V^s$ outside a ball, then $g\in L^2(\pi_h)$ exactly when
$2s<\kappa-1$, and in this range
the CLT holds
$\Longleftrightarrow
 2s<\kappa-2$.
The equality case $2s=\kappa-2$ belongs to the failure regime.
This threshold coincides with the threshold for summability of
convergence in $f_q$-variation, with $q=2s$.
Indeed, by~\eqref{eq:log-euler-f-rates}, for every
$x\in\mathbb R^d$ and $q\in[0,\kappa-1)\setminus\{\kappa-2\}$,
\begin{equation*}
 \sum_{n=1}^{\infty}
 \|Q_h^n(x,\cdot)-\pi_h\|_{f_q}<\infty
 \quad\Longleftrightarrow\quad
 q<\kappa-2.
\end{equation*}
\end{rem}

\begin{rem}
The exponent $\kappa$ of the logarithmic innovation tail loses one unit
in the invariant tail. An excursion created by a jump of size $e^r$
lasts on the order of $r$ steps, since the deterministic Euler update
contracts distances by a fixed factor. The same scale determines the
mixing exponents and the CLT threshold.
For example, for the indicator of the complement of any ball, the CLT
holds precisely when $\kappa>2$.
The natural choice $\rho=1$ in~\eqref{eq:log-euler-W-rates} permits
$1\le p<\kappa-1$. If $1<\kappa\le2$, choosing
$0<\rho<\kappa-1$ gives a nonempty Wasserstein range.
The use of $d_\rho$ is therefore a substantive part of the statement:
finite Euclidean Wasserstein bounds cannot hold.
For $A=\mu I_d$, the condition
\eqref{eq:log-euler-contraction} holds for $0<h<2/\mu$, except
$h=1/\mu$. That exceptional step must be excluded: if $\Sigma$ is
constant, the Euler chain is then independent and identically distributed,
and its invariant logarithmic-tail exponent is $\kappa$, not $\kappa-1$.
\end{rem}

\section{Proofs of the general results}
\label{sec:main_proofs}

\subsection{Return time estimates, petite sets and the failure of the CLT under~\texorpdfstring{\nameref{sub_drift_conditions}}{L(V,phi,Psi)}}
\label{subsec:return_times}
%In the remainder of the paper, we consider the process 
%$\kappa \coloneqq V(X)$ under Assumption~\nameref{sub_drift_conditions}. 
The key estimate required in the proofs of our main theorems, stated in Section~\ref{sec:main_results} above, is given in Proposition~\ref{prop:lyapunov_return_times}, which essentially bounds from below the tail of  the return time $S_{(\ell)}=\inf\{n\in\N: V(X_n)<\ell\}$ of the process $X$ into the set $\{V< \ell\}$.

\begin{prop}~\label{prop:lyapunov_return_times}
Let Assumption~\nameref{sub_drift_conditions} hold with some $\ell_0\in[1,\infty)$.
Then for every $\ell_0'\in[\ell_0,\infty)$,  any measurable  $h:\RP\to\RP$, which is continuous, positive and non-decreasing on $[\ell_0',\infty)$,
any  $q\in(0,1)$, $\ell\in(\ell_0',\infty)$ and $m\in(0,\infty)$,  for all
$x\in\{\ell+m\leq V\}$ and all times $t>G_h(\ell+m)$
we have 
% and $q\in(0,1)$,  inequalities~\eqref{eq:Lyapunov_lower_bound} and~\eqref{eq:Lyapunov_squared_bound} hold for
%all times $t\geq h(\ell)(q(1-q)/\ell\varphi(1/\ell))$:
\begin{align}
\label{eq:Lyapunov_lower_bound}
\P_x\left(\sum_{j=0}^{S_{(\ell)}-1} h\circ V(X_j)\geq t\right) \geq & q^2\1{V(x)\leq G_h^{-1}(t)/(1-q)^2}\frac{C_{\ell,m}}{\Psi(G_h^{-1}(t)/(1-q)^2)}\\\nonumber
&+q^2\1{V(x)> G_h^{-1}(t)/(1-q)^2}.
\end{align}
The constant $C_{\ell,m}\in(0,1)$ is 
given in Assumption~\nameref{sub_drift_conditions}\ref{sub_drift_conditions(ii)}
and the function $G_h^{-1}$ in~\eqref{eq:Lyapunov_lower_bound} is the inverse of the function $G_h$  defined in~\eqref{eq:G_h}.
\end{prop}

The proof of Proposition~\ref{prop:lyapunov_return_times}, given in Section~\ref{subsec:proofs_return_convergence} below,
follows~\cite{brešar2024subexponential} and rests on certain estimates for discrete-time semimartingales in Section~\ref{subsec:return_times_compact} below. 
Substituting $t$ with $\sqrt{t}$ in~\eqref{eq:Lyapunov_lower_bound}
 yields 
 $\P_x\left(\sum_{j=0}^{S_{(\ell)}} h\circ V(X_j)\geq \sqrt{t}\right) \geq q^2C_{\ell,m}/\Psi(G_{h}^{-1}(\sqrt{t})/(1-q)^2)$, implying 
 the following lower bound on the tail of the double sum for starting points 
$x\in\{\ell+m\leq V\}$ and  all times $t> G_h(\ell+m)^2$:
\begin{align}
   \label{eq:Lyapunov_squared_bound}
\P_x\left(\sum_{j,i=0}^{S_{(\ell)}} h\circ V(X_j)\cdot h\circ V(X_i)\geq t\right) \geq& q^2\frac{C_{\ell,m}}{\Psi(G_{h}^{-1}(\sqrt{t})/(1-q)^2)}.
\end{align}
The inequality in~\eqref{eq:Lyapunov_squared_bound} is crucial in establishing that the CLT fails in the proof of Theorem~\ref{thm:CLT}. 

As we are interested in estimating probabilities of events via ergodic averages of sampling algorithms,
in the context of the present paper, it is natural to consider discontinuous functions $h$ in~\eqref{eq:Lyapunov_lower_bound}
 (e.g. an indicator function $h$ of half-line). This is the reason behind the introduction of an additional parameter $\ell_0'\in[\ell_0,\infty)$ in Proposition~\ref{prop:lyapunov_return_times}, not present in~\cite{brešar2024subexponential}.

A non-empty measurable set $B\in\cB(\cX)$ is  \textit{petite} (for the Markov chain $X$) if there exist a probability measure $a$ on $\cB(\N\setminus\{0\})$ and a finite measure $\nu_a$ on $\cB(\cX)$ with  $\nu_a(\cX)>0$, satisfying
\begin{equation}
\label{eq:petite}
\sum_{j=1}^\infty \P_x(X_j\in \cdot )a(j) \geq \nu_a(\cdot) \quad \text{for all $x\in B$}.
\end{equation}

The following lemma shows that, under Assumptions~\nameref{sub_drift_conditions}, every petite set for $X$ belongs to a sublevel set of the Lyapunov function $V$, making it bounded. The proof of Lemma~\ref{lem:bounded_petite_sets}, given in Section~\ref{subsec:proofs_return_convergence} below, follows the ideas in~\cite{brešar2024subexponential}. 

\begin{lem}[Under $\mathbf{L}$-drift condition, petite sets are bounded]
\label{lem:bounded_petite_sets}
Let~\nameref{sub_drift_conditions} hold. Assume that a set $B\in\cB(\cX)$ is petite for the process $X$. Then there exists $r_0\in(1,\infty)$ such that $B\subset \{V\leq r_0\}$.
\end{lem}

A measurable set $D\in\cB(\cX)$ is called an
\textit{$(m,\nu_m)$-small set} if, for $m\in\N\setminus\{0\}$ and a probability measure $\nu_m$ on
$\cB(\cX)$, there exists $\eps\in(0,1)$ such that
\begin{equation}
\label{eq:small}
P^m(x,\cdot)\geq \eps\nu_m(\cdot)
\quad \text{for all $x\in D$.}
\end{equation}
A set is called \textit{small} if it is an $(m,\nu_m)$-small set
for some $m\in\N\setminus\{0\}$ and some probability measure $\nu_m$.
 Under our standing assumptions, by~\cite[Thm~5.2.1]{meynadntweedie}, there exists an $(m,\nu_m)$-small set $D$ for the chain $X$ with the probability measure $\nu_m(\cdot)=\pi(\cdot\cap D)/\pi(D)$ 
supported in $D$.

Recall that for a set $B\in\cB(\cX)$ and $k\in\N$ we denote by $\tau_B(k) = \inf\{n\geq k:X_n\in B\}$ the first return time to $B$ after time $k$.

\begin{prop}
\label{prop:CLT}
Let the Markov chain $X$ satisfy the standard assumptions of
Section~\ref{subsec:General_Assumptions} and the \textbf{L}-drift conditions~\nameref{sub_drift_conditions}. Let
$g:\cX\to\R$ satisfy~\eqref{eq:CLT_necessary} (in particular, implying $\pi(|g|)<\infty$). Let $D$ be a small
set satisfying $\pi(D)>0$ and let $A\in\cB(\cX)$ satisfy
$A\subset D^c$ and
$\pi(A)>0$.
Assume that $g$ is bounded from below on $D$. If there exists
a measurable function $f:\cX\to[0,\infty)$ such that
\begin{equation}
\label{eq:CLT_excursion_domination}
g(x)-\pi(g)\geq f(x),
\qquad x\in D^c,
\end{equation}
and
\begin{equation}
\label{eq:CLT_excursion_condition}
\E_x\left[
\left(
\sum_{k=0}^{\tau_D(1)-1}f(X_k)
\right)^2
\right]=\infty,
\qquad x\in A,
\end{equation}
then the CLT for $S_n(g)$ \textbf{fails}.
\end{prop}

\begin{proof}[Proof of Proposition~\ref{prop:CLT}]
Choose an $(m_0,\nu_{m_0})$-smallness representation of $D$: 
for some $\eps\in(0,1)$,
we have 
$P^{m_0}(x,\cdot)\geq\eps\nu_{m_0}(\cdot)$ for all 
$x\in D$.
The point-wise convergence of the total variation and the dominated convergence theorem imply
$$
\|\nu_{m_0}P^n-\pi\|_{\TV}
\leq
\int_{\cX}\|P^n(x,\cdot)-\pi\|_{\TV}\nu_{m_0}(\ud x)
\to0\qquad\text{as $n\to\infty$.}
$$
Since $\pi(A)>0$, there exists $n\in\N$ such that
$\nu_{m_0}P^n(A)>0$. Put
$m\coloneqq m_0+n$ and 
$\nu_m\coloneqq\nu_{m_0}P^n$.
Then
$P^m(x,\cdot)\geq\eps\nu_m(\cdot)$,
$x\in D$,
and $\nu_m(A)=\nu_{m_0}P^n(A)>0$.

Consider the Nummelin splitting construction associated
with $(m,\nu_m)$. The split chain $\Phi = (\Phi_j)_{j\in\N}$ with state space $\cX\times\{0,1\}$ is defined as follows: at time $jm$, given $X_{jm}=x$,  let $Y_j$ be a Bernoulli random variable in $\{0,1\}$ with success probability
$$
\widetilde\P(Y_j=1\mid X_{jm}=x)
=\eps\1{x\in D}.
$$
If $Y_j=1$, sample $X_{(j+1)m}$ from $\nu_m$. If $Y_j=0$, sample $X_{(j+1)m}$
from
$$
\frac{
P^m(x,\cdot)-\eps\1{x\in D}\nu_m(\cdot)
}{
1-\eps\1{x\in D}
}.
$$
Given the sampled value $X_{(j+1)m}$, sample $Y_{j+1}$ as a Bernoulli variable with success probability equal to $\eps\1{X_{(j+1)m}\in D}$ and define the split chain $\Phi_j\coloneqq (X_{jm},Y_j)$ for $j\in\N$.

Conditional on the two skeleton endpoints $X_{jm}$ and $X_{(j+1)m}$, sample
$X_{jm+1},\ldots,X_{(j+1)m-1}$ according to the corresponding
conditional law of the original $m$-step path of $X$. Thus the first
coordinate  of the chain $\Phi$, augmented by these bridges of $X$, has transition kernel $P$ at all original times. In particular, by construction, the first coordinate of $\Phi$ is a Markov chain on $\cX$ with transition kernel $P^m$.

Clearly, the split chain $\Phi$
has an atom 
$\alpha\coloneqq D\times\{1\}$ as the transition law of $\Phi$ for any starting point in $\alpha$ does not depend on the starting point. 
 Start the construction of $\Phi$ immediately after a transition
from $\alpha$, so that the first coordinate $X_0$ of $\Phi_0$ has law $\nu_m$, and define the return time of $\Phi$ into $\alpha$ by
\begin{equation}
\label{eq:tau*_def}
\tau^\star\coloneqq\inf\{j\in\N:\Phi_j\in\alpha\}.
\end{equation}
Write $\widetilde\P_{\nu_m}$ and $\widetilde\E_{\nu_m}$ for the
corresponding probability and expectation. The assumed positive Harris recurrence of $X$ implies the recurrence of the
split atom for $\Phi$, i.e., $\widetilde \P_a(\tau^\star<\infty)=1$ for all $a\in\alpha$. The following Claim plays a key role in the proof of the proposition.

\smallskip

\noindent\textbf{Claim.} The random variable $N_D$, defined in~\eqref{eq:CLT_small_set_occupation} and describing the occupation time of the small set $D$,
has a finite second moment:
\begin{equation}
\label{eq:CLT_small_set_occupation}
\widetilde\E_{\nu_m}[N_D^2]<\infty,
\qquad\text{where $N_D\coloneqq
\sum_{k=0}^{m(\tau^\star+1)-1}\1{X_k\in D}$.}
\end{equation}

We will establish this Claim at the end of the proof. First we conclude the proof of the proposition using~\eqref{eq:CLT_small_set_occupation}.
Recall that the measurable function $f$ provides a lower bound on $D^c$  in~\eqref{eq:CLT_excursion_domination}. Define 
$$
H\coloneqq
\sum_{k=0}^{m\tau^\star+m-1}
f(X_k)\1{X_k\notin D}.
$$
Since $A\subset D^c$, on $\{X_0\in A\}$, the excursion from $X_0$ to $D$ is contained in
the path of the chain from time $0$ until time  $\tau_D(1) = \inf\{n\geq 1:X_n\in D\}\leq m\tau^*$. Hence, by the assumption in~\eqref{eq:CLT_excursion_condition} and since $f\geq0$, we have
$$
\widetilde\E_{\nu_m}[H^2]
\geq
\int_A
\E_x\left[
\left(
\sum_{k=0}^{\tau_D(1)-1}f(X_k)
\right)^2
\right]\nu_m(\ud x)
=\infty.
$$
Thus,  by~\eqref{eq:CLT_excursion_domination}, we get the inequality
$$
\sum_{k=0}^{m\tau^\star+m-1}
\bigl(g(X_k)-\pi(g)\bigr)
\geq H-bN_D,\quad\text{where $b\coloneqq
\max\left\{-\inf_{x\in D}\{g(x)-\pi(g)\},0\right\}<\infty$,}
$$
implying $H\leq \sum_{k=0}^{m\tau^\star+m-1}
\bigl(g(X_k)-\pi(g)\bigr)
+bN_D$. Squaring this inequality on both sides and using  $0\leq H$ we obtain
$$
H^2
\leq
2\left(
\sum_{k=0}^{m\tau^\star+m-1}
\bigl(g(X_k)-\pi(g)\bigr)
\right)^2
+2b^2N_D^2.
$$
Since $\widetilde\E_{\nu_m}[H^2]=\infty$, by~\eqref{eq:CLT_small_set_occupation} we conclude
\begin{equation}
\label{eq:inf_second_moment}
\widetilde\E_{\nu_m}\left[
\left(
\sum_{k=0}^{m\tau^\star+m-1}
\bigl(g(X_k)-\pi(g)\bigr)
\right)^2
\right]=\infty.
\end{equation}

By~\eqref{eq:CLT_necessary} we have  $g-\pi(g)\in L^2(\pi)$.
Apply \cite[Thm~II.2.3, (i)$\Rightarrow$(iii)]{Chen99}
to the original chain $X$, with the minorization
$P^m(x,\cdot)\geq\varepsilon\nu_m(\cdot)$ for $x\in D$.
In the notation of\cite[Ch.~II]{Chen99}, $\nu^*$ is the split lift of $\nu_m$
and $\tau(0)=\tau^\star$. Consequently, the CLT for $S_n(g)$
would imply
\[
\widetilde{\mathbb E}_{\nu_m}\left[
  \left(
    \sum_{k=0}^{m(\tau^\star+1)-1}
    \bigl(g(X_k)-\pi(g)\bigr)
  \right)^2
\right]<\infty,
\]
contradicting~\eqref{eq:inf_second_moment}. Thus the CLT for $S_n(g)$ fails.
It remains to prove the Claim.

\smallskip

\noindent\underline{Proof of the Claim.}
Define the variables $B_0$ and $B_1$ that count the visits of $X$ in the small set $D$ up to time $\tau^\star$ defined in~\eqref{eq:tau*_def}:
$$
B_0\coloneqq
\sum_{j=0}^{\tau^\star}\1{X_{jm}\in D}\quad\text{and}\quad
B_1\coloneqq
\sum_{j=0}^{\tau^\star}
\sum_{r=1}^{m-1}
\1{X_{jm}\notin D,\ X_{jm+r}\in D}.$$ 
For $r\in\{1,\ldots,m-1\}$ we have 
$\1{X_{jm+r}\in D}\leq \1{X_{jm}\in D}+\1{X_{jm}\notin D,\ X_{jm+r}\in D}$, implying
\begin{align}
\label{eq:CLT_occupation_decomposition}
N_D
&=\sum_{k=0}^{m(\tau^\star+1)-1}\1{X_k\in D}=
\sum_{j=0}^{\tau^\star}
\sum_{r=0}^{m-1}\1{X_{jm+r}\in D}
\leq mB_0+B_1.
\end{align}

The variable $B_0$ is the number of visits of the skeleton chain to
$D$ up to and including the first successful splitting coin toss. Hence
$
\widetilde\P_{\nu_m}(B_0=k)
=\eps(1-\eps)^{k-1}$,
$k\in\N\setminus\{0\}$,
and
\begin{equation}
\label{eq:CLT_geometric_skeleton}
\widetilde\E_{\nu_m}[B_0^2]
=(2-\eps)/\eps^2.
\end{equation}

We next show that every visit counted by $B_1$, although occurring
between two skeleton times, gives a uniformly positive probability
of regeneration within a fixed number of steps. Since
$\nu_mP^j(D)\to\pi(D)>0$ as $j\to\infty$, choose $j_0\in\N$ such that
\begin{equation}
\label{eq:asymp_small}
\nu_mP^j(D)\geq\frac{\pi(D)}2,
\quad \forall j\geq j_0,\quad\text{and pick $K\in\N\setminus\{0\}$ such that
$Km-(m-1)\geq j_0$.}
\end{equation}
For any $j\in\N$ and $1\leq r<m$, %set $k\coloneqq jm+r$ and  
consider the event $$E_{j,r}\coloneqq
\{j<\tau^\star,\ X_{jm}\notin D,\ X_k\in D\},\qquad\text{where $k\coloneqq jm+r$}.$$  Every visit counted by $B_1$ occurs on one of the events $E_{j,r}$.
The choice of $K$ implies that $k+m+(Km-r)=(j+K+1)m$
is a skeleton time satisfying the condition in~\eqref{eq:asymp_small}:
$$
(j+K+1)m\geq Km-r\geq Km-(m-1)\geq j_0.$$
On $E_{j,r}$, we have $X_k\in D$. Hence the small-set inequality
$P^m(X_k,\cdot)\geq\eps\nu_m(\cdot)$ holds on $E_{j,r}$ and, together with~\eqref{eq:asymp_small}, implies
\begin{align*}
\widetilde\P_{\nu_m}\left(
X_{(j+K+1)m}\in D
\,\middle|\,
\widetilde{\cF}_k
\right)=
P^{m+Km-r}(X_k,D)\geq
\eps\nu_mP^{Km-r}(D)
\geq\frac{\eps\pi(D)}{2} \qquad\text{on }E_{j,r},
\end{align*}
where the $\sigma$-algebra $\widetilde \cF_k$ is generated by $X_0,\ldots,X_k$ and $Y_0,\ldots, Y_{\lfloor k/m\rfloor}$.
Recall that, conditional on reaching $D$ at the skeleton time $(j+K+1)$, the splitting
coin succeeds with probability $\eps$ and the inclusion $\{X_{(j+K+1)m}\in D, Y_{j+K+1}=1\}\subset \{\tau^\star\leq j+K+1\}$ holds. This implies
$$\widetilde\P_{\nu_m}\left(
\tau^\star\leq j+K+1\,\middle|\,
\widetilde{\cF}_k
\right)\geq \eps \widetilde\P_{\nu_m}\left(
X_{(j+K+1)m}\in D
\,\middle|\,
\widetilde{\cF}_k
\right) \geq p\coloneqq \frac{\eps^2\pi(D)}{2}
\quad\text{on }E_{j,r}.
$$
Since
$(j+K+2)m\leq L+k$ (recall $k=jm+r\geq jm$), where $L\coloneqq (K+2)m$,
%$
%(j+K+2)m-k=(K+2)m-r\leq L\coloneqq (K+2)m$,
 we have
\begin{equation}
\label{eq:CLT_uniform_regeneration_probability}
\widetilde\P_{\nu_m}\left(
m(\tau^\star+1)\leq k+L
\,\middle|\,
\widetilde{\cF}_k
\right)\geq \widetilde\P_{\nu_m}\left(
\tau^\star\leq j+K+1\,\middle|\,
\widetilde{\cF}_k
\right)
\geq p>0
\qquad\text{on }E_{j,r}.
\end{equation}

Let $\sigma_1$ be the first visit counted by $B_1$:
$$
\sigma_1\coloneqq\inf\{jm+r:\text{ $j\in\N$ and $1\leq r<m$ and }E_{j,r}\text{ occurs}\} \qquad\text{(with $\inf\emptyset=\infty$).}
$$
Define recursively
$
\sigma_{\ell+1}
\coloneqq
\inf\{jm+r>\sigma_\ell+L:\text{ $j\in\N$ and $1\leq r<m$ and }E_{j,r}\text{ occurs}\}$. 
Note that every visit of $X$ in $D$ counted by $B_1$ occurs before time $m(\tau^\star+1)$, implying that 
if $m(\tau^\star+1)\leq\sigma_\ell+L$ then $\sigma_{\ell+1}=\infty$. 
Put differently, we have 
$\{m(\tau^\star+1)\leq\sigma_\ell+L\}\subset\{\sigma_{\ell+1}=\infty\}$.
Hence, by
\eqref{eq:CLT_uniform_regeneration_probability} applied at the stopping time $\sigma_\ell$, we have 
\begin{align*}
\widetilde\P_{\nu_m}\left(
\sigma_{\ell+1}<\infty
\,\middle|\,
\widetilde{\cF}_{\sigma_\ell}
\right) & =
1-\widetilde\P_{\nu_m}\left(
\sigma_{\ell+1}=\infty
\,\middle|\,
\widetilde{\cF}_{\sigma_\ell}
\right)\\
&\leq  1- \widetilde\P_{\nu_m}\left(
m(\tau^\star+1)\leq \sigma_\ell+L
\,\middle|\,
\widetilde{\cF}_{\sigma_\ell}
\right)
\leq 1-p
\qquad\text{on }\{\sigma_\ell<\infty\}.
\end{align*}

Set $M\coloneqq\sup\{\ell\geq1:\sigma_\ell<\infty\}$
with $M\coloneqq 0$ if $\sigma_1=\infty$ and note that 
$\{M\geq\ell\}=\{\sigma_\ell<\infty\}$ for all $\ell\geq1$.
Hence
\begin{align*}
\widetilde\P_{\nu_m}(M\geq\ell+1)
& = \widetilde\E_{\nu_m}\left[ \widetilde\P_{\nu_m}\left(
\sigma_{\ell+1}<\infty
\,\middle|\,
\widetilde{\cF}_{\sigma_\ell}
\right) \1{\sigma_\ell<\infty}\right] \\
& \leq (1-p)\widetilde\P_{\nu_m}(\sigma_\ell<\infty) =
(1-p)\widetilde\P_{\nu_m}(M\geq\ell),
\end{align*}
implying
$
\widetilde\P_{\nu_m}(M\geq\ell)
\leq(1-p)^{\ell-1}$ for $\ell\geq1$. 
Moreover,
\begin{equation}
\label{eq:M2_bound}
\widetilde\E_{\nu_m}[M^2] = \sum_{\ell=1}^\infty (2\ell-1) \widetilde\P_{\nu_m}(M\geq\ell)
\leq
\sum_{\ell=1}^{\infty}(2\ell-1)(1-p)^{\ell-1}
=\frac{2-p}{p^2}<\infty.
\end{equation}

Recall that $L=(K+2)m$. Note that the time of every visit counted by $B_1$ lies in an interval $[\sigma_\ell,\sigma_\ell+L]$ for some $\ell\in\{1,\ldots,M\}$. Indeed, for any time $k$ counted in $B_1$, there is the greatest $\sigma_\ell\leq k$, since $\sigma_1$ is the time of the first counted visit in $B_1$. If $\sigma_\ell+L<k$, then $\sigma_{\ell+1}$ would also be smaller or equal to $k$, contradicting the choice of $\sigma_\ell$.
All the visits counted in $B_1$ are thus in the union of intervals $\cup_{\ell=1}^M[\sigma_\ell,\sigma_\ell+L]$, which contains $M(L+1)$ integers. Hence $B_1\leq M(L+1)$, implying by~\eqref{eq:M2_bound} the inequality 
$$
 \widetilde\E_{\nu_m}[B_1^2]
\leq
(L+1)^2\frac{2-p}{p^2}.
$$
Combining this with
\eqref{eq:CLT_occupation_decomposition}
and~\eqref{eq:CLT_geometric_skeleton} proves the claim:
$$
\widetilde\E_{\nu_m}[N_D^2]
\leq
2m^2\frac{2-\eps}{\eps^2}
+
2(L+1)^2\frac{2-p}{p^2}
<\infty.
$$
\end{proof}

\begin{proof}[Proof of Theorem~\ref{thm:CLT}]
Fix $q\in(0,1)$ for which~\eqref{eq:CLT_integral_test} holds.
By~\cite[Thm~5.2.1]{meynadntweedie}, choose a small set
$D_0$ satisfying $\pi(D_0)>0$. Every small set is petite: if
$D_0$ is $(m,\nu_m)$-small, use in~\eqref{eq:petite} the
probability measure on $\N$ concentrated at $m$. Hence
Lemma~\ref{lem:bounded_petite_sets} gives $r_0\in(1,\infty)$ such
that
$$
D_0\subset\{V\leq r_0\}.
$$

Choose $\ell>\max\{r_0,\ell_0\}$ sufficiently large that $h$ is continuous,
positive and non-decreasing on $[\ell,\infty)$,
 and
\begin{equation}
\label{eq:centered_tail_lower_bound}
g(x)-\pi(g)\geq h\circ V(x),
\qquad x\in\{V>\ell\}.
\end{equation}
Set
$$
D\coloneqq\{V\leq\ell\}.
$$
The set $D$ is compact, has positive $\pi$-measure because it
contains $D_0$, and is small by the assumption of the theorem.
Let $\widehat h(r)\coloneqq h(r)\1{r\geq\ell+1}.$
The functions $G_{\widehat h}$ and $G_h$ agree on
$[\ell+1,\infty)$, so replacing $h$ by $\widehat h$ does not
change~\eqref{eq:CLT_integral_test}. Let
$A\coloneqq\{V\geq\ell+3\}$.
Proposition~\ref{prop:lyapunov_return_times}, applied with
$\ell_0'=\ell+1$, return level $\ell+2$ and gap parameter $1$,
gives, for every $x\in A$ and all sufficiently large $t$,
\begin{equation}
\label{eq:CLT_original_excursion_tail}
\P_x\left(
\left(
\sum_{j=0}^{S_{(\ell+2)}}
\widehat h\circ V(X_j)
\right)^2\geq t
\right)
\geq
\frac{q^2C_{\ell+2,1}}
{\Psi(G_h^{-1}(\sqrt t)/(1-q)^2)}.
\end{equation}
Integrating
\eqref{eq:CLT_original_excursion_tail} and using
\eqref{eq:CLT_integral_test} yields
\begin{equation}
\label{eq:infinite_excursion_second_moment}
\E_x\left[
\left(
\sum_{j=0}^{S_{(\ell+2)}}
\widehat h\circ V(X_j)
\right)^2
\right]=\infty,
\qquad x\in A.
\end{equation}

The non-confinement assumption and stationarity imply $\pi(A)>0$.
Indeed, if $\pi(A)=0$, a chain started from $\pi$ would almost
surely never visit $A$, contradicting non-confinement.

Set $f\coloneqq\widehat h\circ V$. Since $D$ is compact and $g$
is locally bounded, $g-\pi(g)$ is bounded from below on $D$.
Moreover,~\eqref{eq:centered_tail_lower_bound} gives
$$
g(x)-\pi(g)\geq f(x),
\qquad x\in D^c.
$$
For $x\in A$, we have $S_{(\ell+2)}\leq\tau_D(1)$. Since $f\geq0$
and $f=0$ on $D$,
$$
\left(\sum_{k=0}^{\tau_D(1)-1}f(X_k)\right)^2
\geq
\left(\sum_{k=0}^{S_{(\ell+2)}}\widehat h\circ V(X_k)\right)^2.
$$
It follows from~\eqref{eq:infinite_excursion_second_moment} that
$$
\E_x\left[
\left(
\sum_{k=0}^{\tau_D(1)-1}f(X_k)
\right)^2
\right]=\infty,
\qquad x\in A.
$$
Proposition~\ref{prop:CLT} now concludes the proof.
\end{proof}

The proof of Theorem~\ref{thm:modulated_moments} rests on an application of Proposition~\ref{prop:lyapunov_return_times}.

\begin{proof}[Proof of Theorem~\ref{thm:modulated_moments}]
Let $\ell_0\in[1,\infty)$ be as in Assumption~\nameref{sub_drift_conditions}
and pick $\ell_0'\in(\ell_0,\infty)$ and a function $h:\RP\to\RP$, continuous and non-decreasing on $[\ell_0',\infty)$. Then
Proposition~\ref{prop:lyapunov_return_times} implies that 
the inequality in~\eqref{eq:Lyapunov_lower_bound} holds with level $\ell+m/2 \geq \ell_0'$, gap parameter $m/2\in(0,\infty)$, $q\in(0,1)$, all $r\in(G_h(\ell+m),\infty)$, where $G_h$ is given in~\eqref{eq:G_h}, and some constant $C_{\ell,m}>0$.

Fix a set $B\in\cB(\cX)$, contained in $\{V\leq \ell\}$, and recall  the return time $\tau_B(k)= \inf\{n\geq k:  X_n\in B\}$, for any $k\in\N$, of the chain $X$ to the set $B$.
The Markov property
yields the following lower bound for any starting point $x\in\cX$ and $r\geq0$:
\begin{align}
\nonumber \P_x\left(\sum_{k=0}^{\tau_B(1)-1}h\circ V(X_k) \geq r\right) &\geq \P_x\left(\sum_{k=0}^{\tau_B(1)-1}h\circ V(X_k) \geq r,V(X_1)\geq\ell+m\right) \\
\nonumber &\geq \E_x\left[\1{V(X_1) \geq \ell+m}\cdot \P_{X_1}\left(\sum_{k=0}^{\tau_B(0)-1}h\circ V(X_k) \geq r\right)\right]\\
\label{eq:modulated_moments_bound}&\geq \E_x\left[\1{V(X_1) \geq \ell+m}\cdot \P_{X_1}\left(\sum_{k=0}^{S_{(\ell+m/2)}-1}h\circ V(X_k) \geq r\right)\right].
\end{align}
The inequality in~\eqref{eq:modulated_moments_bound} follows from the fact that, since $B\subset \{V\leq \ell\}$, 
we have $S_{(\ell+m/2)}\leq\tau_B(0)$ $\P_x$-a.s. for all 
$x\in\cX$. 
Note that for any $x\in \{V\geq\ell+m\}$ the inequality in~\eqref{eq:Lyapunov_lower_bound} implies 
\begin{equation}
\label{eq:lower_excursion_modulated_moments}
\P_{x}\left(\sum_{k=0}^{S_{(\ell+m/2)}-1}h\circ V(X_k) \geq r\right) \geq q^2\frac{C_{\ell+m/2,m/2}}{\Psi(G_h^{-1}(r)/(1-q)^2)}, \quad \text{for $r\in(G_h(\ell+m),\infty)$},
\end{equation}
since $C_{\ell+m/2,m/2}\in(0,1)$ and the function $\Psi$ maps into $[1,\infty)$. The inequalities in~\eqref{eq:modulated_moments_bound} and~\eqref{eq:lower_excursion_modulated_moments} imply the lower bound in  Theorem~\ref{thm:modulated_moments}.
%Part~\ref{thm:modulated_moments_b} is a special case of %part~\ref{thm:modulated_moments_a} for the function $h \equiv 1$.
\end{proof}

\subsection{Lower bounds on the tails of the invariant measure}
\label{subsec:lower_bound_invariant}

The following result characterises the failure of integrability with respect to the stationary distribution $\pi$ of $X$.

\begin{prop}
\label{prop:criteria_for_infinite_moments}
Let Assumption~\nameref{sub_drift_conditions} hold. Then for every $q\in(0,1)$ and a  measurable function $h:[1,\infty)\to[1,\infty)$,  non-decreasing and continuous on a neighbourhood of infinity, the following implication holds:
\begin{equation}
\label{eq:invariant_infinite_moment_condition}
 r\mapsto 1/\Psi(G_h^{-1}(r)/(1-q)^2)\notin L^{1}_{ \mathrm{loc}}(\infty) \quad
\implies \quad
\int_{\cX} h \circ V(x)\pi( \ud x) = \infty,
%\quad \text{if}\quad 
\end{equation}
where $G_h$ is the increasing function in~\eqref{eq:G_h}. 
\end{prop}

The proof of Proposition~\ref{prop:criteria_for_infinite_moments} combines the estimates from Theorem~\ref{thm:modulated_moments} with a characterization of the invariant measure in terms of return times (see~\cite[Lem.~11.2.4]{MarkovChains}).
It also requires 
the following technical ingredient: the chain $X$, when started in a sufficiently large sublevel set $\{V\leq \ell\}$, has positive stationary exit flux.

\begin{prop}
\label{prop:escape_sublevel}
Let Assumption~\nameref{sub_drift_conditions} hold. 
Let $\ell\in(0,\infty)$ satisfy $\pi(\{V\leq \ell\})>0$. Then %there exists $ m>0$ such that
$$\int_{\{V\leq\ell\}}P(x,\{ V>\ell\})\pi(\ud x)>0.$$
\end{prop}

\begin{proof}[Proof of Proposition~\ref{prop:escape_sublevel}]
Set \(B:=\{V\leq \ell\}\) and recall $\pi(B)>0$. Suppose, for contradiction, that
$\int_B P(x,B^c)\,\pi(\ud x)=0$.
For the probability measure
$\pi_B(\cdot):=\pi(\cdot\cap B)/\pi(B)$,
this implies that
$(\pi_B P)(B^c)=0$
and hence    
$(\pi_B P)(B)=1$.
Moreover, for every measurable set \(C\subseteq B\), stationarity of \(\pi\) gives
$$
\begin{aligned}
    (\pi_B P)(C)
    &=\frac{1}{\pi(B)}
      \int_B P(x,C)\,\pi(\ud x)\leq
      \frac{1}{\pi(B)}
      \int_{\mathcal X}P(x,C)\,\pi(\ud x)
      =\frac{\pi(C)}{\pi(B)}
      =\pi_B(C).
\end{aligned}
$$

Thus, \(\pi_B P\) and \(\pi_B\) are probability measures supported on \(B\), and
\((\pi_B P)(C)\leq\pi_B(C)\) for every measurable \(C\subseteq B\), implying 
$\pi_B P=\pi_B$.
Consequently, a chain with initial distribution \(\pi_B\) remains in \(B\) at every time:
$\mathbb P_{\pi_B}\bigl(V(X_n)\leq\ell
        \text{ for every }n\in\mathbb N\bigr)=1$, contradicting the non-confinement assumption in~\nameref{sub_drift_conditions}:
$\mathbb P_x\left(\limsup_{n\to\infty}V(X_n)=\infty\right)=1$ for all $x\in\mathcal X$.
Therefore, $\int_B P(x,B^c)\,\pi(\ud x)>0$ as claimed.
\end{proof}

\begin{proof}[Proof of Proposition~\ref{prop:criteria_for_infinite_moments}]
By
the standard assumptions in the paper, stated in the first paragraph of Section~\ref{sec:main_results} (above
in~\nameref{sub_drift_conditions}), the process $X$ is irreducible and recurrent. For every $B\in\cB(\cX)$, such that $\pi(B)>0$,  by~\cite[Lem.~11.2.4]{MarkovChains} we have: for any  measurable function $h:[1,\infty) \to[1,\infty)$,  
\begin{equation}
\label{eq:pi_return_times_criteria}
\int_\cX h\circ V(x)\pi(\ud x) <\infty  \>\implies\>   \int_{B}\E_x\left[\sum_{k=1}^{\tau_B(1)}h\circ V(X_k)\right]\pi(\ud x)<\infty,
\end{equation}
where $\tau_B(k) = \inf\{n\geq k:X_n\in B\}$ for $k\in\N$.

Choose \(\ell'_0\ge\ell_0\) (where $\ell_0$ is given in~\nameref{sub_drift_conditions}) sufficiently large that \(h:[1,\infty)\to[1,\infty)\) is continuous and non-decreasing on \([\ell'_0,\infty)\).
Theorem~\ref{thm:modulated_moments} implies that for  $\ell\geq \ell_0'$, such that $\pi(\{V\leq \ell\})>0$, any $m\in(0,1)$ and any measurable  $B\subset \{V\leq \ell\}$ there exist  $C\in(0,\infty)$ such that 
for all $x\in B$ we obtain
\begin{equation}
\label{eq:pi_return_times_estimate}
\E_x\left[\sum_{k=0}^{\tau_B(1)} h\circ V(X_k)\right]
=
\int_0^\infty \P_x\left(\sum_{k=0}^{\tau_B(1)} h\circ V(X_k)\geq r\right)\ud r
\geq \int_{r_0}^\infty \frac{CP(x,\{V\geq \ell+m\})}{\Psi(G_{h}^{-1}(r)/(1-q)^2)}\ud r
\end{equation}
for  $r_0\coloneq G_h(\ell+m)\in(0,\infty)$, where $G_{h}^{-1}$ is the inverse of the increasing function $G_h$ in~\eqref{eq:G_h}. 

Assume now $\int_\cX h\circ V(x)\pi(\ud x) <\infty$.  Fix $B\coloneqq \{V\leq \ell\}$ for some 
 $\ell\in[\ell_0',\infty)$  satisfying $\pi(B)>0$. Note also that $\int_B h\circ V(x)\pi(\ud x) <\infty$.
Thus, by Proposition~\ref{prop:escape_sublevel} 
we have $\int_BP(x,\{ V>\ell\})\pi(\ud x)>0$.
Hence there exists $m>0$ such that $\int_{B} P(x,\{V\geq \ell+m\})\pi(\ud x)>0$, which, together with the inequality in~\eqref{eq:pi_return_times_estimate}, the assumption $r\mapsto 1/\Psi(G_h^{-1}(r)/(1-q)^2)\notin L^{1}_{ \mathrm{loc}}(\infty)$ and Fubini's theorem implies $$\int_B h\circ V(x)\pi(\ud x)+\int_{B}\E_x\left[\sum_{k=1}^{\tau_B(1)}h\circ V(X_k)\right]\pi(\ud x)=\int_{B}\E_x\left[\sum_{k=0}^{\tau_B(1)}h\circ V(X_k)\right]\pi(\ud x)=\infty.$$
Since $\int_B h\circ V(x)\pi(\ud x) <\infty$, we conclude $\int_{B}\E_x\left[\sum_{k=1}^{\tau_B(1)}h\circ V(X_k)\right]\pi(\ud x)=\infty$.
By the implication in~\eqref{eq:pi_return_times_criteria}, we obtain a contradiction with our assumption $\int_\cX h\circ V(x)\pi(\ud x) <\infty$. This concludes the proof of the proposition. 
\end{proof}

It is well known that a finite moment of a distribution can be translated into an upper bound on its tail decay via Markov's inequality. In contrast, an infinite moment does not imply a  lower bound on the tail decay of that distribution.  However, the infiniteness of expectations for a sufficiently large class of functions does translate into a lower bound on the tail decay of the distribution. 
The following proposition, which is of independent interest, presents such a result for any distribution $\pi$, not necessarily depending on a Markov chain $X$ satisfying 
\textbf{L}-drift conditions~\nameref{sub_drift_conditions}.

\begin{prop}
\label{prop:integral_test}
Let $\pi$ be a probability measure on the space $(\cX,\cB(\cX))$ and $W:\cX\to[1,\infty)$.
Let $A:\RP \to(0,\infty)$ be differentiable, eventually increasing and submultiplicative with $\lim_{r\to\infty} A(r) = \infty$, and let $B:\RP\to(0,\infty)$ be eventually decreasing and continuous with $\lim_{r\to\infty}B(r)= 0$. Moreover, assume that for some $c\in(0,\infty)$ and any eventually increasing continuous $h:\RP\to\RP$ the following implication holds:
\begin{equation}
\label{eq:invariant_infinite_moment_condition_1}
\frac{1}{A\circ (cB_h^{-1})}\notin L^1_\mathrm{loc}(\infty) \implies \int_{\cX} h\circ W(x)\pi(\ud x) = \infty,
\end{equation}
where $B_h^{-1}$ is the inverse of the function $B_h=h/B$ in the neighbourhood of infinity. Then, for any increasing continuous function $d:[1,\infty)\to[1,\infty)$, satisfying $\lim_{r\to\infty}d(r) = \infty$, there exists a constant $c_{\pi,d}\in(0,1)$ such that
\begin{equation}
\label{eq:integral_test_lower_bound}
\pi(\{W\geq r\})\geq \frac{c_{\pi,d}}{A(cr)B(r)d(r)} \quad \text{holds for all $r\in[1,\infty)$.}
\end{equation}
\end{prop}

\begin{proof}
Note that~\eqref{eq:integral_test_lower_bound}
is a consequence of  the following statement, 
\begin{equation*}
%\label{eq:integral_test_lower_bound_1}
\exists r_0\in(0,\infty)\text{ such that, }\quad
1/L_{c,d}(r)\leq\pi(\{W\geq r\})\quad
\text{ for all $r\in[r_0,\infty)$,}
\end{equation*}
where $L_{c,d}(r)\coloneqq A(cr)d(r)B(r)$, which we now establish by contradiction.  

Assume  that \textit{for every} $r\in(0,\infty)$ there exists $r'\in(r,\infty)$ satisfying
$1/L_{c,d}(r')>\pi(\{W\geq r'\})$.
Pick $r_0>1$ sufficiently large so that $A$ is increasing on $[cr_0,\infty)$ and $B$ is decreasing on $[r_0,\infty)$.  Recursively define an increasing sequence $(r_n)_{n\in\N}$, satisfying 
$$
r_{n}>\max\{d^{-1}(\exp(n-1)d(1))r_{n-1},1+r_{n-1}\}\quad\text{and}\quad 1/L_{c,d}(r_n)>\pi(\{W\geq r_n\}),\quad\text{$n\in\N\setminus\{0\}$.}$$ 
In particular, since $d$ is increasing, we have
\begin{equation}
\label{eq:growth_of_r_n}
    d(r_n)>\exp(n-1)\quad\&\quad 
    r_n>\max\{d^{-1}(2d(1))^{n-1}r_0,1+r_{n-1}\}
    \qquad\text{for all $n\in\N\setminus\{0\}$.}
\end{equation}
Using the sequence $(r_n)_{n\in\N}$, we construct a non-decreasing  function $h:[1,\infty)\to[1,\infty)$, 
satisfying 
\begin{equation}
    \label{eq:and}
\int_\cX h \circ W(x)\pi(\ud x)<\infty\qquad
\text{\textbf{and}} \qquad
\frac{1}{A\circ (cB_h^{-1})}\notin L^1_\mathrm{loc}(\infty).
\end{equation}
The proposition then follows from the  implication in assumption~\eqref{eq:invariant_infinite_moment_condition_1}.

Define the function 
$\mu:\RP\to\RP$ by 
$\mu(r)\coloneqq1$ for $r\in[0,r_0)$ and
$\mu(r) \coloneqq \min\{1,1/L_{c,d}(r_n)\}$ for $r\in [r_n,r_{n+1})$, $n\in\N$.  
Since the tail function $r\mapsto \pi(\{W\geq r\})$ is non-increasing, we have  
$\pi(\{W\geq r\})\leq \mu(r)$ for all $r\in[r_1,\infty)$. 
Define  a absolutely continuous non-decreasing function $h:[0,\infty)\to[1,\infty)$ by 
 $h(r) = 1$ 
for $r\in[0,r_1)$. On  $[r_n,r_{n+1})$, for $n\in\N\setminus\{0\}$, set % we define the derivative of $h$ by 
\begin{equation}
\label{eq:def_h_prime}   
h'(r) = \begin{cases}
A(c(r_n+1))d(r_n)^{1/2}B(r),& r\in [r_n,1+r_{n});\\
1/(r_n(r_{n+1}-r_n)),& r\in [1+r_n,r_{n+1}).
\end{cases}
\end{equation}
Since $A$ is differentiable, increasing and submultiplicative (i.e.
$A(c(r_n+1))\leq C_A A(cr_n)A(c)$
for all $r_n\in[1,\infty)$, some constant $C_A\in(0,\infty)$ and $c\in(0,\infty)$ in~\eqref{eq:invariant_infinite_moment_condition_1}), $B$ decreasing and $\mu\leq1$,
we have
$$
h'(r)\mu(r) \leq \begin{cases}
C_AA(c)/d(r_n)^{1/2},&
r\in[r_n,r_n+1); \\
1/(r_n(r_{n+1}-r_n)),& r\in[r_n+1,r_{n+1}).
\end{cases}
$$
The identity
$1+\int_{r_1}^{W(x)}h'(r)\ud r=h(W(x))$ for all $x\in\cX$
and Fubini's theorem
imply the equality
$\int_{\cX} h(W(x))\pi(\ud x)=1+\int_{r_1}^\infty h'(r)\pi(\{W\geq r\})\ud r$.
Recall 
$\pi(\{W\geq r\})\leq \mu(r)$ for $r\geq r_1$
and note
\begin{align}
\label{eq:finite_H_pi}
\nonumber\int_{\cX} h(W(x))\pi(\ud x) &=1+\int_{r_1}^\infty h'(r)\pi(\{W\geq r\})\ud r 
 \leq 1+ \int_{r_1}^\infty h'(r)\mu(r)\ud r\\
&\nonumber=1+\sum_{n=1}^\infty \left(\int_{r_n}^{1+r_n}h'(r)\mu(r)\ud r + \int_{1+r_n}^{r_{n+1}}h'(r)\mu(r)\ud r \right)\\
&\leq1+C_AA(c)\sum_{n=1}^{\infty}d(r_n)^{-1/2} + \sum_{n=1}^{\infty} 1/r_n<\infty,
\end{align}
where the final inequality follows from~\eqref{eq:growth_of_r_n},
which makes both sums in~\eqref{eq:finite_H_pi} clearly finite (for the second sum recall that $d^{-1}(2d(1))>1$).

Recall that the function $r\mapsto B_h(r)=h(r)/B(r)$ is increasing on $[r',\infty)$ for some $r'\in(1,\infty)$ and denote by $B_h^{-1}:[B_h(r'),\infty)\to\RP$ its inverse. By substituting $r = B_h(u)$, for some $n_0\in\N$ we obtain:
\begin{align}
\nonumber
\int_{B_h(r')}^\infty \frac{1}{A(cB_h^{-1}(r))}\ud r & \geq \int_{r'}^{\infty} \frac{(h(u)/B(u))'}{A(cu)}\ud u \geq \sum_{n=n_0}^\infty\int_{r_n}^{r_n+1} \frac{(h(u)/B(u))'}{A(cu)}\ud u\\
&\nonumber\geq\sum_{n=n_0}^\infty \int_{r_n}^{1+r_n}\frac{B(u)A(c(r_n+1))d(r_n)^{1/2}}{A(cu)B(u)}\ud u \\
\label{eq:final_lower_bound_infinite}
&\geq \sum_{n=n_0}^\infty\int_{r_n}^{1+r_n}d(r_n)^{1/2}\ud u =  \sum_{n=n_0}^\infty d(r_n)^{1/2} = \infty.
\end{align}
The first inequality holds by neglecting the non-negative singular measure $\ud B_h^s(u)$ in the derivative $\ud B_h(u)=B_h'(u) \ud u + \ud B_h^s(u)$, which may be present since $B_h$ is not necessarily absolutely continuous.   
Since $B_h'\geq h'/B\geq 0$ on $[r',\infty)$ (recall that  the function $1/B$ is continuous and increasing and hence almost everywhere differentiable with a non-negative derivative) 
the second and third inequalities hold.
%from the  definition of $h'$ given in~\eqref{eq:def_h_prime} above and the fact that the function $1/B$ is continuous and increasing and hence almost everywhere differentiable with a non-negative derivative. 
The fourth inequality in the previous display follows from the fact that $A$ is increasing. The divergence of the sum is a consequence of the first inequality in~\eqref{eq:growth_of_r_n}.

Both conditions in~\eqref{eq:and} hold: the first is given by~\eqref{eq:finite_H_pi} and the second follows from the inequality in~\eqref{eq:final_lower_bound_infinite}. 
This contradiction with assumption~\eqref{eq:invariant_infinite_moment_condition_1} implies the proposition.
\end{proof}

Combining Propositions~\ref{prop:criteria_for_infinite_moments} and~\ref{prop:integral_test}  yields the lower bounds on the tails of the invariant measure $\pi$. This result will be crucial for the study of the tails of the unadjusted (biased) algorithms.

\begin{proof}[Proof of Theorem~\ref{thm:invariant}]
By replace $D$ in~\eqref{eq:def_L_eps_q} with a continuous increasing minorant $d:[1,\infty)\to[1,\infty)$ of $2D$ tending to infinity.
Apply Propositions~\ref{prop:criteria_for_infinite_moments} and~\ref{prop:integral_test} with $W=V$, $\Psi(r)=A(r)$, $c=1/(1-q)^2$ and $B(r) = r\phi(1/r)$.
Note that $A(r)=\Psi(r)\to\infty$ as $r\to\infty$, so Proposition~\ref{prop:integral_test} applies (if $\Psi$ was bounded, Proposition~\ref{prop:criteria_for_infinite_moments} would imply that the constant equal to one, i.e., $h\equiv1$, is not $\pi$-integrable).
\end{proof}

\subsection{Lower bounds on the rates of convergence in \texorpdfstring{$f$}{f}-variation  and the Wasserstein distance}

The following lemma is a simple generalisation of the well known result~\cite[Thm~3.6]{MR2540073}  (see~\cite[Lem.~5.1]{brešar2024subexponential} for a continuous-time version of Lemma~\ref{lem:lower_bound_f_convergence_rate}). Its proof is in Section~\ref{subsec:proofs_return_convergence} below.

\begin{lem}
\label{lem:lower_bound_f_convergence_rate}
Let $X$ be  a Markov chain with an invariant measure $\pi$ on the state space $\cX$. Let functions $H,f,G:\cX \to [1,\infty)$ be such that $G = H/f$, $\pi(f)<\infty$ and~\ref{lem:lower_bound_f_convergence_rate_a} \& \ref{lem:lower_bound_f_convergence_rate_b} hold.
\begin{myenumi}[label=(\alph*)]
    \item \label{lem:lower_bound_f_convergence_rate_a}
There exists a continuous $a:[1,\infty)\to\RP$ such that the function $A(r)\coloneqq ra(r)$ is increasing, $\lim_{r\uparrow\infty}A(r)=\infty$ and $\int_{\{G\geq r\}}f(x)\pi(\ud x)\geq a(r)$ for all $r\in[1,\infty)$.
\item \label{lem:lower_bound_f_convergence_rate_b} There exists a function $v:\cX\times\N\to[1,\infty)$, increasing in the second argument and satisfying  $P^nH(x) \leq v(x,n)$ for all $x\in\cX$ and $n\in\N$. Assume also that $A(1)<2v$ on $\cX\times(\N\setminus\{0\})$.
\end{myenumi}
Then the following bound holds for every $n\in\N\setminus\{0\}$ and $x\in\cX$:
$$
\left(a\circ A^{-1}\circ (2v)\right)(x,n)/2\leq \|\pi(\cdot)-P^n(x,\cdot)\|_{f}.
$$
\end{lem}

The next lemma provides lower bounds on the $L^p$-Wasserstein distance between $\pi$ and $P^n(x,\cdot)$. Its proof is based on the combination of the idea in~\cite[Thm~3.6]{MR2540073} and the Kantorovich-Rubinstein duality~\cite[Prop.~7.29]{Villani09}, see also~\cite[Thm~1.2]{MR4429313} for a continuous time version of Lemma~\ref{lem:lower_bound_wasserstein}. Its proof can be found in Section~\ref{subsec:proofs_return_convergence} below.

\begin{lem}
\label{lem:lower_bound_wasserstein}
  Let $X$ be  a Markov chain with an invariant measure $\pi$ on a metric space $(\cX,d)$, that is  a locally compact Polish length space, and assume that for some $p\in[1,\infty)$, we have $\pi,P^n(x,\cdot)\in\cP_p(\cX)$ for all $x\in\cX$ and $n\in\N$. Let $g:[1,\infty)\to[1,\infty)$ be continuous and increasing, $g(1)=1$ and $\lim_{r\to\infty}g(r)=\infty$. Let the function $W:\cX \to [1,\infty)$ satisfy $|W(x)-W(y)|\leq L_W d(x,y)$ for all $x,y\in\cX$ and some constant $L_W\in(0,\infty)$. Define  $H\coloneqq 2^{-p}W^p\cdot g\circ (2W)$ and assume that~\ref{lem:lower_bound_f_convergence_rate_a} \& \ref{lem:lower_bound_f_convergence_rate_b} hold.
\begin{myenumi}[label=(\alph*)]
    \item \label{lem:lower_bound_Wasserstein_convergence_rate_a}
There exists a continuous $a:[1,\infty)\to\RP$, such that the function $A(r)\coloneqq ra(r)$ is increasing, $\lim_{r\uparrow\infty}A(r)=\infty$ and $\int_{\{g\circ W\geq r\}} (W(x)/2)^p\pi(\ud x)\geq a(r)$ for all $r\in[1,\infty)$.
\item \label{lem:lower_bound_Wasserstein_convergence_rate_b} There exists a function $v:\cX\times\N\to[1,\infty)$, increasing in the second argument and satisfying  $P^n(H)(x) \leq v(x,n)$ for all $x\in\cX$ and $n\in\N$. Assume also that \(A(1)<2^{2p}v\) on $\cX\times(\N\setminus\{0\})$.
\end{myenumi}
Then the following bound holds for every $n\in\N\setminus\{0\}$ and $x\in\cX$:
$$
(a\circ A^{-1}\circ (2^{2p} v))(x,n)^{1/p}/(2L_W)\leq \cW_p(\pi,P^n(x,\cdot)).
$$
\end{lem}

The following lemma facilitates the applications of Lemmas~\ref{lem:lower_bound_f_convergence_rate} and~\ref{lem:lower_bound_wasserstein} in the context of the \textbf{L}-drift conditions~\nameref{sub_drift_conditions}. Its proof is in Section~\ref{subsec:proofs_return_convergence} below.

\begin{lem}
\label{cor:f_invariant_lower_bound}
Let  Assumption~\nameref{sub_drift_conditions} hold. Let $f_\star:[1,\infty)\to(0,\infty)$  be non-decreasing.  Consider an increasing continuous $g:[1,\infty)\to[1,\infty)$, satisfying $\lim_{r\to\infty}g(r)=\infty$ and $g(1)=1$. Then, for every $q\in(0,1)$, the function $L_q$ in~\eqref{eq:def_L_eps_q} and the constant   $c_q\in(0,\infty)$ in~\eqref{eq:main_result_invariant} we have
\begin{equation}
\label{eq:lower_bound_tail_f}
\int_{\{g\circ V\geq r\}} f_\star\circ V(x)\pi(\ud x) \geq c_q f_\star(g^{-1}(r))/L_q(g^{-1}(r))
\quad \text{for all $r\in[1,\infty)$.}
\end{equation}
\end{lem}

We now establish the lower bounds on the convergence in $f$-variation and $L^p$-Wasserstein distance in Theorem~\ref{thm:f_rate}.

\begin{proof}[Proof of the Theorem~\ref{thm:f_rate}]
\noindent \underline{f-variation.} By the assumption in~\ref{assumption:f_convergence} of  Theorem~\ref{thm:f_rate}, a non-decreasing  $f_\star:[1,\infty)\to[1,\infty)$ and continuous  $h,g:[1,\infty)\to[1,\infty)$ satisfy $g = h/f_\star$ on $[1,\infty)$, with $g$ increasing and $\lim_{r\to\infty}g(r)=\infty$. Moreover, there exists a function
$v:\cX\times\N\to[1,\infty)$, increasing in the second argument,
such that 
$P^n(h\circ V)(x)\leq v(x,n)$
for all $x\in\cX$ and $n\in\N$.
By Lemma~\ref{cor:f_invariant_lower_bound}, for any $q\in(0,1)$ 
the inequality in~\eqref{eq:lower_bound_tail_f} holds with the constant $c_q\in(0,\infty)$ from~\eqref{eq:main_result_invariant}.

Any function $a_g:[1,\infty)\to \RP$, satisfying the assumption   in  Theorem~\ref{thm:f_rate} (i.e. the inequality 
$a_g(r)\leq  c_q f_\star(g^{-1}(r))/L_q(g^{-1}(r))$,  $r\in[1,\infty)$, and that the function $A_g(r)= ra_g(r)$ is increasing with $\lim_{r\to\infty} A_g(r)=\infty$), by~\eqref{eq:lower_bound_tail_f} also satisfies Assumption~\ref{lem:lower_bound_f_convergence_rate_a} of Lemma~\ref{lem:lower_bound_f_convergence_rate} with 
$f\coloneqq f_\star\circ V$ and
$G\coloneqq g\circ V$.
As observed in the previous paragraph, the functions 
$H\coloneqq h\circ V$ and $v(x,t)$
satisfy the condition in 
Assumption~\ref{lem:lower_bound_f_convergence_rate_b} of Lemma~\ref{lem:lower_bound_f_convergence_rate}:
$P^nH(x)\leq v(x,n)$ 
for all $x\in\cX$ and $n\in\N$.
%An application of Lemma~\ref{lem:lower_bound_f_convergence_rate} 
%concludes the proof of the theorem in the case of $f$-variation.

We verify the range condition needed to apply Lemma~\ref{lem:lower_bound_f_convergence_rate}.
Since $g(1)=1$ and both $f_\star$ and $g$ are
non-decreasing, the function $h=f_\star g$ is
non-decreasing and $h(1)=f_\star(1)$.
Moreover,~\eqref{eq:main_result_invariant} gives
$c_q/L_q(1)\leq\pi(\{V\geq1\})=1$.
Consequently, for every $x\in\mathcal X$ and $n\geq1$,
\[
A_g(1)=a_g(1)
\leq f_\star(1)\frac{c_q}{L_q(1)}
\leq f_\star(1)=h(1)
\leq P^n(h\circ V)(x)
\leq v(x,n)<2v(x,n).
\]
Since $A_g$ is continuous, strictly increasing and
unbounded, $A_g^{-1}(2v(x,n))$ is well defined and
strictly greater than $1$.
If \(\pi(f)<\infty\), then all the
assumptions of Lemma~\ref{lem:lower_bound_f_convergence_rate}  are satisfied,
and the claimed result follows from that lemma in the case of $f$-variation.

Assume now \(\pi(f)=\infty\). Since \(G\geq1\) and
\(H=fG\), we have
$P^n f(x)\leq P^nH(x)\leq v(x,n)<\infty$.
For \(m\geq1\), put \(f_m\coloneqq f\wedge m\). Since \(0\leq f_m\leq f\),
\[
\left\|\pi-P^n(x,\cdot)\right\|_f
\geq
\pi(f_m)-P^nf_m(x)
\geq
\pi(f_m)-v(x,n).
\]
By monotone convergence, \(\pi(f_m)\uparrow\pi(f)=\infty\) as $m\to\infty$.
Consequently,
$\left\|\pi-P^n(x,\cdot)\right\|_f=\infty$,
so the claimed lower bound holds trivially.

\noindent\underline{$L^p$-Wasserstein distance.} By assumption in~\ref{assumption:Wasserstein_convergence} of Theorem~\ref{thm:f_rate}, continuous  $h,g:[1,\infty)\to[1,\infty)$ satisfy $h(r)= g(2r)f_\star(r)= g(2r)r^p/2^p$, $r\in[1,\infty)$, with $g$ increasing and $\lim_{r\to\infty}g(r)=\infty$. Thus, $H\coloneqq h\circ V=2^{-p}V^p\cdot   g\circ (2V)$, where $V$ is Lipschitz with a constant $L_V$.
Moreover, there exists a function
$v:\cX\times\N\to[1,\infty)$, increasing in the second argument,
such that 
$P^n(h\circ V)(x)\leq v(x,n)$
for all $x\in\cX$ and $n\in\N$.
By Lemma~\ref{cor:f_invariant_lower_bound}, for any $q\in(0,1)$ and the constant $c_q\in(0,\infty)$ in~\eqref{eq:main_result_invariant},  the inequality in~\eqref{eq:lower_bound_tail_f} holds.
%for all $r\in[1,\infty)$.

Any function $a_g:[1,\infty)\to \RP$, satisfying 
the assumption in  Theorem~\ref{thm:f_rate} (i.e. the inequality 
$a_g(r)\leq  c_q (g^{-1}(r)/2)^p/L_q( g^{-1}(r))$, $r\in[1,\infty)$, and that the function $r\mapsto ra_g(r)$ is increasing with $\lim_{r\to\infty} ra_g(r)=\infty$), by~\eqref{eq:lower_bound_tail_f} also satisfies Assumption~\ref{lem:lower_bound_Wasserstein_convergence_rate_a} of Lemma~\ref{lem:lower_bound_wasserstein} with 
$W=V$ and
$g$.
As observed in the previous paragraph, the functions 
$H\coloneqq h\circ V$ and $v(x,t)$
satisfy the condition in 
Assumption~\ref{lem:lower_bound_Wasserstein_convergence_rate_b} of Lemma~\ref{lem:lower_bound_wasserstein}:
$P^nH(x)\leq v(x,n)$ 
for all $x\in\cX$ and $n\in\N$.
We verify the range condition needed to apply Lemma~\ref{lem:lower_bound_wasserstein}.
Since $g(1)=1$, we have $g^{-1}(1)=1$.
Moreover,~\eqref{eq:main_result_invariant}  gives
$c_q/L_q(1)\leq\pi(\{V\geq1\})=1$.
Consequently, for every $x\in\mathcal X$ and $n\geq1$,
\[
A_g(1)=a_g(1)
\leq \left(\frac{g^{-1}(1)}{2}\right)^p
     \frac{c_q}{L_q(g^{-1}(1))}
=2^{-p}\frac{c_q}{L_q(1)}
\leq 2^{-p}
<2^{2p}v(x,n),
\]
where the last inequality follows from $p\geq1$
and $v(x,n)\geq1$.
Since $A_g$ is continuous, strictly increasing and
unbounded, $A_g^{-1}(2^{2p}v(x,n))$ is well defined
and strictly greater than $1$.
An application of Lemma~\ref{lem:lower_bound_wasserstein} 
concludes the proof of the theorem.
\end{proof}

\section{Technical estimates for return times and convergence}
\label{sec:return_times_convergence}
This section contains technical proofs of some propositions and lemmas from Section~\ref{sec:main_proofs}.

\subsection{Return times to compact sets for discrete-time semimartingales}
\label{subsec:return_times_compact}

This section develops a general theory for the analysis of return times of discrete-time semimartingales to bounded sets. The results in this section are discrete-time analogues
of the results in~\cite[Sec.~4]{brešar2024subexponential}.
Throughout this section we fix a probability space $(\Omega,\cF,\P)$ with 
 a filtration $(\cF_n)_{n\in\N}$ and denote $u\wedge v\coloneqq \min\{u,v\}$ for any $u,v\in\R$.
We begin with an elementary maximal inequality. 

\begin{prop}[Maximal inequality]
\label{prop:maximal}
Let $(\cF_n)_{n\in\N}$ be a filtration and $\xi = (\xi_n)_{n\in\N}$ an $(\cF_n)$-adapted process taking values in $[0,1]$. 
Define an $(\cF_n)$-stopping time
$\tau_r\coloneqq\inf\{n\in\N:\xi_n> r\}$ 
(with convention $\inf \emptyset =\infty$)  
and assume that, for some $r\in(0,1]$
and a locally bounded measurable function 
$f:\N\times[0,1] \to \RP$, the process 
 %$W=(W_t)_{t\in\RP}$,
$(\xi_{n\wedge \tau_r} - \sum_{j=0}^{n\wedge \tau_r-1} f(j,\xi_j))_{n\in\N}$
is an $(\cF_n)$-supermartingale. 
Then, for any $t\in \N$, we have $$\P\left(\sup_{j\in[0,t)\cap\N}\xi_j> r\Big\vert\cF_0\right)\leq r^{-1}\left(\xi_0 + \E\left[\sum_{j=0}^{t\wedge \tau_r-1} f(j,\xi_j)\Big\vert\cF_0\right]\right)\quad\text{a.s.}$$
\end{prop}

\begin{proof}
Pick any $t\in\N$ and consider an  $(\cF_n)$-stopping time $\tau_r \wedge t$, bounded above by $t$. Note that $\sup_{0\leq u<\tau_r\wedge t} \xi_u\leq r$ (with convention $\sup\emptyset=-\infty$). Since $f$ is bounded on the compact set $[0,t]\times [0,r]$, there exists a constant $C\in(0,\infty)$ such that $\sup_{0\leq u < \tau_r\wedge t}f(u,\xi_u)\leq C$ a.s. Thus
\begin{equation}\label{eq:maximal_ineq_upper_bound}
0\leq \E\left[\sum_{j=0}^{\tau_r \wedge t-1} f(j,\xi_j) \Big\vert \cF_0\right]\leq Ct<\infty\quad\text{a.s.}
\end{equation}

Since $(\xi_{n\wedge \tau_r} - \sum_{j=0}^{n\wedge \tau_r-1} f(j,\xi_j))_{n\in\N}$
is an $(\cF_n)$-supermartingale, we obtain: $\xi_{\tau_r \wedge t} -\sum_{j=0}^{t\wedge \tau_r-1} f(j,\xi_j)$ is integrable and
$\E[\xi_{\tau_r \wedge t} -\sum_{j=0}^{t\wedge \tau_r-1} f(j,\xi_j)\vert \cF_0] \leq \xi_0$. The inequality in~\eqref{eq:maximal_ineq_upper_bound}
 and the fact that $\xi$ is non-negative imply the following:
\begin{align}
\nonumber
\E[\xi_{\tau_r \wedge t}\vert \cF_0]&=
\E\left[\xi_{\tau_r \wedge t} -\sum_{j=0}^{\tau_r \wedge t-1} f(j,\xi_j)\Big\vert \cF_0\right]+\E\left[\sum_{j=0}^{\tau_r \wedge t-1} f(j,\xi_j) \Big\vert \cF_0\right]
\\
&\leq \xi_0 + \E\left[\sum_{j=0}^{\tau_r \wedge t-1} f(j,\xi_j) \Big\vert \cF_0\right].
\label{eq:maximal_supermart_inequality}
\end{align}
Moreover, by the definition of $\tau_r$ in the proposition we have $\{\sup_{u\in[0,t)\cap\N}\xi_u> r\} = \{\tau_r < t\}$
a.s. Furthermore, on the event $\{\tau_r < t\}$ we have $\xi_{\tau_r\wedge t} = \xi_{\tau_r} \geq r$ a.s. Thus, by~\eqref{eq:maximal_supermart_inequality},
we have
\begin{align*}
\P\left(\sup_{u\in[0,t)\cap\N}\xi_u > r\Big\vert\cF_0\right) &=  \P\left(\tau_r < t\Big\vert\cF_0\right)
\leq r^{-1}\E[\xi_{\tau_r \wedge t}\mathbbm1\{\tau_r <t\}\vert \cF_0]\leq r^{-1}\E[\xi_{\tau_r \wedge t}\vert \cF_0]
\\&\leq r^{-1}\left(\xi_0 + \E\left[\sum_{j=0}^{\tau_r \wedge t-1} f(j,\xi_j)\Big\vert \cF_0\right]\right),
\end{align*}
implying the proposition.
\end{proof}

Proposition~\ref{prop:maximal} will be applied in the proof of Lemma~\ref{lem:return_times}. To state the lemma, consider an $(\cF_n)$-adapted  process $\kappa \coloneqq (\kappa_n)_{n \in \N}$  taking values in $[1,\infty)$.
Let $\cT$ denote the set of all $\N\cup\{\infty\}$-valued stopping times with respect to the filtration $(\cF_n)_{n\in \N}$.
For any $\ell,r \in (1,\infty)$ and a stopping time $T \in \cT$, define the first entry times (after $T$) into the intervals $[1,\ell)$ and $(r,\infty)$ by
\begin{align}
\label{eq::lambda}
    \lambda_{\ell,T} \coloneqq  T + \inf\{s \in \N: T < \infty,~ \kappa_{T+s} < \ell\}, \\
    \rho_{r,T} \coloneqq  T + \inf\{s \in \N: T < \infty, ~\kappa_{T+s} > r\},
    \label{eq::rho}
\end{align}
respectively. If $T = 0$, for simplicity we write $\lambda_\ell \coloneqq \lambda_{\ell,0}$ and $\rho_r \coloneqq \rho_{r,0}$.

%The following Lemma is the central tool of our approach to lower bounds. 

\begin{lem}
\label{lem:return_times}
Let $\kappa = (\kappa_n)_{n\in\N}$ be  $[1,\infty)$-valued $(\cF_n)$-adapted process satisfying $\limsup_{n\to\infty}\kappa_n =\infty$ a.s.
Suppose that there exist a level $\ell\in(1,\infty)$ and a non-decreasing continuous function $\varphi:(0,1]\to(0,\infty)$, such that 
the process 
$$
\left(1/\kappa_{(\rho_{r_q}+t)\wedge \lambda_{r,\rho_{r_q}}} - \sum_{j=\rho_{r_q}}^{(\rho_{r_q} + t) \wedge \lambda_{r,\rho_{r_q}}-1} \varphi(1/\kappa_{j})\right)_{t\in\N}\quad\text{is an $(\cF_{\rho_{r_q}+t})$-supermartingale}
$$
for every $q\in(0,1)$ and $r\in(\ell,\infty)$, where $r_q \coloneqq r/(1-q)^2$. 
Pick any $q\in(0,1)$ and
a function $f:\RP \to \RP$, non-decreasing on $[\ell,\infty)$.
Recall by~\eqref{eq:G_h} the function $G_f(r)=q(1-q)f(r)/(r\varphi(1/r))$. 
Then, for every $r\in(\ell,\infty)$, the following inequality holds:
\begin{align}
\label{eq:return_lower_bound}
\P\left(\sum_{j=0}^{\lambda_{\ell}-1} f(\kappa_j)\geq
G_f(r)
\Big\vert \cF_0\right)
 \geq q^2\P(\rho_{r_q}<\lambda_{\ell}\vert\cF_0)\quad\text{a.s.}
%\\
% \label{eq:return_squared}
% \P\left(\sum_{j,i=0}^{\lambda_{\ell}} f(\kappa_j)f(\kappa_i)\geq (f(r)q(1-q)/(r\varphi(1/r)))^2\Big\vert \cF_0\right)&\geq q\P(\rho_{r_q}<\lambda_{\ell}\vert \cF_0)\quad\text{a.s.}
\end{align}
\end{lem}

\begin{rem}
\noindent (I) The  assumption $\limsup_{n\to\infty}\kappa_n =\infty$ a.s.
in Lemma~\ref{lem:return_times}
implies $\P(\rho_r<\infty)=1$ for all $r\in[1,\infty)$, making $\kappa_{\rho_{r_q}}$ well defined. The main step in the proof of inequality~\eqref{eq:return_lower_bound} in Lemma~\ref{lem:return_times} consists of establishing the following: after reaching the level $r_q$, with probability at least $q^2$, 
the process $\kappa$ spends at least $q(1-q)/(r\varphi(1/r))$ units of time before returning below the level $r$. %The inequality in~\eqref{eq:main_lower_bound} in Corollary~\ref{cor:return_times}
%follows from~\eqref{eq:return_lower_bound}  via %standard submartingale arguments.

\noindent (II) Note that  $f\equiv1$ in Lemma~\ref{lem:return_times} yields a lower bound on the probability  
$\P(
\lambda_{\ell} \geq G_1(r)\vert\cF_0)$.
Moreover, if it also holds $r\varphi(1/r)\to0$ as $r\to\infty$,
then $t=G_1(r)\to\infty$ as $r\to\infty$ and 
the bound in~\eqref{eq:return_lower_bound} 
yields a lower bound on the tail probability 
$\P(\lambda_{\ell} \geq t\vert\cF_0)$
as $t\to\infty$.\\
\noindent (III) 
Lemma~\ref{lem:return_times} does not require a Markovian structure for $\kappa$ and is in particular independent of Assumption~\nameref{sub_drift_conditions}.
\end{rem}

\begin{proof}[Proof of Lemma~\ref{lem:return_times}]
%\underline{\it Inequality~\eqref{eq:return_lower_bound}}
Pick $q\in (0,1)$ and $r\in(\ell,\infty)$.
We start by showing that, once the process $\kappa$ reaches the level
$r_q=r/(1-q)^2$, with probability at least $q^2$ it takes
$q(1-q)/(r\varphi(1/r))$ units of time for $\kappa$ to return to the
interval $[1,r)$. More precisely, we establish the following inequality:
\begin{equation}
\label{eq:r_squared_time_to_come_back}
\P\left(
\lambda_{r,\rho_{r_q}}
\geq
\rho_{r_q}+\frac{q(1-q)}{r\varphi(1/r)}
\middle|
\cF_{\rho_{r_q}}
\right)
\geq q^2
\quad\text{a.s.}
\end{equation}
By the non-confinement assumption $\limsup_{n\to\infty}\kappa_n =\infty$, we have $\rho_{r_q}<\infty$ a.s. Define the process $(\xi_t)_{t\in\N}$ by  $\xi_t \coloneqq 1/\kappa_{\rho_{r_q} + t}$. Note that
$\tau_{1/r}=\inf\{n>0:\xi_n > 1/r\}=\lambda_{r,\rho_{r_q}}-\rho_{r_q}$
by~\eqref{eq::lambda}
and
hence
$t\wedge\tau_{1/r}=((\rho_{r_q}+t)\wedge \lambda_{r,\rho_{r_q}})-\rho_{r_q}$.
Moreover, by the definition of $\tau_{1/r}$, we have 
$$\left\{\sup_{j\in[0,t)\cap\N}\xi_j> 1/r\right\} = \{\tau_{1/r} < t\}\qquad\text{a.s. for any $t\in\N$.}$$
% and let $f(u) \coloneqq Cu^2$ for $u\in\RP$. 
By assumption, the process $(\xi_{t\wedge\tau_{1/r}} - \sum_{j=0}^{t\wedge\tau_{1/r}-1}\varphi(\xi_j))_{t\in\N}$ is an $(\cF_{\rho_{r_q}+t})$-supermartingale.  
Moreover, since $\varphi:(0,1]\to\RP$ is non-decreasing and continuous,
it has a unique extension (via its right-limit at $0$) to a continuous function $\varphi:[0,1]\to\RP$.
Applying Proposition~\ref{prop:maximal} (with a continuous function $f(u,s)=\varphi(s)$) to $\xi$  and the stopping time $\tau_{1/r}$ yields
\begin{align*}
\P\left(
\lambda_{r,\rho_{r_q}}<\rho_{r_q}+t
,\middle|,
\cF_{\rho_{r_q}}
\right)
&=
\P\left(
\tau_{1/r}<t
,\middle|,
\cF_{\rho_{r_q}}
\right)\
=
\P\left(
\sup_{0\leq u<t}\xi_u>1/r
,\middle|,
\cF_{\rho_{r_q}}
\right)\\
&\leq
r\left(
\xi_0+
\E\left[
\sum_{j=0}^{t\wedge\tau_{1/r}-1}\varphi(\xi_j)
,\middle|,
\cF_{\rho_{r_q}}
\right]
\right)\\
& \leq
r\left(
\xi_0+
\varphi(1/r)
\E\left[
t\wedge\tau_{1/r}
,\middle|,
\cF_{\rho_{r_q}}
\right]
\right)\\
&\leq
r\left(
1/r_q+\varphi(1/r)t
\right)\
=
(1-q)^2+r\varphi(1/r)t,
\qquad t\in\N.
\end{align*}
Here, the second inequality follows from the fact that $\varphi$ is
non-decreasing and that $\xi_u\leq1/r$ on
${u<t\wedge\tau_{1/r}}$. The third inequality follows from
$t\wedge\tau_{1/r}\leq t$, while the last equality follows from the
definition of $r_q$.

Taking complements, we obtain
\begin{equation}
\label{eq:integer_return_bound}
\P\left(
\lambda_{r,\rho_{r_q}}\geq\rho_{r_q}+t
,\middle|,
\cF_{\rho_{r_q}}
\right)
\geq
1-\bigl((1-q)^2+r\varphi(1/r)t\bigr),
\qquad t\in\N.
\end{equation}
Set $a\coloneqq\frac{q(1-q)}{r\varphi(1/r)}$. If $a\leq1$, then
$\lambda_{r,\rho_{r_q}}-\rho_{r_q}\geq1\geq a$ a.s.,
since $\kappa_{\rho_{r_q}}>r_q>r$. Hence
\eqref{eq:r_squared_time_to_come_back} holds in this case.
Suppose now that $a>1$ and set $t=\lceil a\rceil$. Since
$\lambda_{r,\rho_{r_q}}-\rho_{r_q}$ is integer-valued, we have 
$$
    \left\{
        \lambda_{r,\rho_{r_q}}
        \geq
        \rho_{r_q}+a
    \right\}
    =
    \left\{
        \lambda_{r,\rho_{r_q}}
        \geq
        \rho_{r_q}+\lceil a\rceil
    \right\}.
$$
Furthermore, $\lceil a\rceil<a+1$, so~\eqref{eq:integer_return_bound}
gives
\begin{align*}
\P\left(
\lambda_{r,\rho_{r_q}}
\geq
\rho_{r_q}+a
,\middle|,
\cF_{\rho_{r_q}}
\right)
&\geq
1-(1-q)^2-r\varphi(1/r)\lceil a\rceil\\
&>
1-(1-q)^2-r\varphi(1/r)(a+1)\
=
q-r\varphi(1/r).
\end{align*}
Since $a>1$ we have 
$r\varphi(1/r)<q(1-q)$ and hence $
    q-r\varphi(1/r)>q-q(1-q)=q^2$, implying~\eqref{eq:r_squared_time_to_come_back}.

Note that on the event $\{\lambda_{r,\rho_{r_q }} \geq \rho_{r_q} + q(1-q) /(r\varphi(1/r))\}$, 
for any  function $f$, non-decreasing on $[\ell,\infty)\supset [r,\infty)$, 
we have $f(\kappa_{\rho_{r_q} + t}) \geq f(r)$ for all $t\in[0,\frac{q(1-q)}{r\varphi(1/r)})\cap\N$. Since $r>\ell$,
on the event 
$\{\rho_{r_q } < \lambda_{\ell}\}$, the inequality 
$\lambda_{\ell}\geq \lambda_{r,\rho_{r_q }}$
holds,
implying 
the following inclusion:
\begin{equation}
\label{eq:key_inclusion}
\left\{\sum_{j=0}^{\lambda_{\ell}-1} f(\kappa_j) \geq f(r) q(1-q) /(r\varphi(1/r))\right\}\supset \{\rho_{r_q } < \lambda_{\ell}\}\cap\{\lambda_{r,\rho_{r_q }} \geq \rho_{r_q } +q(1-q)/ (r\varphi(1/r))\}.
\end{equation}
Indeed, the half-open interval
$\left[
        0,\frac{q(1-q)}{r\varphi(1/r)}
    \right)\cap\N
$
contains
$
    \left\lceil
        \frac{q(1-q)}{r\varphi(1/r)}
    \right\rceil$
integers, and hence the corresponding sum is at least
$q(1-q)f(r)/(r\varphi(1/r))$.

By the inequality in~\eqref{eq:r_squared_time_to_come_back}, we thus obtain the inequality in~\eqref{eq:return_lower_bound}:
\begin{align*}
    \P\Bigg( \sum_{j=0}^{\lambda_{\ell}-1} &f(\kappa_j) \geq f(r)q(1-q) /(r\varphi(1/r))\Big \vert\cF_0\Bigg) \\ %\label{eq:inequality_length_maximum}
    &\geq \E\left[\mathbbm{1}\{\rho_{r_q } < \lambda_{\ell}\}\P\left(\lambda_{r,\rho_{r_q }} \geq \rho_{r_q } + q(1-q)/(r\varphi(1/r))\vert \cF_{\rho_{r_q}}\right)  \big\vert\cF_0\right] 
    \geq q^2\P(\rho_{r_q }<\lambda_{\ell}\vert\cF_0)\quad\text{a.s.}
\end{align*}    
\end{proof}

\subsection{From return times to convergence}
\label{subsec:proofs_return_convergence}

The  proof of Proposition~\ref{prop:lyapunov_return_times} is based on Lemma~\ref{lem:return_times} and requires showing that  Assumption~\nameref{sub_drift_conditions} 
implies the assumptions of Lemma~\ref{lem:return_times} for the process $\kappa\coloneqq V(X)$.

\begin{proof}[Proof of Proposition~\ref{prop:lyapunov_return_times}]
%The proof relies on applying Lemma~\ref{lem:return_times} to the process $\kappa = V(X)$. 
Consider the process $\kappa =V(X)$. 
Recall the definition of the return time $\lambda_{\ell,T}$ and the first-passage time $\rho_{r,T}$ after $T$ (where $\ell,r\in(0,\infty)$ and $T$ an $(\cF_t)$-stopping time) for the process $\kappa$ in~\eqref{eq::lambda} and~\eqref{eq::rho}, respectively. As in Section~\ref{subsec:return_times_compact}, we  denote $\lambda_\ell=\lambda_{\ell,0}$
and 
$\rho_r=\rho_{r,0}$.

By the non-confinement property in Assumption~\nameref{sub_drift_conditions} we have  
$\P_x(\limsup_{n\to\infty}\kappa_n=\infty)=1$
for every $x\in\cX$.
Thus
we obtain $S_{(\ell)}=\lambda_\ell$ and 
$T^{(r)}=\rho_r<\infty$.
Moreover, by~\nameref{sub_drift_conditions}\ref{sub_drift_conditions(i)}, there exists $\ell_0\in[1,\infty)$ such that, for every $r\in(\ell_0,\infty)$ and any $q\in(0,1)$, the process 
$$
\left(1/\kappa_{(\rho_{r_q}+n)\wedge \lambda_{r,\rho_{r_q}}} - \sum_{j=\rho_{r_q}}^{(\rho_{r_q} + n) \wedge \lambda_{r,\rho_{r_q}}-1} \varphi(1/\kappa_{j})\right)_{n\in\N}
$$
is an $(\cF_{\rho_{r_q}+n})$-supermartingale under $\P_x$ for every $x\in\cX$, where $r_q = r/(1-q)^2$ and $\varphi$ is increasing and continuous.
Thus, by inequality~\eqref{eq:return_lower_bound}  in Lemma~\ref{lem:return_times},  for any $\ell_0'\in[\ell_0,\infty)$, $\ell\in(\ell_0',\infty)$ and $r\in(\ell,\infty)$, any $q\in(0,1)$ and  function $h:\RP\to\RP$, non-decreasing on $[\ell_0',\infty)$, 
we have
\begin{align}
\label{eq:Markov_lower_bound}
\P_x\left(\sum_{k=0}^{\lambda_\ell-1} h(\kappa_k) \geq G_h(r)\right)\geq q^2\P_x(\rho_{r_q}<\lambda_\ell)\quad\text{for all $x\in\cX$,}
\end{align}
where, by~\eqref{eq:G_h},  $G_h(r)=q(1-q)h(r)/(r\varphi(1/r))$.
Recall that by~\nameref{sub_drift_conditions}\ref{sub_drift_conditions(i)} the function $G_h$ is continuous and
increasing  on $[\ell_0',\infty)$
(and thus 
invertible on $(\ell_0',\infty)$), with inverse $G_h^{-1}$ defined on $(G_h(\ell_0'),\infty)$.
For any $m\in(0,\infty)$ and $t>G_h(\ell+m)$, set
$r\coloneqq G_h^{-1}(t)>\ell+m$ and note  $r_q=G_h^{-1}(t)/(1-q)^2>\ell+m$. 
The inequality in~\eqref{eq:Lyapunov_lower_bound} 
follows from~\eqref{eq:Markov_lower_bound}
and inequality~\eqref{eq:assumption_heavy_tail_bound} in~\nameref{sub_drift_conditions}\ref{sub_drift_conditions(ii)} for all 
$x\in\{\ell+m\leq V\}$, since on the subset $x\in\{r_q< V\}$
we have 
$\P_x(\rho_{r_q}<\lambda_\ell)=\P_x(T^{(r_q)}<S_{(\ell)})=1$ by definition.
\end{proof}

The bound in~\eqref{eq:Lyapunov_lower_bound} of Proposition~\ref{prop:lyapunov_return_times} plays a key role in the proof of Lemma~\ref{lem:bounded_petite_sets}, which states that under the \textbf{L}-drift conditions~\nameref{sub_drift_conditions} all petite sets are bounded.

\begin{proof}[Proof of Lemma~\ref{lem:bounded_petite_sets}]
Let $B$ be an arbitrary petite set with a probability measure $a$ on $\cB(\N\setminus\{0\})$ and a non-zero measure $\nu_a$ on $\cB(\cX)$ such that~\eqref{eq:petite} holds. Since $\cup_{\ell=1}^\infty \{V<\ell\} = \cX$ and, by~\eqref{eq:petite}, $1\geq \nu_a(\cX)>0$, there exists $\ell_1\in(\ell_0,\infty)$ (where $\ell_0$ is given in Assumption~\nameref{sub_drift_conditions}) such that $c \coloneqq\nu_a(\{V<\ell_1\})\in(0,1]$. 
%Recall that $\kappa =V(X)$ and the definition of return time $\lambda_\ell$ (for $\ell\in[1,\infty)$) in~\eqref{eq::lambda}. 
By Proposition~\ref{prop:lyapunov_return_times} (with $h\equiv1$), there exist $\ell'\in(\ell_1,\infty)$, such that for every $q\in(0,1)$ and $x\in\cX$ we have
\begin{equation}
\label{eq:return_bound_petite}
\P_x(S_{(\ell')} \geq t) \geq q^2\quad\text{ for all $t\in(G_1(\ell'+1),\infty)$ and 
$x\in\{V>G_1^{-1}(t)/(1-q)^2\}$,}
\end{equation}
 where $G_1^{-1}:(G_1(\ell'+1),\infty)\to(\ell'+1,\infty)$ is the inverse of the function $G_1$ in~\eqref{eq:G_h} (with $h\equiv 1$).

 Pick 
$q\in(0,1)$ such that $q^2>1-c/2$.
 Since $a$ is a probability measure on $\cB(\N\setminus\{0\})$, there exists $t_1\in(G_1(\ell'+1),\infty)$ with $a([t_1,\infty)\cap\N)<c/2$.
Define $r_0\coloneqq G_1^{-1}(t_1)/(1-q)^2$. 
Since $G_1^{-1}$ is increasing, we have
$r_0>G_1^{-1}(t_1)>\ell'+1\geq \ell_1$. Moreover,
since the return times satisfy $S_{(\ell')}\leq S_{(\ell_1)}$,
for any  $x\in \{V> r_0\}$, the inequality in~\eqref{eq:return_bound_petite} yields
$\P_x(S_{(\ell_1)}<t_1)\leq\P_x(S_{(\ell')}<t_1)\leq1-q^2<c/2$.

For  $x\in \{V> r_0\}$ and $j\in[0, t_1)\cap\N$,  we have $\{V(X_j)< \ell_1\}\subset\{S_{(\ell_1)}\leq j\}$ and hence $$\P_x(V(X_j)< \ell_1)\leq \P_x(S_{(\ell_1)}\leq j)\leq \P_x(S_{(\ell_1)}<t_1)<c/2.$$ 
Since $a([t_1,\infty)\cap\N)<c/2$, by~\eqref{eq:petite} the following inequalities hold for all $x\in B
\cap \{V> r_0\}$,
$$
c =\nu_a(\{V<\ell_1\})\leq\sum_{j=1}^\infty \P_x(V(X_j)< \ell_1)a(j) \leq \sum_{1\leq j<t_1}\P_x(V(X_j)< \ell_1)a(j) + a([t_1,\infty)) <c,
$$ 
implying $B\cap\{V> r_0\}=\emptyset$. Put differently, $B\subset\{V\leq r_0\}$ and the lemma follows.
\end{proof}

The proofs of Lemmas~\ref{lem:lower_bound_f_convergence_rate} and~\ref{lem:lower_bound_wasserstein} rest on the following idea from~\cite{MR2540073}: the comparison a the tails of the invariant measure and the marginal law of the process yields a good lower bound for the convergence rate. The continuous time analogues of   Lemmas~\ref{lem:lower_bound_f_convergence_rate} and~\ref{lem:lower_bound_wasserstein} can be found in~\cite{brešar2024subexponential} and~\cite{MR4429313}, respectively.

\begin{proof}[Proof of Lemma~\ref{lem:lower_bound_f_convergence_rate}]
It follows from the definition of $f$-variation distance and Markov inequality that, for every $n\in\N$ and every $r\geq 1$, one has the lower bound
\begin{align*}
\|\pi(\cdot)-P^n(x,\cdot)\|_{f} &\geq \int_{\{G\geq r\}}f(x)\pi(\ud x) - \E_{x}[f(X_n)\mathbbm{1}\{G(X_n)\geq r\}] \\
&\geq a(r) - \frac{1}{r}\E_{x}[f(X_n)G(X_n)\mathbbm{1}\{G(X_n)\geq r\}] \\
&\geq a(r) - \frac{1}{r}\E_{x}[H(X_n)] \geq a(r) - \frac{v(x,n)}{r}.
\end{align*}
Since \(A(r)=ra(r)\) is continuous, strictly increasing and unbounded, and \(A(1)<2v(x,n)\), there exists a unique solution \(r=r(n)>1\) 
%satisfying
%Let $r=r(n)$ be the unique solution 
to the equation $ra(r) = 2v(x,n)$ for every $n\in\N$ and $x\in\cX$. Put differently we have 
$r(n)=A^{-1}(2v(x,n))$ for all $n\in \N$ and $v(x,n)/r(n)=a(r(n))/2$.
Thus we obtain
$$
a(r(n))-v(x,n)/r(n) = a(r(n))/2=a(A^{-1}(2v(x,n)))/2,
$$
which, combined with the previous display, concludes the proof.
\end{proof}

\begin{proof}[Proof of Lemma~\ref{lem:lower_bound_wasserstein}]
For any $s\in(1,\infty)$, define the Lipschitz function  $f_{s}:\cX\to \RP$ by 
$$
f_{s}(x) \coloneqq (W(x)-g^{-1}(s)/2)\1{W(x)\geq g^{-1}(s)/2},\qquad \text{ $x\in\cX$.}
$$
If $W(x)\geq g^{-1}(s)$, then $f_s(x)\geq W(x)/2$, implying 
\begin{equation}
\label{eq:Wasserstein_bound}
\int_{\cX}f_{s}(x)^p\pi(\ud x)\geq \int_{\{g\circ W\geq s\}} (W(x)/2)^p\pi(\ud x) \geq a(s).
\end{equation}
Recall $H=2^{-p}W^p\cdot g\circ (2W)$. By Markov's inequality and Assumption~\ref{lem:lower_bound_Wasserstein_convergence_rate_b}  in Lemma~\ref{lem:lower_bound_wasserstein}, we get
\begin{align}
\nonumber \int_\cX f_s(y)^p P^n (x,\ud y) &= P^n((W-g^{-1}(s)/2)^p\1{g\circ (2W)\geq s})(x)\leq  P^n(W^p\1{g\circ (2W)\geq s})(x)\\
& \leq \frac{1}{s}P^n(W^p \cdot g\circ (2 W))(x)
=\frac{2^p}{s}P^n(H)(x)\leq \frac{2^pv(x,n)}{s}.
\label{eq:Wasserstein_growth} 
\end{align}

By the Kantorovich-Rubinstein theorem~\cite[Prop.~7.29]{Villani09} we have
\begin{equation}
\label{eq:Kantorovich-Rubinstein}
\cW_p(\mu_1,\mu_2)\geq \left |\left(\int_{\cX} |f(x)|^p\mu_1(\ud x)\right)^{1/p}-\left(\int_{\cX} |f(x)|^p\mu_2(\ud x)\right)^{1/p}\right|/L_f,
\end{equation}
for all $\mu_1,\mu_2\in \cP_p(\cX)$ and any Lipschitz function $f:\cX\to\R$ with Lipschitz constant $L_f$.
The inequalities in~\eqref{eq:Wasserstein_bound} and~\eqref{eq:Wasserstein_growth}, together with the bound in~\eqref{eq:Kantorovich-Rubinstein} applied to $f_s$, imply
$$
\cW_p(\pi,P^n(x,\cdot))\geq  \left(a(s)^{1/p}-\left(\frac{2^p v(x,n)}{s}\right)^{1/p}\right)/L_W
$$

Since \(A(r)=ra(r)\) is continuous, strictly increasing and unbounded, and \(A(1)<2^{2p}v(x,n)\), there exists a unique solution  $s=s(n)>1$  of the equation $s(n)a(s(n))=2^{2p}v(x,n)$ for all $x\in\cX$ and $n\in\N$. It follows that $s(n) = A^{-1}(2^{2p} v)(x,n)$ for all $n\in\N$ and $(2^p v(x,n)/s(n))^{1/p} = a(s(n))^{1/p}/2$.
This yields
$$
 a(s(n))^{1/p}-\left(\frac{2^p v(x,n)}{s(n)}\right)^{1/p}= a(s(n))^{1/p}/2\geq  \left(a\circ A^{-1} \circ (2^{2p}v)\right)(x,n)^{1/p}/2,
$$
which completes the proof.
\end{proof}

\begin{proof}[Proof of Lemma~\ref{cor:f_invariant_lower_bound}]
Recall that $g:[1,\infty)\to[1,\infty)$ is continuous, increasing with $\lim_{r\to\infty}g(r)=\infty$ and $g(1)=1$.
Pick $q\in(0,1)$. Then, by Theorem~\ref{thm:invariant}, there exists a constant $c_q\in(0,\infty)$ such that 
$$\pi(\{V\geq g^{-1}(r)\})\geq c_q/L_q(g^{-1}(r))\quad\text{for all $r\in[1,\infty)$.}
$$
Using this inequality and the fact that $f_\star$ is non-decreasing, for $r\in[g(1),\infty)$
we obtain  
$$
\int_{\{g\circ V\geq r\}}  f_\star\circ V(x)\pi(\ud x)\geq  f_\star(g^{-1}(r))\pi(\{V\geq g^{-1}(r)\})\geq c_qf_\star(g^{-1}(r))/L_q(g^{-1}(r)),
$$
completing the proof.
% \begin{align*}
% \int_{\{g\circ V\geq r\}}  f_\star\circ V(x)\pi(\ud x) &= \int_{\{V\geq g^{-1}(r)\}} \left(f_\star(1) + \int_1^{V(x)} f_\star'(y)\ud y\right)\pi(\ud x) \\
% &= f_\star(1)\pi(\{V\geq g^{-1}(r)\}) + \int_1^\infty  f_\star'(y) \pi(\{V\geq \max\{g^{-1}(r), y\}\})\ud y\\
% &\geq  f_\star(1)\pi(\{V\geq g^{-1}(r)\}) + \int_1^{g^{-1}(r)} f_\star'(y) \pi({\{V\geq g^{-1}(r)\}}) \ud y\\
% &=  f_\star(g^{-1}(r))\pi(\{V\geq g^{-1}(r)\})\geq c_qf_\star(g^{-1}(r))/L_q(g^{-1}(r)).\qedhere
% \end{align*}
\end{proof}

\subsection{Verifying assumptions of main results}
\label{subsec:verifying}
The non-confinement assumption, which is stated as $\limsup_{n\to\infty}V(X_n)=\infty$ a.s. in the \textbf{L}-drift condition~\nameref{sub_drift_conditions}, is implied by the assumption that for every $u\in\RP$, there exists $n\in\N$ such that $ P^n(x,\{V>u\})>0$ for all $x\in\{V\leq u\}$. The latter assumption holds for all algorithms $X$ studied in this paper, since $\{V\leq u\}$ is compact and the chain $X$ will leave it in a fixed number of steps with positive probability.

% \begin{proof}[Proof of Lemma~\ref{lem:non_confinement}]
% Pick arbitrary $r\in(1,\infty)$.
% It suffices to prove that 
% $\{V(X_n)\geq r\}$ occur infinitely often. 
% By assumption there exist $n_r\in\N$ and $c_r>0$ such that $\inf_{x\in\cX} P^{n_r}(x,\{V\geq r\})>c_r$, implying
% \begin{equation}
% \label{eq:Levy_Borel_Cantelli_Ass}
%     \P_x( V(X_{n_r+n})\geq r|\cF_n)=P^{n_r}(X_n,\{V\geq r\})>c_r\quad\text{a.s. for all $n\in\N$ and $x\in\cX$,}    
% \end{equation}
% where the $\sigma$-algebra $\cF_n$ is generated by $X_0,\ldots,X_n$. By~\eqref{eq:Levy_Borel_Cantelli_Ass}, the events $E_{k}\coloneqq \{V(X_{k n_r})\geq r\}$ and $\sigma$-algebras $\cF_{(k-1)n_r}$ satisfy 
% $\P_x(E_k|\cF_{(k-1)n_r})\geq c_r$ a.s. for all $k\in\N\setminus\{0\}$.
% Thus we obtain $\sum_{k=1}^\infty \P_x(E_k|\cF_{(k-1)n_r})=\infty$ a.s. By
% L\'evy's extension of the Borel-Cantelli lemmas~\cite[Sec.~12.15]{MR1155402} we get $\sum_{k=1}^\infty \1{E_k}=\infty$ almost surely, concluding the proof of the claim.
% \end{proof} 

\begin{proof}[Proof of Lemma~\ref{lem:non_confinement}]
Fix $u\geq1$ and set $K_u:=\{V\leq u\}$. If $K_u\neq\emptyset$,
 define
$\tau_u:=\inf\{n\geq1:X_n\notin K_u\}$.
By assumption, for every $z\in K_u$ there exists
$n_z\geq1$ such that
$p_z:=P^{n_z}(z,K_u^c)>0$.

Since $V$ is continuous, $K_u^c=\{V>u\}$ is open.
The Feller property therefore implies that
$x\mapsto P^{n_z}(x,K_u^c)$ is lower semicontinuous.
Consequently, there exists an open neighbourhood $U_z$
of $z$ such that
$P^{n_z}(x,K_u^c)>p_z/2$
for every $x\in U_z$.
By compactness, finitely many such neighbourhoods
$U_{z_1},\ldots,U_{z_J}$ cover $K_u$. Set
$N_u:=\max_{1\leq i\leq J}n_{z_i}$ and
$\varepsilon_u:=\frac12\min_{1\leq i\leq J}p_{z_i}>0$.
For every $x\in K_u$, choosing $i$ with $x\in U_{z_i}$
gives
$\mathbb P_x(\tau_u\leq N_u)
\geq P^{n_{z_i}}(x,K_u^c)
\geq\varepsilon_u$.
Applying the Markov property successively at times
$N_u,2N_u,\ldots$, we obtain
\[
\sup_{x\in K_u}\mathbb P_x(\tau_u>kN_u)
\leq(1-\varepsilon_u)^k,
\quad k\geq1, \quad\text{and hence $\mathbb P_x(\tau_u=\infty)=0$ for every $x\in K_u$.}
\]

Fix an arbitrary starting state $x\in\mathcal X$.
For every integer $m\geq0$, the Markov property yields
\[
\begin{aligned}
&\mathbb P_x\bigl(\cap_{n\geq m}\{X_n\in K_u\}
                 %\text{ for all }n\geq m
                 \bigr)
%&\qquad=
=\mathbb E_x\left[
  \mathbf1_{\{X_m\in K_u\}}
  \mathbb P_{X_m}(\tau_u=\infty)
\right]=0.
\end{aligned}
\]
Taking the union over $m\geq0$ shows that
$\mathbb P_x(V(X_n)>u\>\text{infinitely often})=1$.
Finally, taking the countable intersection over
positive integer levels $u$ implies
$\mathbb P_x\left(
  \limsup_{n\to\infty}V(X_n)=\infty
\right)=1$.
\end{proof}

Lemma~\ref{lem:overshoot} represents a key technical step in verifying the assumption on the function $\Psi$ in the \textbf{L}-drift condition~\nameref{sub_drift_conditions}. Its proof is elementary and requires controlling the overshoot (via assumption~\eqref{eq:incremental_overshoot})  and applying a submartingale inequality (which holds by assumption~\eqref{eq:sub_mart_ass}).

\begin{proof}[Proof of Lemma~\ref{lem:overshoot}]
Fix $\ell>\ell_0>0$ and $m>0$, and let
$r>\max\{r_0,\ell+m\}$.
We first consider $x$ such that $\ell+m\le V(x)\le r$ (the lemma holds trivially if $V(x)>r$).
For  $i\in\N$ define the events 
$$
B_i\coloneqq \{V(X_i)\leq r,V(X_{i-1})\leq r,\dots,V(X_0)\leq r\} \quad\&\quad A_i \coloneqq \{V(X_i)>r\}\cap B_{i-1},
$$
and set $B_{-1}\in \cF_0$ such that $\P_x(B_{-1})=1$ for all $x\in\cX$.
For any $x\in\cX$ we have
$\{T^{(r)} = i\} = A_i$, $\P_x$-a.s., implying $\sum_{i=0}^\infty\P_x(A_i\cap \{S_{(\ell)}> T^{(r)}\})=\P_x(S_{(\ell)}> T^{(r)})$. For any $k\in\N\setminus\{0\}$ we obtain
% The third equality involves applying Markov property at time $k$, where $\nu_k(\cdot)=\P_x(X_{k-1}\in\cdot|B_{k-1})$ is the law of $X_{k-1}$ on the event $B_{k-1}$. Moreover, the inequality holds from the assumption of the lemma and the last equality holds since $\mathbbm{1}\{X_k>r\}\cdot \mathbbm{1}_{B_{k-1}}  = \mathbbm{1}_{A_k}$,
% \begin{align*}
% \E_x[\Psi\circ V(X_{k})\mathbbm{1}_{A_k}] &= \E_x[\E_x[\Psi\circ V(X_{k})\mathbbm{1}_{A_k}|\cF_{k-1}]]\\
% & =\E_x[\mathbbm{1}_{B_{k-1}} \E_{X_{k-1}}[\Psi\circ V(X_{k})\mathbbm{1}\{X_k>r\}]] \\
% & \leq \E_x[\mathbbm{1}_{B_{k-1}} C\Psi(r) \P_{X_{k-1}}(V(X_k)>r)]\\
% & = C\Psi(r) \E_x[\mathbbm{1}_{B_{k-1}}  \P_x(V(X_k)>r|\cF_{k-1})]=  C\Psi(r)\P_x(A_k).
% \end{align*}
\begin{align}
\nonumber
\E_x[\Psi\circ V(X_{k})\mathbbm{1}_{A_k}\mathbbm{1}\{S_{(\ell)}> T^{(r)}\}] &= \E_x[\E_x[\Psi\circ V(X_{k})\mathbbm{1}_{B_{k-1}}\mathbbm{1}\{V(X_k)>r\}\mathbbm{1}\{S_{(\ell)}> k\}|\cF_{k-1}]]\\
\nonumber
& =\E_x[\mathbbm{1}_{B_{k-1}} \mathbbm{1}\{S_{(\ell)}> k-1\}\E_{X_{k-1}}[\Psi\circ V(X_{1})\mathbbm{1}\{X_1>r\}]] \\
\nonumber
& \leq \E_x[\mathbbm{1}_{B_{k-1}}\mathbbm{1}\{S_{(\ell)}> k-1\} C\Psi(r) \P_{X_{k-1}}(V(X_1)>r)]\\
\nonumber
& = C\Psi(r) \E_x[\mathbbm{1}_{B_{k-1}} \mathbbm{1}\{S_{(\ell)}> k-1\} \P_x(V(X_1)>r|\cF_{k-1})]\\
\nonumber
& = 
C\Psi(r) \E_x[\mathbbm{1}_{B_{k-1}}\mathbbm{1}\{V(X_k)>r\}\mathbbm{1}\{S_{(\ell)}> k\}]\\
\label{eq:final_inequality_first_passage_event}
& =  C\Psi(r)\P_x(A_k\cap \{S_{(\ell)}> k\}),
\end{align}
where the second equality follows from the Markov property (the $\sigma$-algebra $\cF_n$ is generated by $X_0,\ldots,X_n$) and the $\P_x$-a.s. equality $\mathbbm{1}\{X_k>r\}\mathbbm{1}\{S_{(\ell)}> k\}=\mathbbm{1}\{X_k>r\}\mathbbm{1}\{S_{(\ell)}> k-1\}$. The inequality follows from assumption~\eqref{eq:incremental_overshoot} in the lemma since on $B_{k-1}$ we have $V(X_{k-1})\leq r$  $\P_x$-a.s.
% Since for every $x\in\cX$ we have $\limsup_{n\to\infty}V(X_n) = \infty$, $\P_x$-a.s., we have $\P_x(T^{(r)}<\infty)=1$, and thus $\{T^{(r)} = k\} = A_k$, $\P_x$-a.s., implying $\sum_{k=1}^\infty\P_x(A_k\cap \{S_{(\ell)}> T^{(r)}\})=\P_x(S_{(\ell)}> T^{(r)})$. By~\eqref{eq:final_inequality_first_passage_event}, we conclude that
By definition the events $A_k=\{T^{(r)} = k\}$, $k\in\N$, are $\P_x$-a.s. pairwise disjoint, implying:
\begin{align}
\nonumber
\E_x[\Psi\circ V(X_{T^{(r)}})\mathbbm{1}\{S_{(\ell)}> T^{(r)}\}] & = \sum_{k = 1}^{\infty}\E_x[\Psi\circ V(X_{k})\mathbbm{1}_{A_k}\mathbbm{1}\{S_{(\ell)}> T^{(r)}\}] \\
\label{eq:overshoot}
& \leq C\Psi(r)\sum_{k=1}^\infty\P_x(A_k\cap \{S_{(\ell)}> T^{(r)}\})=C\Psi(r) \P_x(S_{(\ell)}> T^{(r)}).
\end{align}

By assumption~\eqref{eq:sub_mart_ass}, picking $\ell\in(\ell_0,\infty)$  makes  the process $\Psi\circ V(X_{\cdot\wedge S_{(\ell)}\wedge T^{(r)}})$ a submartingale. Moreover, by~\eqref{eq:overshoot}, the submartingale $\Psi\circ V(X_{\cdot\wedge S_{(\ell)}\wedge T^{(r)}})$ is uniformly integrable since it is bounded by an integrable variable (recall that we are assuming $V(x)\leq r$):
$$0<\Psi\circ V(X_{n\wedge S_{(\ell)}\wedge T^{(r)}})\leq \Psi(r) + \Psi\circ V(X_{T^{(r)}})\mathbbm{1}\{S_{(\ell)}> T^{(r)}\}\quad\text{for all $n\in\N$.}$$
Since $\Psi\circ V(X_{S_{(\ell)}\wedge T^{(r)}})\mathbbm{1}\{S_{(\ell)}<T^{(r)}\}\leq \Psi(\ell)$ $\P_x$-a.s., the dominated convergence theorem  and  
inequality~\eqref{eq:overshoot} imply
\begin{align*}
\Psi\circ V(x)&\leq \liminf_{n\to\infty}\E_x[\Psi\circ V(X_{n\wedge S_{(\ell)}\wedge T^{(r)}}))]=\E_x[\Psi\circ V(X_{S_{(\ell)}\wedge T^{(r)}})]\\
&\leq \Psi(\ell) + \E_x[\Psi\circ V(X_{T^{(r)}})\mathbbm{1}\{S_{(\ell)}> T^{(r)}\}]\leq \Psi(\ell) + C\Psi(r)\P_x(S_{(\ell)}> T^{(r)}).
\end{align*}
By possibly enlarging the constant $C>0$ in assumption~\eqref{eq:overshoot} and using the fact that $\Psi$ is increasing, for any $m>0$  we have
$$
\P_x(T^{(r)}<S_{(\ell)})\geq (\Psi\circ V(x)-\Psi(\ell))/(C\Psi(r))\geq (\Psi(\ell+m)-\Psi(\ell))/(C\Psi(r))
$$
for any 
 $r\in(\max\{r_0,\ell+m\},\infty)$ and  $x\in\{V\geq \ell+m\}$.
 Setting $C_{\ell,m}\coloneqq (\Psi(\ell+m)-\Psi(\ell))/C$
completes the proof of the lemma.
\end{proof}

\section{Proofs for  ULA and related chains}
\label{sec:examples_proofs}

\subsection{Unadjusted Langevin algorithm}
\label{subsec:ULA_proofs}

Recall that the target density $\pi$ satisfies~\eqref{eq:examples_pi} with $v\in(0,\infty)$ and $\lim_{r\to\infty}r\ell'(r) =0$. and
assume that it is twice continuously differentiable.
The ULA chain $X^{(h)}$, defined in~\eqref{eq:ULA}, follows the Markov kernel $Q_h$ in~\eqref{eq:ULA_kernel} for some fixed $h>0$.

\begin{prop}
\label{prop:drift_ULA}
Fix $h>0$ and let $X^{(h)}$ be an ULA chain. For any $s\in\R$, define  the function 
$V_s(x) \coloneqq (|x|\vee 1)^s$. Then the following statements hold.
\begin{myenumi}[label =(\alph*)]
\item For any $s\in(0,v+2)$, there exist constants $c_1,c_2\in(0,\infty)$ and $r>0$ such that
$$
Q_hV_s - V_s\leq -c_1V_s^{1-2/s} + c_2\1{V_s\leq r} \quad\text{on $\R^d$.}
$$
\item The $\mathbf{L}$-drift condition~\nameref{sub_drift_conditions} is satisfied by the functions $V=V_1$, $\Psi(r) = r^{s}\vee 1$ for any $s\in(v+2,\infty)$, and $\varphi(1/r)=C/r^{3}$ where $C\in(0,\infty)$ is some constant.
\end{myenumi}
\end{prop}

\begin{proof}
For $r\in(0,\infty)$ and $z\in\R^d$, define the ball $B(z,r)\coloneqq\{z'\in\R^d:|z'-z|\leq r\}$. Denote by $y\mapsto p(y-\mu,\sigma^2)$ the density of the normal random vector in $\R^d$ with mean $\mu\in\R^d$ and the covariance matrix $\sigma^2\Id$, where $\sigma^2\in(0,\infty)$ and $\Id\in \R^{d\times d}$ is the identity matrix. For any $s\in\R\setminus\{0\}$ and any function $\phi(y)\propto |y|^2$, it follows that $|2y|^{s}\exp(-\phi(y))\leq \exp(-|y|^{3/2})$ for all $|y|$ sufficiently large. 
Moreover, for  $s\in\R$ we have 
$V_s(x+y)\leq V_{5s}(y)\vee 1$ for all $|y|>|x|^{1/4}/2$ and all $x\in\R^d$ sufficiently large.

Since, by assumption on $\pi$, it holds that $\nabla \log \pi(x)=O(1/|x|)$ as $|x|\to\infty$ (i.e., differently put $\limsup_{|x|\to\infty}|x|\cdot|\nabla \log\pi(x)|<\infty$),  for any $s\in\R$ we obtain
\begin{align}
\nonumber
\int_{B(x,\frac{|x|^{1/4}}{2})^c} V_s(y) p\left(y-(x+h\nabla \log \pi(x)),2h\right)\ud y&=\int_{B(x,\frac{|x|^{1/4}}{2})^c} V_s(y) p\left(y-x+O(1/|x|),2h\right)\ud y\\
\nonumber
&=\int_{B(0,\frac{|x|^{1/4}}{2})^c} V_s(y+x) p\left(y+O(1/|x|),2h\right)\ud y\\
\nonumber&\leq \int_{B(0,\frac{|x|^{1/4}}{2})^c} (V_{5s}(y)\vee 1) p\left(y/2,2h\right)\ud y\\
&\leq \int_{B(0,\frac{|x|^{1/4}}{2})^c}\exp(-|y|^{3/2})\ud y 
\label{eq:ULA_long_jumps}= o(\exp(-|x|^{1/4})),
\end{align}
as $|x|\to\infty$.
Note that for $|x|$ sufficiently large we have
$$
\nabla V_s(x) = sn(x)|x|^{s-1} =V_s(x) sn(x)|x|^{-1} \quad\text{and}\quad \Hessian(V_s)(x) = V_s(x) s|x|^{-2}(\Id + (s-2)n(x)n(x)^\intercal), 
$$
where $\Hessian(f)$ denotes a Hessian matrix of a twice differentiable function $f:\R^d\to\R$, $n(x) = x/|x|$ and $n(x)^\intercal$ its transpose for $x\in \R^d\setminus\{0\}$. For $|x|$ large and $y\in B(0,|x|^{1/4})$, second order Taylor expansion yields
\begin{equation}
\label{eq:ULA_Lyapunov_representation}
\frac{V_s(x+y)-V_s(x)}{V_s(x)}=s|x|^{-1}\langle n(x),y\rangle +|x|^{-2}\frac{s}{2}(|y|^2+(s-2)\langle n(x),y\rangle^2) + o(|x|^{-2})
\end{equation}
as $|x|\to\infty$. 
Moreover, for any $\mu_0\in(0,\infty)$ and $z\in \mathbb{S}^{d-1}$,~\cite[Lem.~A.5]{Roberts07} and an argument analogous to the one in~\eqref{eq:ULA_long_jumps} yield the following as $r\to\infty$:
\begin{align}
\label{eq:ULA_spherical}
\int_{B(0,r)} \langle y,z\rangle^2 p(y,2h) \ud y&= \frac{1}{d}\int_{B(0,r)} |y|^2 p(y,2h)\ud y=\frac{1}{d}\int_{\R^d} |y|^2 p(y,2h)\ud y+o(\exp(-r))\\
\nonumber&=2h+o(\exp(-r)),\quad\text{and similarly}\\
\nonumber \int_{B(0,r)}\langle z,y\rangle p(y-\mu_0 z,2h)\ud y &= \int_{\R^d}\langle z,y\rangle p(y-\mu_0 z,2h)\ud y+o(\exp(-r)) =\mu_0+o(\exp(-r)).
\end{align}
Thus, the identities in~\eqref{eq:ULA_long_jumps},~\eqref{eq:ULA_Lyapunov_representation}  yield the following expression for $(Q_hV_s(x)-V_s(x))/V_s(x)$: 
\begin{align*}
%\frac{Q_hV_s(x)-V_s(x)}{V_s(x)} 
& \int_{B(x,|x|^{\frac{1}{4}})} \frac{V_s(y)-V_s(x)}{V_s(x)}p\left( y-x+x\frac{h(v+d)}{|x|^2}+o(1/|x|),2h\right)\ud y+o(\exp(-|x|^{\frac{1}{4}}))\\
%&=\int_{B(0,|x|^{\frac{1}{4}})} \left(\langle \nabla  \log V_s(x),y\rangle +\frac{s}{2|x|^{2}}(|y|^2+(s-2)\langle n(x),y\rangle^2)\right)p\left( y+xh\frac{v+d}{|x|^2},2h\right)\ud y+ o(|x|^{-2})\\
 &= \int_{B(0,|x|^{\frac{1}{4}})} \left(|x|^{-1}\langle sn(x),y\rangle +\frac{s}{2|x|^{2}}(|y|^2+(s-2)\langle n(x),y\rangle^2)\right)p\left( y+hx\frac{v+d}{|x|^2},2h\right)\ud y+ o(|x|^{-2}).
\end{align*}
Moreover, by~\eqref{eq:ULA_spherical}, we obtain
\begin{align}
\nonumber \frac{Q_hV_s(x)-V_s(x)}{V_s(x)} =& \frac{-sh(v+d)}{|x|^2} +  \int_{B(0,|x|^{\frac{1}{4}})}\frac{s}{2|x|^{2}}(|y|^2+(s-2)\langle n(x),y\rangle^2)p\left( y,2h\right)\ud y + o(|x|^{-2})\\
\nonumber =& \frac{-sh(v+d)+sdh}{|x|^2} + \frac{s(s-2)}{2d|x|^2}\int_{B(0,|x|^{\frac{1}{4}})}|y|^2p\left( y,2h\right)\ud y + o(|x|^{-2})\\
\label{eq:ULA_sub}=& sh(s-2-v)|x|^{-2}(1+o(1)) \quad\text{as $|x|\to\infty$.}
\end{align}

Choosing $s\in(0,v+2)$ in~\eqref{eq:ULA_sub} and noting that $Q_h V_s$ is bounded on compacts yields part~(a) of the proposition. Moreover, letting $s=-1$ yields $Q_h(1/V_1)-1/V_1\leq \varphi(1/V_1)$, where $\varphi(1/r)\coloneqq \tilde C/r^3$ for an appropriate constant $\tilde C\in(0,\infty)$. 

Part~(b) of the proposition will follow (with functions $\varphi$ and $V=V_1$) once we prove that 
$\Psi(r)\coloneqq r^s$, for any $s\in(v+2,\infty)$, satisfies
 the \textbf{L}-drift condition 
\nameref{sub_drift_conditions}\ref{sub_drift_conditions(ii)}. The identity in~\eqref{eq:ULA_sub} implies that there exists $\ell_0\in(0,\infty)$, such that for $s\in(v+2,\infty)$ we have 
$$Q_h(\Psi\circ V_1)(x)-\Psi\circ V_1(x)\geq 0,\quad\text{ for all $x\in\{V_1\geq \ell_0\}$.}$$

\noindent\textbf{Claim.}
Recall the ULA chain $X^{(h)}$ in~\eqref{eq:ULA}. There exist constants $r_0,C_\Psi\in(0,\infty)$ such that
$$
\E_x\left[
\Psi\circ V(X^{(h)}_1)\1{V(X^{(h)}_1)\geq r}
\right]
\leq
C_\Psi\Psi(r)\P_x(V(X^{(h)}_1)\geq r)
$$
for all $r\in[r_0,\infty)$ and $x\in\{V\leq r\}$, where
$V=V_1$ is defined in the statement of the proposition.

This Claim and Lemma~\ref{lem:overshoot} imply that $\Psi$ satisfies
\nameref{sub_drift_conditions}\ref{sub_drift_conditions(ii)}. Note that the non-confinement of $X^{(h)}$, neede for the application of Lemma~\ref{lem:overshoot}, follows from Lemma~\ref{lem:non_confinement} since the sublevel sets of \(V_1\) are compact, the transition kernel $Q_h$ is Feller and there is a positive probability of leaving any fixed ball in a single step. 

\noindent\underline{Proof of Claim.}
Under the assumptions on $\pi$, the function $\nabla\log\pi$ is bounded.
Put
$$
K\coloneqq h\sup_{z\in\R^d}|\nabla\log\pi(z)|<\infty.
$$
Let $Z$ be a centred Gaussian random vector with covariance
$2h\Id$. For $|x|\leq r$, write
$$
X^{(h)}_1=x+h\nabla\log\pi(x)+Z
\qquad\text{and}\qquad
\mu_x\coloneqq x+h\nabla\log\pi(x).
$$
Then $|\mu_x|\leq r+K$.

For $\mu_x\neq0$, set $e_x=\mu_x/|\mu_x|$, and otherwise let $e_x$
be any fixed unit vector. Since
$$
\{\langle Z,e_x\rangle\geq r\}
\subseteq
\{|\mu_x+Z|\geq r\},
$$
the standard one-dimensional Gaussian tail estimate gives constants
$c_1>0$ and $r_1>1$ such that
\begin{equation}
\label{eq:ULA_overshoot_lower_probability}
\P_x(V(X^{(h)}_1)\geq r)
=
\P(|\mu_x+Z|\geq r)
\geq
c_1r^{-1}\exp\left(-\frac{r^2}{4h}\right)
\end{equation}
for all $r\geq r_1$, uniformly over $|x|\leq r$.

We split
\begin{align*}
\E_x\left[
\Psi\circ V(X^{(h)}_1)\1{V(X^{(h)}_1)\geq r}
\right]
&\leq
(3r)^s\P_x(V(X^{(h)}_1)\geq r)\\
&\quad+
\E\left[
|\mu_x+Z|^s\1{|\mu_x+Z|>3r}
\right].
\end{align*}
For $r\geq2K$, on the event $\{|\mu_x+Z|>3r\}$ we have
$
|Z|>3r-|\mu_x|
\geq2r-K
\geq\frac{3r}{2}$.
Moreover, $|\mu_x|\leq r+K\leq3r/2\leq|Z|$, and hence
$|\mu_x+Z|\leq2|Z|$. Therefore,
$$
\E\left[
|\mu_x+Z|^s\1{|\mu_x+Z|>3r}
\right]
\leq
2^s\E\left[
|Z|^s\1{|Z|>3r/2}
\right].
$$
A standard Gaussian tail estimate gives a constant $c_2>0$ such that
$$
\E\left[
|Z|^s\1{|Z|>3r/2}
\right]
\leq
c_2r^{s+d-2}
\exp\left(-\frac{9r^2}{16h}\right)
$$
for all sufficiently large $r$. Combining this estimate with
\eqref{eq:ULA_overshoot_lower_probability}, and increasing $r_0$ if
necessary, yields
$$
\E\left[
|\mu_x+Z|^s\1{|\mu_x+Z|>3r}
\right]
\leq
c_3r^s\P_x(V(X^{(h)}_1)\geq r)
$$
for some $c_3>0$, uniformly over $|x|\leq r$. Thus,
$$
\E_x\left[
\Psi\circ V(X^{(h)}_1)\1{V(X^{(h)}_1)\geq r}
\right]
\leq
(3^s+c_3)r^s\P_x(V(X^{(h)}_1)\geq r)
=
C_\Psi\Psi(r)\P_x(V(X^{(h)}_1)\geq r),
$$
which proves the Claim.
\end{proof}

\begin{proof}[Proof of Theorem~\ref{thm:ULA}]
The assumptions imply 
$\nabla\log\pi(x)=-(v+d)\frac{x}{|x|^2}+o(|x|^{-1})$,
as $|x|\to\infty$. Since \(\pi\) is positive and \(C^2\), the gradient is bounded on compact sets and therefore globally bounded. The Gaussian transition density is continuous and strictly positive. Consequently the chain is Lebesgue-irreducible and aperiodic, and every compact set is one-step small.

Pick \(a\in(2,v+2)\). By Proposition~\ref{prop:drift_ULA}(a) we have 
$Q_hV_a-V_a\le-c_1V_a^{1-2/a}+c_2\1{V_s\leq r}$, where
$V_a(x)=(1\vee|x|)^a$ and $c_1,c_2,r>0$. This supplies a Foster drift to a small set $C$, making the ULA chain $X^{(h)}$ in~\eqref{eq:ULA} positive Harris recurrence with a unique invariant probability measure \(\pi_h\) and total-variation convergence. This verifies our standard assumptions in Section~\ref{subsec:General_Assumptions}. 
Non-confinement of $X^{(h)}$ follows from Lemma~\ref{lem:non_confinement} since the sublevel sets of \(V_1\) are compact, the transition kernel $Q_h$ is Feller and there is a positive probability of leaving any fixed ball in one step.

By Proposition~\ref{prop:drift_ULA}(b), the \textbf{L}-drift conditions
\nameref{sub_drift_conditions} are satisfied by the functions
$V(x)=|x|\vee1$, $\Psi(r)=r^k$, for any $k\in(v+2,\infty)$, and
$\varphi(1/r)=C/r^3$ for some constant $C\in(0,\infty)$.

\smallskip

\noindent \underline{Proof of Part (I)} For part~(Ia), since $\pi_h(g^2)<\infty$, we have
$\pi_h(g)\in\R$, and therefore
$$
g(x)-\pi_h(g)\geq\frac12|x|^s \qquad\text{for all $|x|$ sufficiently large.}
$$
 Applying
Theorem~\ref{thm:CLT} with the function $h:r\mapsto r^s/2$, its
integral condition reduces, up to positive multiplicative constants,
to
$$
\int_{c'}^\infty
c''t^{-k/(2(s+2))}\ud t=\infty
$$
for some $c',c''\in(0,\infty)$. Since $2s>v-2$, we may choose
$$
v+2<k<2(s+2),
$$
and the preceding integral diverges.

For part~(Ib), choose
$$
0<c<
\liminf_{|x|\to\infty}\bigl(g(x)-\pi_h(g)\bigr).
$$
Applying Theorem~\ref{thm:CLT} with the constant function
$h:r\mapsto c$, the integral condition reduces, up to positive
multiplicative constants, to
$$
\int_{c'}^\infty t^{-k/4}\ud t=\infty.
$$
If $v<2$, we may choose $k\in(v+2,4)$, proving part~(Ib).

%Pick $\eps>0$ and let the function $h_\eps(r) = r^{v+2-\eps}\vee 1$. By part~(a) of Proposition~\ref{prop:drift_ULA}, the inequality $Q_h^n(h_\eps\circ V)(x)\leq h_\eps\circ V(x) + C_\eps n$, holds for all $n\in\N$ and some $C_\eps\in(0,\infty)$. Moreover, for some $q\in(0,1)$, the function $L_q$ in~\eqref{eq:def_L_eps_q} satisfies $L_q(r) \geq  r^{2-k-\eps}$. Thus, $a(r)\leq L_q(h_\eps^{-1}(r))$, where $h_\eps^{-1}$ denotes the inverse of $h_\eps$, is satisfied by $a(r)= c_qr^{(2-k-\eps)/(v+2-\eps)}$. Thus, by Theorem~\ref{thm:f_rate}, applied with the functions $g=h_\eps$, $f\equiv 1$ and $a_g=a$, for any $x\in\R^d$ there exists $C_\TV\in(0,\infty)$ such that
%$
%\| Q_h^n(x,\cdot)-\pi_h\|_\TV\geq  C_\TV n^{(2-k-\eps)/(4+v-k-2\eps)}\quad\text{for all $n\in\N\setminus\{0\}$,}$ which concludes the proof since the bound holds for any $p\in(v+2,\infty)$. Moreover, applying Theorem~\ref{thm:f_rate} with $f_\star(r)=r^p$ for $p\in(0,v)$, yields a lower bound on the $f_p$-variation distance. 

\smallskip

\noindent \underline{Proof of Part (II)} Recall the \textbf{L}-drift conditions
\nameref{sub_drift_conditions} are satisfied by the functions
$V(x)=|x|\vee1$, $\Psi(r)=r^k$, for any $k\in(v+2,\infty)$, and
$\varphi(1/r)=C/r^3$ for some constant $C\in(0,\infty)$.
Fix \(p\in[0,v)\). Choose sufficiently small  \(\eta,\kappa,\delta>0\) so that 
$a>\max\{2,\beta\}$ and $\beta>p$, where 
$a\coloneqq v+2-\eta$, $k\coloneqq v+2+\kappa$ and 
$\beta\coloneqq k-2+\delta=v+\kappa+\delta$.

%For example, take \(\eta<v\) and \(\eta+\kappa+\delta<2\), with the parameters sufficiently small. 
Choose an admissible slowly growing \(D\), such as \(D(r)=1+\log r\), and note \(D(r)\le C_\delta r^\delta\). By definition~\eqref{eq:def_L_eps_q} of $L_q$ we have 
$\frac{c_q}{L_q(r)}\ge c r^{-\beta}$ for $r\ge1$. Define the function \(h(r)\coloneqq r^a\).  Proposition~\ref{prop:drift_ULA}(a) yields, for some constant \(b_a>0\),
the inequality 
$Q_h^nh\circ V_1(x)\le h\circ V_1(x)+b_an=:v_a(x,n)$ for all $x\in\R^d$ and $n\in\N$.

For part (IIa), take \(f_\star(r)=r^p\), and \(g(r)=r^{a-p}\). A sufficiently small constant \(c_0>0\) makes
$a_g(t)\coloneqq c_0t^{-\theta}
\le f_\star(g^{-1}(t))\frac{c_q}{L_q(g^{-1}(t))}$, for all $t\ge1$, where $\theta\coloneqq \frac{\beta-p}{a-p}\in(0,1)$.
Set
 $A_g(t)\coloneqq c_0t^{1-\theta}$, where 
Thus \(A_g\) is strictly increasing and unbounded, and for some constant $C>0$ we have
\[
a_g(A_g^{-1}(u))=C u^{-\theta/(1-\theta)}
=C u^{-(\beta-p)/(a-\beta)}.
\]
Theorem~\ref{thm:f_rate}(a) now gives, for fixed \(x\),
$\|Q_h^n(x,\cdot)-\pi_h\|_{f_p}
\ge C_x n^{-(\beta-p)/(a-\beta)}$,  $n\ge1$.
As \(\eta,\kappa,\delta\downarrow0\), the exponent tends to \((v-p)/2\). Choosing them sufficiently small gives the stated bound with any prescribed \(\varepsilon>0\). This includes the \(p=0\), the total-variation case.

For (IIb), take the functions prescribed by Theorem~\ref{thm:f_rate}(b). In particular, for \(r\ge2\),
$f_\star(r)=(r/2)^p$,
and
$g(r)=\frac{h(r/2)}{f_\star(r/2)}=2^{2p-a}r^{a-p}$,
with the required increasing linear extension on the closed interval \([1,2]\). Consequently
$f_\star(g^{-1}(t))\frac{c_q}{L_q(g^{-1}(t))}
\ge c t^{-(\beta-p)/(a-p)}$
for large \(t\). A smaller positive constant makes the inequality valid on the remaining compact interval. Use the same power \(a_g(t)=c_0t^{-\theta}\). Since \(V_1\) is 1-Lipschitz and $Q_h^n(x,\cdot)$ has all moments finite (since $\nabla \log \pi$ is bounded and the innovations are standard Gaussian), Theorem~\ref{thm:f_rate}(b) yields
$W_p(Q_h^n(x,\cdot),\pi_h)
\ge C_x n^{-(\beta-p)/(p(a-\beta))}$.
The exponent tends to \((v-p)/(2p)\), implying (IIb).
\smallskip

\noindent \underline{Proof of Part (III)} To prove the upper bound on the tail of the invariant measure $\pi_h$ for the ULA chain $X_h$, we consider the drift condition in Proposition~\ref{prop:drift_ULA}(a) with $s$ arbitrary close to $v+2$ and apply~\cite[Thm~16.1.12]{MarkovChains}. To obtain the lower bound we note that~\nameref{sub_drift_conditions} are satisfied, by the functions $V(x)=|x|\vee 1$, $\Psi(r) = r^{k}$, for any $k\in(v+2,\infty)$ and $\varphi(1/r)=C/r^{3}$ for some constant $C\in(0,\infty)$, and apply Theorem~\ref{thm:invariant}.
\end{proof}

\begin{proof}[Proof of Corollary~\ref{cor:infinite_asymp_var}]
\ In the regime $s\in(v-2,v)$ of Theorem~\ref{thm:ULA}\ref{thm:ULA_CLT} the CLT fails. We now prove that the asymptotic variance is in this case infinite.  
Indeed, using the stationary convention, we have
$\sigma_g^2=\lim_{n\to\infty}n\operatorname{Var}_{\pi_h}(S_n(g))$.
%Here is an argument for the divergent side. 

Let \(R_n=B\sqrt n\), where \(B>0\). 
Using Doob's submartingale maximal inequality,  we will prove below that  
for some $B$ sufficiently large, uniformly over \(|x|\ge R_n\) and all $n\in\N$,
we have
\begin{equation}
\label{eq:uniformity_estimate}
\mathbb P_x\!\left(\min_{0\le j<n}|X^{(h)}_j|\ge R_n/2\right)\ge 1/2.
\end{equation}

If  \(g(x)-\pi_h(g)\ge c_1|x|^s\) eventually for some constant $c_1>0$, on the event in~\eqref{eq:uniformity_estimate} the centred sum satisfies $\sum_{j=0}^{n-1}(g(X^{(h)}_j)-\pi_h(g))\geq nR_n^sc_1/2^s$
for all $n\in\N$. Consequently, for any small \(\zeta>0\), Theorem~\ref{thm:ULA}\ref{thm:ULA_invariant} implies there exist constants $c_2,c_3>0$ such that for all $n\geq1$ we have 
\[
n\operatorname{Var}_{\pi_h}(S_n(g))
=\frac1n\mathbb E_{\pi_h}\left[\left(\sum_{j=0}^{n-1}(g(X^{(h)}_j)-\pi_h(g))\right)^2\right]
\ge c_3nR_n^{2s}\pi_h(|\cdot|\ge R_n)
\ge c_4n^{1+s-(v+\zeta)/2}.
\]
Since \(2s>v-2\), we may choose 
\(0<\zeta<2+2s-v\) so that $n\operatorname{Var}_{\pi_h}(S_n(g))
\to\infty$ as $n\to\infty$.  

For a bounded observable $g$ with a strictly positive centered tail lower bound, the same calculation for \(s=0\)  works when \(v<2\).

It remains to prove~\eqref{eq:uniformity_estimate}. Fix $h>0$ and consider the ULA chain $X^{(h)}$ in~\eqref{eq:ULA}.
Since $|\nabla\log\pi(y)|=O(1/|y|)$ as $|y|\to\infty$, there exist
$C_0,r_0>0$ such that
$|\nabla\log\pi(y)|\leq \frac{C_0}{|y|}$
for all $|y|\geq r_0$.
Choose $B>0$ satisfying
$B^2\geq \max\{4r_0^2,16hC_0,64hd\}$,
and set $R_n\coloneqq B\sqrt{n}$ for $n\geq1$.
Define the Gaussian martingale
$M_0\coloneqq0$, $M_k\coloneqq\sqrt{2h}\sum_{j=1}^k Z_j$,
$k\geq1$.
For every starting state $x\in\R^d$, the process $(|M_k|^2)_{k\geq0}$
is a nonnegative submartingale and $\E_x[|M_n|^2]=2hdn$. 
Define $G_n\coloneqq
\left\{\max_{0\leq k\leq n}|M_k|<R_n/4\right\}$ and note that
 Doob's maximal inequality gives
\[
\P_x\left(G_n^c\right)
\leq \frac{16\E_x[|M_n|^2]}{R_n^2}
=\frac{32hd}{B^2}
\leq\frac12, \quad\text{implying $\P_x(G_n)\geq1/2$ for all $x\in\R^d$ and $n\geq0$.}
\]

Now fix $|x|\geq R_n$ and define
$\tau_{x,n}\coloneqq
\inf\{k\in\N:|X_k^{(h)}-x|\geq R_n/2\}$
(with $\inf\emptyset\coloneqq\infty$).
For $j<\tau_{x,n}$, we have
$|X_j^{(h)}|
\geq |x|-|X_j^{(h)}-x|
>R_n/2\geq r_0$,
and hence $|\nabla\log\pi(X_j^{(h)})|\leq2C_0/R_n$.
On $G_n\cap\{\tau_{x,n}\leq n\}$, the recursion in~\eqref{eq:ULA}
thus yields
\begin{align*}
|X_{\tau_{x,n}}^{(h)}-x|
&=\left|
h\sum_{j=0}^{\tau_{x,n}-1}\nabla\log\pi(X_j^{(h)})
+M_{\tau_{x,n}}
\right|
\leq \frac{2hC_0n}{R_n}+\frac{R_n}{4}
\leq \frac{R_n}{8}+\frac{R_n}{4}
<R_n/2,
\end{align*}
contradicting the definition of $\tau_{x,n}$.
It follows that $G_n\subseteq\{\tau_{x,n}>n\}$ and, on $G_n$,
\[
|X_j^{(h)}|\geq |x|-|X_j^{(h)}-x|>R_n/2
\qquad\text{for every }0\leq j\leq n.
\]
Consequently,~\eqref{eq:uniformity_estimate}
holds for all $|x|\geq R_n$ and $n\in\N$ concluding the proof of the corollary.
\end{proof}

\subsection{Unadjusted algorithms with  increments of infinite variance}
\label{subsec:proof_for_levy_driven}

\begin{proof}[Proof of Theorem~\ref{thm:levy_driven_discretisation}] 
Let $h>0$ and recall that $B(z,r)=\{z'\in\R^d:|z'-z|\leq r\}$ for $r\in(0,\infty)$ and $z\in\R^d$. Let $p_{\alpha,h}$ denote the density of $Z^{(h)}_k$ for any $k\in\N$ in~\eqref{eq:alpha_stable_ULA}. 
%The transition kernel of the chain $X$ in~\eqref{eq:alpha_stable_ULA}
%has transition density $p_h(x,y)\coloneqq p_{\alpha,h}(y-x-hb(x))$
%for $x,y\in\R^d$. 
Recall $1<\alpha<2$ and that the assumption on
$p_{\alpha,h}$ in~\eqref{eq:p_alpha_assumption} implies the following: there exists a constant $C_{h,p}>0$ such that for every $x\in\R^d$ we have
\begin{equation}
\label{eq:heavy_Tailed_move}
 \P(|Z^{(h)}_1|>r)=\int_{B(0,r)^c} p_h(x)\ud x \geq  C_{h,p}r^{-\alpha}\quad\text{ for all $r\geq1$.}
\end{equation}
Moreover, since the vector field $b$ is continuous, for every radius $R>0$ we have $$M_R\coloneqq \sup_{x\in B(0,R)}|x+hb(x)|<\infty.$$

\noindent\underline{Part~\ref{thm:levy_driven_discretisation:invariant}}
Let $R>0$ satisfy $\pi_h(B(0,R))>0$.
Note that $\{|Z^{(h)}_1|>r+M_R\}\subset \{|x+hb(x)+Z^{(h)}_1|>r\}$ for $x\in B(0,R)$ and $r\geq0$. By~\eqref{eq:heavy_Tailed_move} there exists a sufficiently small constant $C_{h,\pi_h}>0$, such that
\[
\begin{aligned}
\pi_h(|\cdot|>r)
=\int_{\R^d}\mathbb P(|x+b(x)+Z^{(h)}_1|>r)\,\pi_h(\ud x)
\ge\pi_h(B_R) \P(|Z^{(h)}_1|>r+M_R)\ge C_{h,\pi_h}r^{-\alpha},\quad r\geq1.
\end{aligned}
\]

\noindent \underline{Part~\ref{thm:levy_driven_discretisation:transient}.} By assumption, we have $\liminf_{|x|\to\infty} |b(x)|/|x|\to\infty$ and hence 
\begin{equation}
\label{eq:heavy_tailed_ULA_drift}
|x+hb(x)|-2h|x|\ge2|x|
%\left|x+hb(x)-\frac{x+hb(x)}{|x+hb(x)|}2h|x|\right|>|x| 
\qquad\text{for all $|x|$ sufficiently large.}
\end{equation}
Recall from~\eqref{eq:alpha_stable_ULA} that $X_1=x+hb(x)+Z^{(h)}_1$. Hence, on the event \(|Z^{(h)}_1|\le2h|x|\), we have \(|X_1|\ge2|x|\) for $|x|$ sufficiently large. Since \(0<1/V\le1\), where  $V(x) \coloneqq \log |x| \vee 1$, recursion~\eqref{eq:alpha_stable_ULA} and the tail estimate for the increment in~\eqref{eq:p_alpha_assumption} yield
\[
\begin{aligned}
\E_x[1/V(X_1)]-1/V(x)
&\leq \E\left[\frac{\1{|Z^{(h)}_1|\le2h|x|}}{V(X_1)}\right]-1/V(x)+ \E\1{|Z^{(h)}_1|>2h|x|}\\
& \le\frac1{\log(2|x|)}-\frac1{\log|x|}+C|x|^{-\alpha}
=-\frac{\log2}{\log|x|\log(2|x|)}+C|x|^{-\alpha}<0
\end{aligned}
\]
for all sufficiently large \(|x|\) and some constant $C>0$.
Since
$\inf_{|x|\le R}\mathbb P_x(X_1\in B(0,R)^c)
\ge\mathbb P(|Z^{(h)}_1|>R+M_R)>0$ for any $R>0$,
we have 
$$\P_x(X\text{ visits $B(0,R)^c$ infinitely often})=1.$$ 
Since $R>0$ was arbitrary, we obtain \(\limsup_{n\to\infty}V(X_n)=\infty\) almost surely.

Note that $\lim_{|x|\to\infty}1/V(x)\to0$. Thus the chain $X$ is transient by~\cite[Lemma~10]{Menshikov21} (applied with $f\coloneqq 1/V$, the fact that $\E_x1/V(X_1)-1/V(x)<0$ on a complement of a compact and the non-confinment property $\P_x(\limsup_{n\to\infty}V(X_n)=\infty)=1$ for all $x\in\R^d$).
\end{proof}

\begin{proof}[Proof of the estimate in~\eqref{eq:lower_Bound_estimate_stable_ULA}]
Fix $h>0$, let $1<\alpha<2$ and $0<v<\alpha$, and put
$\delta=\alpha-v$ and $V(x)=(1+|x|^2)^{1/2}$.
Assume the innovations are isotropic $\alpha$-stable and, outside
a ball,
$\langle x,b_\pi(x)\rangle\le-c|x|^{2-\delta}$
and $|b_\pi(x)|\le C|x|^{1-\delta}$
(see \cite[Lemmas~2.4--2.5]{Majka21} for these estimates under
the corresponding radial regularity assumptions).

Taylor expansion and the stable density bound give, for every
$a\in(\delta,\alpha)$ and a sufficiently large compact ball $C$,
\[
 Q_hV^a-V^a\le-c_aV^{a-\delta}+B_a\mathbf{1}_C,
 \qquad
 Q_h(1/V)-1/V
 \le C_hV^{-1-\delta}+B_h\mathbf{1}_C.
\]
For the reciprocal estimate, the noise contribution is
$O(V^{-1-\alpha}\log(e+V))=o(V^{-1-\delta})$.
The first inequality implies positive Harris recurrence and
\[
 Q_h^nV^a(x)\le V^a(x)+B_an.
\]
The positive transition density gives aperiodicity and,
together with recurrence, non-confinement.

Moreover, uniformly over $z\in\R^d$,
$\mathbb P(|z+L_h|>r)\ge c_hr^{-\alpha}$
for all sufficiently large $r$: use $|L_h|>3r$ when $|z|\le2r$,
and $|L_h|\le r$ otherwise.
A single increment therefore verifies
Assumption~\nameref{sub_drift_conditions}\ref{sub_drift_conditions(ii)}
with $\Psi(r)=r^\alpha$.
Thus the $\mathbf L$-drift conditions hold with
$\varphi(u)=C_hu^{1+\delta}$ and $\Psi(r)=r^\alpha$.

Since $L_q(r)\asymp r^vD(r)$, Theorem~\ref{thm:invariant}
yields, for every $\eta>0$,
\[
 \pi_h(V\ge r)\ge \frac{c}{r^vD(r)}
 \ge c_\eta r^{-(v+\eta)},\qquad r\ge1,
\]
where $D$ is chosen so that $D(r)\le C_\eta r^\eta$.
Choose $\eta>0$ and $\max\{\delta,v+\eta\}<a<\alpha$.
In Theorem~\ref{thm:f_rate}(a), take $f_\star\equiv1$,
use $H(r)=r^a$ as the moment function, and set
\[
 g(r)=r^a,\qquad
 a_g(r)=c r^{-(v+\eta)/a},\qquad
 A_g(r)=c r^{1-(v+\eta)/a},
\]
with $c>0$ sufficiently small.
The theorem gives
\[
 \|Q_h^n(x,\cdot)-\pi_h\|_{\TV}
 \ge C\bigl(V^a(x)+B_an\bigr)^{-(v+\eta)/(a-v-\eta)}
 \ge c_x n^{-(v+\eta)/(a-v-\eta)}.
\]
Taking $\eta$ sufficiently small and $a$ sufficiently close
to $\alpha$ proves~\eqref{eq:lower_Bound_estimate_stable_ULA}  for every $\varepsilon>0$.
\end{proof}

\subsection{Proofs for the logarithmic-noise Euler scheme}
\label{subsec:log-tails_SDE_chain_proofs}

Throughout this subsection, fix an admissible $h$ and abbreviate
$Q=Q_h$, $\pi=\pi_h$, $B=B_h$, $a=a_h$, $b=b_h$. Denote the innovation by $Z=Z^{(h)}_1$ and the Euler chain in~\eqref{eq:log-euler-chain} by $X=X^{(h)}$.
All constants may depend on $h$. Condition $ 0<h<\min\{1/\|A\|,2\lambda/\|A\|^2\}$
implies~\eqref{eq:log-euler-contraction}, since $B_h=I_d-hA$ satisfies the following inequalities for every $x\in\R^d$:
\[
 (1-h\|A\|)|x|\leq |(I_d-hA)x|\le(1-2\lambda h+h^2\|A\|^2)^{1/2}|x|.
\]
In what follows only $0<b\le a<1$ is used.

\begin{rem}[Reduction to a coercive drift matrix]
\label{rem:coercivity_assump}
The coercivity assumption on the drift matrix $A$ (i.e., $\langle Ax,x\rangle\ge\lambda|x|^2$) can also be imposed after a linear
change of coordinates. Indeed, if every eigenvalue of $A$ has positive
real part, then
$H=\int_0^\infty e^{-A^\top t}e^{-At}\,\ud t$
is positive definite and satisfies $A^\top H+HA=I_d$.
The matrix $H^{1/2}AH^{-1/2}$ is coercive, and the transformed dispersion
$H^{1/2}\Sigma(H^{-1/2}x)$ still satisfies the required bounds.
Thus the conclusions hold for such $A$ and sufficiently small $h$;
the equivalent norms change only the constants in the tail and metric
comparisons.
\end{rem}

\begin{lem}[Logarithmic drift and the lower-drift conditions]
\label{lem:log-euler-drift}
Let $V$ be given by~\eqref{eq:log-euler-V-metric}.
For every $u\in(0,\kappa)$,
\begin{equation}\label{eq:log-euler-drift-expansion}
 QV^u(x)-V(x)^u
 =u\log\left(\frac{|Bx|}{|x|}\right)V(x)^{u-1}
       +o(V(x)^{u-1}),\qquad |x|\to\infty,
\end{equation}
uniformly in the direction of $x$. Also,
\begin{equation}\label{eq:log-euler-inverse-expansion}
 Q(1/V)(x)-1/V(x)
 =-\log\left(\frac{|Bx|}{|x|}\right)V(x)^{-2}
       +o(V(x)^{-2}).
\end{equation}
Consequently, for $u\in(1,\kappa)$ there exist $c_u,b_u,R_u>0$ such that
\begin{equation}\label{eq:log-euler-upper-drift}
 QV^u-V^u\le-c_uV^{u-1}+b_u\mathbf1_{\{|x|\le R_u\}}.
\end{equation}
The chain $X$ is aperiodic and positive Harris recurrent, every compact set
is one-step small, and Assumption~\nameref{sub_drift_conditions} holds with
\begin{equation}\label{eq:log-euler-L-functions}
 \varphi(z)=Cz^2\quad(0<z\le1),\qquad \Psi(r)=r^\kappa\quad(r\ge1).
\end{equation}
\end{lem}

\begin{proof}
Put $U=\log(1+|Z|)$. The tail assumption implies
\begin{equation}\label{eq:log-euler-U-tail}
 \mathbb P(U>t)\le Ct^{-\kappa},\qquad
 \mathbb E[U^u\mathbf1_{\{U>t\}}]\le C_ut^{u-\kappa}
 \quad(0<u<\kappa),
\end{equation}
for large $t$, by integration of the tail.
The density bound~\eqref{eq:log-euler-noise-density} prevents a large innovation from placing excessive mass in a
small neighborhood of the point that cancels the deterministic update.

Write $r=|x|$ and $y=Bx$.
Then $V(y)=V(x)+d_x+O(r^{-1})$, uniformly in $x/r$, where
\[
 d_x=\log(|Bx|/|x|)\in[\log b,\log a]\subset(-\infty,0).
\]
On $\{|Z|\le\sqrt r\}$, the bounds on $\Sigma$ imply
\[
 |V(y+\Sigma(x)Z)-V(y)|\le Cr^{-1/2}.
\]
For a fixed $u>0$, the contribution to
$\mathbb E[V(y+\Sigma(x)Z)^u-V(y)^u]$ on this event is therefore
$O(V(x)^{u-1}r^{-1/2})$.
On its complement $\{|Z|>\sqrt r\}$, we have 
$V(y+\Sigma(x)Z)\le C+\max\{V(x),U\}$,
and~\eqref{eq:log-euler-U-tail} gives an absolute contribution
$O(V(x)^{u-\kappa})$ for $u<\kappa$.
Since $\kappa>1$, this is $o(V(x)^{u-1})$.
Expanding $V(y)^u-V(x)^u$ proves~\eqref{eq:log-euler-drift-expansion}.

For the inverse, an additional estimate excludes near cancellation.
On $\{|y+\Sigma(x)Z|\le\sqrt r\}$, for large $r$,
\[
 |\Sigma(x)^{-1}(w-y)|\ge cr\qquad(|w|\le\sqrt r).
\]
Using~\eqref{eq:log-euler-noise-density} and
$|\det\Sigma(x)|\ge\sigma_-^d$, the probability of this event is
$O(r^{-d/2})$. Its contribution to an inverse difference is of the
same order, since $0<1/V\le1$.
On $\{|Z|>\sqrt r,\ |y+\Sigma(x)Z|>\sqrt r\}$, both inverse values
are at most $C/V(x)$. Hence the contribution is $O(V(x)^{-(\kappa+1)})$.
On $\{|Z|\le\sqrt r\}$ it is $O(V(x)^{-2}r^{-1/2})$.
All three terms are $o(V(x)^{-2})$, so expansion of $1/V(y)$ gives
\eqref{eq:log-euler-inverse-expansion}.

The transition density of $X_1$, started at $X_0=x$, is
\begin{equation}
\label{eq:density_Levy_driver}
 q(x,w)=\frac{p_h(\Sigma(x)^{-1}(w-Bx))}{|\det\Sigma(x)|}.
\end{equation}
It is continuous and strictly positive. In particular the chain is
Lebesgue-irreducible, aperiodic and Feller, and every compact set is
one-step small. On compact sets $QV^u$ is bounded, by the logarithmic
moment bound. Thus~\eqref{eq:log-euler-drift-expansion} and
$\log a<0$ give~\eqref{eq:log-euler-upper-drift}.
Choosing $u\in(1,\kappa)$ gives positive Harris recurrence and uniqueness
of the invariant probability by the Foster--Lyapunov criterion;
see~\cite[Chapter~11]{meynadntweedie}.

To check non-confinement and the exit-height condition, first note that
for all sufficiently large $R$, uniformly in $x$,
\begin{equation}\label{eq:log-euler-uniform-jump}
 Q(x,\{|w|>R\})\ge c(\log R)^{-\kappa}.
\end{equation}
Indeed, when $|Bx|\le2R$ use the event $|Z|>3R/\sigma_-$;
when $|Bx|>2R$ use $|Z|\le R/\sigma_+$, whose probability tends to one as $R\to\infty$.
The same argument gives a positive uniform probability of exiting any
fixed ball. Since the chain is Feller and $V$ has compact sublevel sets,
Lemma~\ref{lem:non_confinement} gives
$\limsup_{n\to\infty}V(X_n)=\infty$ almost surely from every starting state.
Moreover,~\eqref{eq:log-euler-uniform-jump} implies
\[
 \inf_x Q(x,\{V>r\})\ge c r^{-\kappa}
\]
for all large $r$. If $V(x)\ge\ell+m$, the event $V(X_1)>r$ implies
$T^{(r)}<S_{(\ell)}$ unless $V(x)>r$, in which case $T^{(r)}=0$ already.
This verifies Assumption~\nameref{sub_drift_conditions}\ref{sub_drift_conditions(ii)},
with the passage times $T^{(r)}$ and $S_{(\ell)}$ defined there, after
increasing $\ell_0$. The inverse expansion verifies
Assumption~\nameref{sub_drift_conditions}\ref{sub_drift_conditions(i)},
allowing its compact-set term. Finally, $r\varphi(1/r)=C/r$ is strictly decreasing to zero as $r\to\infty$ and
$\Psi(r)=r^\kappa$ is increasing, differentiable and submultiplicative.
\end{proof}

\begin{lem}[Stationary and finite-time tails]\label{lem:log-euler-tails}
For every $\eps>0$ there exist $c,C>0$ such that
\[
 c(\log r)^{1-\kappa-\eps}\le\pi(|\cdot|>r)
       \le C(\log r)^{1-\kappa}
\]
for all sufficiently large $r$.
For every fixed $x$, there is $t_x<\infty$ such that, for all $n\ge1$
and $t\ge t_x$,
\begin{align}
 Q^n(x,\{V>t\})&\le Cn t^{-\kappa},
       \label{eq:log-euler-marginal-tail}\\
 Q^n\bigl(V^q\mathbf1_{\{V>t\}}\bigr)(x)
       &\le C_q n t^{q-\kappa},\qquad 0\le q<\kappa.
       \label{eq:log-euler-marginal-weighted-tail}
\end{align}
\end{lem}

\begin{proof}
Recall the recursion of the Euler chain $X$ in~\eqref{eq:log-euler-chain}, yielding
a pathwise upper bound
\begin{equation}\label{eq:log-euler-pathwise-upper}
 |X_n|\le a^n|X_0|+
       \sigma_+\sum_{j=1}^n a^{n-j}|Z_j|.
\end{equation}
For independent copies $(Z_j)_{j\ge0}$, the series
$H=\sum_{j=0}^\infty a^j|Z_j|$ is finite almost surely by Borel-Cantelli:
if $0<\eta<-\log a$, the sum of the probabilities
$\mathbb P(|Z_j|>e^{\eta j})$ is finite since $\kappa>1$.
The weights $w_j=(1-\sqrt a)a^{j/2}$ sum to one, so
$\{H>r\}\subset\cup_{j\in \N}\{|Z_j|> w_jr\}$. Thus there exist constants $C,C'>0$ such that 
\begin{align*}
 \mathbb P(H>r)
 &\le\sum_{j=0}^\infty
     \mathbb P\bigl(|Z_j|>r(1-\sqrt a)a^{-j/2}\bigr)
 \le C\sum_{j=0}^\infty(\log r+cj)^{-\kappa}
 \le C'(\log r)^{1-\kappa}, \quad r\geq2.
\end{align*}
Start~\eqref{eq:log-euler-pathwise-upper} in stationarity and let
$n\to\infty$. The term $a^n|X_0|$ tends to zero almost surely,
and the finite sum has the same distribution as a partial sum of $H$.
Using a fixed factor of two in the threshold proves the stationary
upper bound.

For the stationary lower bound, Lemma~\ref{lem:log-euler-drift}
verifies Assumption~\nameref{sub_drift_conditions} with
$\varphi(z)=Cz^2$ and $\Psi(t)=t^\kappa$.
Taking $q=1/2$ in Theorem~\ref{thm:invariant}, we have
\[
 L_{1/2}(t)
 =t\varphi(1/t)\Psi(4t)D(t)
 =C4^\kappa t^{\kappa-1}D(t),
\]
where $D$ is as in~\eqref{eq:def_L_eps_q}.
Consequently,
$\pi(V\ge t)\ge c_D\,\frac{t^{1-\kappa}}{D(t)}$
for all sufficiently large $t$.
Since $\pi$ has a density (given by $y\mapsto \int_{\R^d}q(x,y)\pi(\ud x)$, where $q(x,y)$ is a density of $Q(x,\ud y)$), $V(x)=1+\log(1+|x|)$ and $D$ can be chosen to grow more slowly than every
positive power, for any fixed $\eps>0$ we obtain
\[
 \pi(|\cdot|>r)
 \ge c'_D\,
 \frac{(\log r)^{1-\kappa}}{D(1+\log(1+r))}\ge c_\varepsilon
       (\log r)^{1-\kappa-\varepsilon}
\]
for all sufficiently large $r$.

For a deterministic starting state $x$,~\eqref{eq:log-euler-pathwise-upper}
also gives
\[
 |X_n|\le |x|+\frac{\sigma_+}{1-a}\max_{1\le j\le n}|Z_j|.
\]
A union bound proves~\eqref{eq:log-euler-marginal-tail} for
$t\ge t_x$, with $t_x$ independent of $n$.
For $0<q<\kappa$, integrate this bound using
\[
 \mathbb E[W^q\mathbf1_{\{W>t\}}]
 =t^q\mathbb P(W>t)+q\int_t^\infty u^{q-1}\mathbb P(W>u)\,du.
\]
The case $q=0$ is~\eqref{eq:log-euler-marginal-tail} itself.
\end{proof}

\begin{proof}[Proof of Theorem~\ref{thm:log-euler}]
Lemma~\ref{lem:log-euler-drift} proves the recurrence assertions.
Lemma~\ref{lem:log-euler-tails}, the definition of $V$ in~\eqref{eq:log-euler-V-metric}, and the tail-integral
formula give part~(I). The bounds extend to $r\ge2$ by changing constants:
$\pi$ charges every non-empty open set because its transition density $y\mapsto \int_{\R^d}q(x,y)\pi(\ud x)$, for $q(x,y)$ in~\eqref{eq:density_Levy_driver}, is
strictly positive.

To settle the critical moment in~\eqref{eq:log-euler-log-moments}, choose $D(t)=\log t$ for all
sufficiently large $t$ in Theorem~\ref{thm:invariant}.
The \textbf{L}-drift condition~\nameref{sub_drift_conditions} established in Lemma~\ref{lem:log-euler-drift} then gives
$\pi(V>t)\ge c\,\frac{t^{1-\kappa}}{\log t}$.
Thus,
$\pi(V^{\kappa-1})
 \ge c'\int^\infty\frac{\mathrm dt}{t\log t}
 =\infty$.
Together with the stationary upper bound, this proves
\eqref{eq:log-euler-log-moments}.

\smallskip\noindent\emph{Part (II): upper bounds.}
Fix $q\in[0,\kappa-1)$ and choose $u\in(q+1,\kappa)$.
With $W=V^u$,~\eqref{eq:log-euler-upper-drift} has the form
$QW-W\le-cW^{1-1/u}+b_0\mathbf1_C$,
where $C$ is compact and one-step small, and $\sup_CW<\infty$.
Apply~\cite[Theorem~2.8 and equation~(2.9)]{MR2071426} with
power $1-1/u$ and interpolation parameter $q/(u-1)$.
It gives
$\|Q^n(x,\cdot)-\pi\|_{V^q}=o(n^{-(u-1-q)})$.
Taking $u<\kappa$ sufficiently close to $\kappa$, and enlarging a
constant to include the finitely many small $n$, proves the upper bound
in~\eqref{eq:log-euler-f-rates}.

\smallskip\noindent\emph{Part (II): lower bounds.}
%Use the test function $V^q\mathbf1_{\{V>t\}}$, justified by bounded
%truncation if $q>0$. By part~(I) and
%\eqref{eq:log-euler-marginal-weighted-tail},
%\[
% \|Q^n(x,\cdot)-\pi\|_{V^q}\ge  \pi\bigl(V^q\mathbf1_{\{V>t\}}\bigr)-Q^n\bigl(V^q\mathbf1_{\{V>t\}}\bigr)(x)
% \ge c t^{q+1-\kappa}-C_q n t^{q-\kappa},
% \qquad t\ge t_x.
%\]
%Choose $t=Kn$, where $K$ is sufficiently large and independent of $n$.
%This proves the stated lower bound without an additional power loss.
Theorem~\ref{thm:f_rate}\ref{assumption:f_convergence}  yields a
lower bound with exponent arbitrarily close to $\kappa-1-q$.
Indeed, choose $u\in(\max\{1,\kappa-1\},\kappa)$, set
$f_\star(r)=r^q$, $h(r)=r^u$ and $g(r)=r^{u-q}$, and use
$Q^nV^u(x)\le V(x)^u+C_un$.
With the functions in~\eqref{eq:log-euler-L-functions}, the function
$L_\eta$ of Theorem~\ref{thm:invariant}, for a fixed $\eta\in(0,1)$,
is a constant multiple of $r^{\kappa-1}D(r)$.
For $0<\delta<u-\kappa+1$, choose $D$ as in that theorem with
$D(r)\le C_\delta r^\delta$, and then choose $c>0$ small enough that
\[
 a_g(r)=c r^{-(\kappa-1-q+\delta)/(u-q)}
\]
is an admissible lower envelope. The resulting $A_g(r)=ra_g(r)$ is
strictly increasing and unbounded. Theorem~\ref{thm:f_rate}\ref{assumption:f_convergence} implies a lower bound
of order $n^{-(\kappa-1-q+\delta)/(u-\kappa+1-\delta)}$.
Letting $u\uparrow\kappa$ and $\delta\downarrow0$ gives any positive
power loss.

\smallskip\noindent\emph{Part (III): the finite Wasserstein range.}
The function $r\mapsto\log(1+r)$ is increasing and subadditive, as is
its $\rho$-th power for $0<\rho\le1$, making $d_\rho$ is a metric.
Completeness and the equality of its topology with the Euclidean topology
follow directly from its definition. Furthermore,
\begin{equation}\label{eq:log-euler-V-Lipschitz}
 |V(x)-V(y)|\le\log(1+|x-y|)=d_1(x,y).
\end{equation}
Put $q=\rho p$. The finite-moment assertion follows from part~(I) and
\eqref{eq:log-euler-marginal-tail} when $q<\kappa-1$.
Couple the common part of two measures identically and couple their
residual parts independently. 
By~\cite[Thm.~6.15]{Villani09}, applied to the metric $d_\rho$
with reference point $0$, and since $q=\rho p$,
\[
 W_{p,\rho}(\mu,\nu)^p
 \le 2^{p-1}\int_{\mathbb R^d}
       d_\rho(0,x)^p\,|\mu-\nu|(\mathrm dx)
 \le 2^{p-1}\|\mu-\nu\|_{V^q},
\]
where the last inequality uses
$d_\rho(0,x)^p=(\log(1+|x|))^q\le V(x)^q$.
Apply the upper bound in (II) with a loss smaller than $p\varepsilon$
to obtain the upper bound in~\eqref{eq:log-euler-W-rates}.

Consider the  lower bound. Every coupling $(Y,Z')$ with laws $\pi$ and
$Q^n(x,\cdot)$  by~\eqref{eq:log-euler-V-Lipschitz} satisfies $\{P(V(Y)>2t,\ V(Z')\le t\}
\subset\{|V(Y)-V(Z')|>t\}\subseteq
\{d_\rho(Y,Z')>t\}$ and hence
\begin{align*}
 \mathbb E[d_\rho(Y,Z')^p]
 &\ge t^q\mathbb P(d_\rho(Y,Z')> t) \ge t^q\mathbb P(V(Y)>2t,\ V(Z')\le t)\\
 &\ge t^q\{\pi(V>2t)-Q^n(x,\{V>t\})\}\ge c t^{q+1-\kappa}-Cn t^{q-\kappa}.
\end{align*}
 Taking $t=Kn^{1/(1-\eta)}$ for some $\eta\in(0,1)$ and choosing $K$ sufficiently large and independent of $n$ proves the lower bound.

\smallskip\noindent\emph{Part (III): infinite transport costs.} By Lemma~\ref{lem:log-euler-tails} we have  $Q^nV^{\kappa-1}(x)<\infty$ for all $n\in\N$ and $x\in\R^d$. In the proof of Part~(I) we established $\pi(V^{\kappa-1})=\infty$. By the definitions in~\eqref{eq:log-euler-V-metric} and the inequality in~\eqref{eq:log-euler-V-Lipschitz},  we have $V(y)^{\kappa-1}\leq 2^{\kappa-1} (V(z)^{\kappa-1}+d_1(y,z)^{\kappa-1})$. Since $\rho p\geq \kappa-1$, we obtain
$V(y)^{\kappa-1}\leq 2^{\kappa-1} (V(z)^{\kappa-1}+1+d_\rho(y,z)^p)$ for all $z,y\in\R^d$. Hence for any coupling $(Y,Z)$  of $\pi$ and $Q^n(x,\cdot)$ we have $\E[d_\rho(Y,Z)^p]=\infty$. For the Euclidean norm, a similar argument using $V(y)^{\kappa-1}\leq 2^{\kappa-1} (V(z)^{\kappa-1}+1+|x-y|^p)$ implies infinite transport cost for the $\cW_p$-distance.

\smallskip\noindent\emph{Part (IV): failure of the CLT.}
All compact sets are small and $V$ has compact sublevel sets.
By Lemma~\ref{lem:log-euler-drift}, Assumption~\nameref{sub_drift_conditions}
holds with $\varphi(z)=Cz^2$ and $\Psi(r)=r^\kappa$.
If $s>0$ and $g\ge cV^s$ eventually, then
$g-\pi(g)\ge(c/2)V^s$ eventually. Apply Theorem~\ref{thm:CLT}
with the lower envelope $h_g(r)=(c/2)r^s$.
The function $G_{h_g}$ defined in~\eqref{eq:G_h} satisfies
$G_{h_g}(r)\asymp r^{s+1}$, so the integral in
\eqref{eq:CLT_integral_test} is, up to constants,
$\int^\infty t^{-\kappa/(2(s+1))}\,dt$.
It diverges exactly when $2s\ge\kappa-2$.
For a bounded $g$ satisfying  $ \liminf_{|x|\to\infty}\bigl(g(x)-\pi_h(g)\bigr)>0$,
a positive constant envelope gives $G_{h_g}(r)\asymp r$ and the divergent
integral $\int^\infty t^{-\kappa/2}\,dt$ for $\kappa\le2$.
This also explains why the critical case is included here.

\smallskip\noindent\emph{Part (IV): the sufficient condition and nondegeneracy.}
Suppose $|g|\le CV^s$ and $2s<\kappa-2$.
Choose $u\in(2s+2,\kappa)$ and use $W=V^u$ in
\eqref{eq:log-euler-upper-drift}.
In~\cite[Theorem~5 and its polynomial specialization, pp.~245--246]{MR2446322},
set
\[
 \gamma=1-1/u,\qquad \psi(w)=w^{s/u}.
\]
The required inequalities hold because
$\gamma>1/2$ and $s/u<\gamma-1/2$.
For the one-step minorization one may take its measure to be the uniform
law on a sufficiently large ball $C$, so that it gives $C$ mass one.
The required initial moment is finite for every deterministic initial
state. The cited theorem gives finite stationary asymptotic variance
and a CLT if that variance is positive.

For completeness, positivity follows here for any $g$ that is not
$\pi$-almost everywhere constant. Use the preceding minorization with
constant $\epsilon\in(0,1)$ and uniform measure $\nu$ on $C$.
The one-step split chain has independent regenerative tours, and the
usual regenerative variance formula identifies the asymptotic variance
with the second moment of a centred tour reward, divided by the mean
tour length. These moments are finite under the preceding sufficient
condition. If the variance were zero, the centred reward would vanish
almost surely. Tours of length one, which occur with probability
$\epsilon$, force $g-\pi(g)=0$ $\nu$-almost everywhere, hence
Lebesgue-almost everywhere on $C$.
If $g$ is not constant $\pi$-almost everywhere, choose a bounded
measurable set $D\subset C^c$ of positive Lebesgue measure on which
$g-\pi(g)$ has a fixed nonzero sign.
A tour can, with positive probability, start in $C$, fail the first
splitting coin, move to $D$, return to $C$, and then regenerate.
Strict positivity of the transition density ensures each of these
transitions has positive probability; outside $C$ the residual transition
from $C$ equals the original density divided by $1-\epsilon$.
Its three-step centred reward is nonzero, a contradiction.
Thus the variance is positive. Finally, the claimed $L^2$ range for
$g=V^s$ outside a ball follows from~\eqref{eq:log-euler-log-moments}.
\end{proof}

\appendix
\section{An infinite-variation driver with logarithmic tails}
\begin{example}[An explicit infinite-variation L\'evy driver]\label{ex:log-euler-driver}
Fix $\alpha_0\in(1,2)$, $\kappa>1$, and $c_0,c_1>0$.
Let $\omega$ be a bounded, nonnegative measurable function on the unit
sphere, with strictly positive surface integral.
Let $L$ be the L\'evy process with zero Gaussian part, zero drift
relative to the truncation $z\mathbf 1_{\{|z|\le1\}}$, and L\'evy measure
\begin{equation*}
 \nu(dz)=\left(
 \frac{c_0\mathbf 1_{\{0<|z|\le1\}}}{|z|^{d+\alpha_0}}
 +\frac{c_1\omega(z/|z|)\mathbf 1_{\{|z|>1\}}}
 {|z|^d[\log(e+|z|)]^{\kappa+1}}
 \right)dz.
\end{equation*}
Then $L$ is a pure-jump process with infinite activity and almost surely
infinite total variation on every nontrivial time interval.
For every $t>0$, its law has a strictly positive $C^\infty$ density $p_t$,
all of whose derivatives are bounded, and
\[
 p_t(z)\le K_t(1+|z|)^{-d},\qquad
 \mathbb P(|L_t|>r)\asymp(\log r)^{-\kappa}\quad(r\to\infty).
\]
Thus it satisfies all the noise assumptions
in~\eqref{eq:log-euler-noise-density}--\eqref{eq:log-euler-noise}.
Its small jumps satisfy Orey's condition; the proof below includes the
multidimensional Fourier argument, with
\cite[Proposition~28.3]{MR3185174} giving the one-dimensional criterion.
Taking $\omega\equiv1$ gives an isotropic example. Taking $d=1$,
$\omega(1)=1$ and $\omega(-1)=0$ gives a positive logarithmic jump tail
and no negative jumps of size greater than one, satisfying the tail
asymmetry in~\cite[Section~3.2, equation~(18)]{brešar2024subexponential}.
The assumptions allow genuinely state-dependent $\Sigma$ in
every dimension and do not require the large jumps to be isotropic.
\end{example}

\begin{proof}
\emph{Existence and infinite variation.}
Polar coordinates show that
\[
 \int_{\mathbb R^d}(1\wedge|z|^2)\nu(dz)<\infty,
 \qquad
 \int_{|z|\le1}|z|\nu(dz)
 =c_0|\mathbb S^{d-1}|\int_0^1 r^{-\alpha_0}\,dr=\infty.
\]
Consequently $\nu$ is a L\'evy measure of infinite mass, and the specified
triplet defines a pure-jump L\'evy process. The variation criterion
\cite[Theorem~21.9]{MR3185174} gives infinite variation almost surely.
Stationary increments and a countable intersection over rational
endpoints give this conclusion simultaneously on every nontrivial
time interval.

\medskip\noindent\emph{Orey's condition and smooth densities.}
For each unit vector $e$ and $0<r\le1$,
\begin{equation*}
 \int_{|z|\le r}\langle e,z\rangle^2\nu(dz)
 =\frac{c_0|\mathbb S^{d-1}|}{d(2-\alpha_0)}r^{2-\alpha_0}.
\end{equation*}
In dimension one this is Orey's condition in
\cite[Proposition~28.3]{MR3185174}. Here the bound is uniform in direction,
and its Fourier proof works in every dimension as follows.
Let $T$ denote the symmetric L\'evy process with zero drift and L\'evy
measure $\nu_0(dz)=c_0\mathbf 1_{\{0<|z|\le1\}}|z|^{-d-\alpha_0}dz$.
Write $\mathbb E e^{i\langle\xi,T_t\rangle}=e^{-t\psi_0(\xi)}$.
Since $1-\cos u\ge c u^2$ for $|u|\le1$, for $|\xi|\ge1$,
\[
 \psi_0(\xi)=\int(1-\cos\langle\xi,z\rangle)\nu_0(dz)
 \ge c|\xi|^2\int_{|z|\le|\xi|^{-1}}
       \langle\xi/|\xi|,z\rangle^2\nu_0(dz)
 \ge c'|\xi|^{\alpha_0}.
\]
Thus $|\xi|^k e^{-t\psi_0(\xi)}$ is integrable for every $k\ge0$ and
$t>0$. Fourier inversion gives a $C^\infty$ density $q_t$ of $T_t$
with bounded derivatives of every order.

We also need a spatial bound and strict positivity, which do not follow
from smoothness alone. Since the jumps of $T$ are bounded,
$\psi_0$ is $C^\infty$, its first derivatives are $O(|\xi|)$, and its
derivatives of order at least two are bounded. These assertions follow
by differentiating the integral, using
$|\sin\langle\xi,z\rangle|\le|\xi||z|$ for the first derivative and
$\int |z|^k\nu_0(dz)<\infty$ for $k\ge2$.
Together with the preceding exponential decay, these estimates show
that $e^{-t\psi_0}$, and hence $q_t$, belongs to the Schwartz class.
In particular,
\begin{equation}\label{eq:log-euler-small-density}
 q_t(z)\le C_{t,N}(1+|z|)^{-N}\quad\text{for every }N>0.
\end{equation}

To check positivity, first observe that $T_t$ has full support.
Indeed, split off the jumps in $\{1/4<|z|<1/2\}$.
Their normalized density $a$ is symmetric and bounded, and $a*a$ is
continuous with $(a*a)(0)=\int a(z)^2dz>0$.
Hence $a*a>0$ on a ball $B(0,\delta)$, and
$a^{*2n}>0$ on $B(0,n\delta)$.
The Poisson mixture of these annular jumps therefore has full support;
convolution with the independent remaining jumps preserves full support.
Furthermore, Fourier inversion gives
\[
 q_{t/2}(0)=(2\pi)^{-d}\int_{\mathbb R^d}
                  e^{-(t/2)\psi_0(\xi)}\,d\xi>0.
\]
By continuity $q_{t/2}$ is positive near zero. Full support and the
convolution identity
$q_t(x)=\int q_{t/2}(x-y)\mathbb P(T_{t/2}\in dy)$ now give
$q_t(x)>0$ for every $x$.

\medskip\noindent\emph{The density and tails of the full driver.}
Use the L\'evy--It\^o decomposition $L=T+H$, where $H$ comprises the
jumps of size greater than one and is independent of $T$.
Writing $H_t=\sum_{j=1}^{N_t}J_j$, its rate
$\vartheta=\nu(\{|z|>1\})$ is finite and strictly positive, and its
jump density is
\begin{equation*}
 j_\kappa(z)=\frac{c_1}{\vartheta}
 \frac{\omega(z/|z|)\mathbf 1_{\{|z|>1\}}}
 {|z|^d[\log(e+|z|)]^{\kappa+1}}.
\end{equation*}
Polar coordinates give
\[
 \mathbb P(|J_1|>r)\asymp(\log r)^{-\kappa},
 \qquad j_\kappa(z)\le C(1+|z|)^{-d}.
\]
The full density is $p_t=q_t*\mathcal L(H_t)$; consequently it is
strictly positive and $C^\infty$, with all derivatives bounded.

If probability densities $f,g$ satisfy bounds with respective constants
$K_f,K_g$ and envelope $(1+|z|)^{-d}$, split their convolution into
$\{|y|\ge|z|/2\}$ and its complement to obtain
\[
 (f*g)(z)\le 2^d(K_f+K_g)(1+|z|)^{-d}.
\]
Iteration, starting from~\eqref{eq:log-euler-small-density}, bounds the convolution of $q_t$ and
$j_\kappa^{*n}$ by $C_tD^n(1+|z|)^{-d}$ for a finite $D$.
The Poisson mixture satisfies~\eqref{eq:log-euler-noise-density}, because
$\mathbb E[D^{N_t}]<\infty$.

For a lower tail bound, use the independent event
$\{N_t=1,\ |J_1|>2r,\ |T_t|\le r\}$.
Its probability is bounded below by $c_t(\log r)^{-\kappa}$ for large $r$.
For the upper bound, $T_t$ has moments of every order
by~\eqref{eq:log-euler-small-density}, so its tail is negligible on this scale.
On $\{1\le N_t=k\le\sqrt r\}$,
\[
 \mathbb P\left(\sum_{j=1}^k|J_j|>r/2\right)
 \le k\mathbb P(|J_1|>r/(2k))
 \le Ck(\log r)^{-\kappa}.
\]
Average this inequality over $N_t$, and use the Poisson tail for
$N_t>\sqrt r$, to obtain the upper bound in~\eqref{eq:log-euler-noise}.
In the stated one-dimensional asymmetric example,
$\nu((r,\infty))\asymp(\log r)^{-\kappa}$ and
$\nu((-\infty,-r))=0$ for $r>1$.
The asymmetry assumption of the source example is therefore satisfied
with $m_c=\kappa$. The finite-rate process $H$ is only the large-jump
part of $L$; its small-jump part $T$ has infinite activity and variation.
\end{proof}

\section*{Acknowledgements}
The main part of this work was undertaken while MB was a postdoc at Warwick Statistics, funded by the 
EPSRC grant EP/V009478/1 he is currently supported by CUHK-SZ start-up UDF01004230. AM and GR were also supported by EP/V009478/1. The authors would like to thank the Isaac Newton Institute for Mathematical Sciences, Cambridge, for support  during the INI satellite programme \textit{Heavy tails in machine learning}, hosted by The Alan Turing Institute, London, and the INI programme \textit{Stochastic systems for anomalous diffusion} hosted at INI in Cambridge, where work on this paper was undertaken. These programmes were supported by the EPSRC grant EP/R014604/1. AM was also supported by the EPSRC grant EP/W006227/1.

 \section*{AI statement}

The main ideas underlying the proofs of the principal results were developed by the authors without AI assistance. During the writing process, AI tools (GPT-6 Astra and GPT-5.6 Sol) were used for checking technical arguments, identifying relevant references, and correcting spelling and language. In particular, the proofs of Theorems~\ref{thm:CLT} and~\ref{thm:levy_driven_discretisation} were improved with the assistance of AI tools.

\bibliography{lower}

\newcommand{\etalchar}[1]{$^{#1}$}
\providecommand{\bysame}{\leavevmode\hbox to3em{\hrulefill}\thinspace}
\providecommand{\MR}{\relax\ifhmode\unskip\space\fi MR }
\providecommand{\MRhref}[2]{%
  \href{http://www.ams.org/mathscinet-getitem?mr=#1}{#2}
}
\providecommand{\href}[2]{#2}
\begin{thebibliography}{ALPW26}

\bibitem[AIM96]{Aspandiiarov96}
S.~Aspandiiarov, R.~Iasnogorodski, and M.~Menshikov, \emph{Passage-time moments for nonnegative stochastic processes and an application to reflected random walks in a quadrant}, Ann. Probab. \textbf{24} (1996), no.~2, 932--960.

\bibitem[ALPW22]{MR4524509}
Christophe Andrieu, Anthony Lee, Sam Power, and Andi~Q. Wang, \emph{Comparison of {M}arkov chains via weak {P}oincar\'{e} inequalities with application to pseudo-marginal {MCMC}}, Ann. Statist. \textbf{50} (2022), no.~6, 3592--3618. \MR{4524509}

\bibitem[ALPW26]{MR5033571}
\bysame, \emph{Weak {P}oincar\'{e} inequalities for {M}arkov chains: theory and applications}, Ann. Appl. Probab. \textbf{36} (2026), no.~1, 46--107. \MR{5033571}

\bibitem[APS19]{MR3910024}
Ari Arapostathis, Guodong Pang, and Nikola Sandri\'{c}, \emph{Ergodicity of a {L}\'{e}vy-driven {SDE} arising from multiclass many-server queues}, Ann. Appl. Probab. \textbf{29} (2019), no.~2, 1070--1126. \MR{3910024}

\bibitem[BM25]{brevsar2024non}
Miha Brešar and Aleksandar Mijatović, \emph{{Nonasymptotic bounds for forward processes in denoising diffusions: Ornstein–Uhlenbeck is hard to beat}}, The Annals of Applied Probability \textbf{35} (2025), no.~6, 4439 -- 4463.

\bibitem[BM26]{brešar2024subexponential}
Miha Bre\v{s}ar and Aleksandar Mijatovi\'{c}, \emph{Subexponential lower bounds for {$f$}-ergodic {M}arkov processes}, Probab. Theory Related Fields \textbf{194} (2026), no.~1-2, 997--1054. \MR{5008799}

\bibitem[Che99]{Chen99}
Xia Chen, \emph{Limit theorems for functionals of ergodic {M}arkov chains with general state space}, Mem. Amer. Math. Soc. \textbf{139} (1999), no.~664, xiv+203. \MR{1491814}

\bibitem[DFM16]{Durmus16}
Alain Durmus, Gersende Fort, and \'{E}ric Moulines, \emph{Subgeometric rates of convergence in {W}asserstein distance for {M}arkov chains}, Ann. Inst. Henri Poincar\'{e} Probab. Stat. \textbf{52} (2016), no.~4, 1799--1822. \MR{3573296}

\bibitem[DFMS04]{MR2071426}
Randal Douc, Gersende Fort, Eric Moulines, and Philippe Soulier, \emph{Practical drift conditions for subgeometric rates of convergence}, Ann. Appl. Probab. \textbf{14} (2004), no.~3, 1353--1377. \MR{2071426}

\bibitem[DGM08]{MR2446322}
Randal Douc, Arnaud Guillin, and Eric Moulines, \emph{Bounds on regeneration times and limit theorems for subgeometric {M}arkov chains}, Ann. Inst. Henri Poincar\'{e} Probab. Stat. \textbf{44} (2008), no.~2, 239--257. \MR{2446322}

\bibitem[DL18]{MR3843830}
George Deligiannidis and Anthony Lee, \emph{Which ergodic averages have finite asymptotic variance?}, Ann. Appl. Probab. \textbf{28} (2018), no.~4, 2309--2334. \MR{3843830}

\bibitem[DMPS18]{MarkovChains}
Randal Douc, Eric Moulines, Pierre Priouret, and Philippe Soulier, \emph{Markov chains}, Springer Series in Operations Research and Financial Engineering, Springer, Cham, 2018. \MR{3889011}

\bibitem[FM00]{MR1796485}
Gersende Fort and Eric Moulines, \emph{{$V$}-subgeometric ergodicity for a {H}astings-{M}etropolis algorithm}, Statist. Probab. Lett. \textbf{49} (2000), no.~4, 401--410. \MR{1796485}

\bibitem[GCS{\etalchar{+}}14]{MR3235677}
Andrew Gelman, John~B. Carlin, Hal~S. Stern, David~B. Dunson, Aki Vehtari, and Donald~B. Rubin, \emph{Bayesian data analysis}, third ed., Texts in Statistical Science Series, CRC Press, Boca Raton, FL, 2014. \MR{3235677}

\bibitem[Hai09]{MR2540073}
Martin Hairer, \emph{How hot can a heat bath get?}, Comm. Math. Phys. \textbf{292} (2009), no.~1, 131--177. \MR{2540073}

\bibitem[HMW21]{Majka21}
Lu-Jing Huang, Mateusz~B. Majka, and Jian Wang, \emph{Approximation of heavy-tailed distributions via stable-driven {SDE}s}, Bernoulli \textbf{27} (2021), no.~3, 2040--2068. \MR{4278802}

\bibitem[Jon04]{MR2068475}
Galin~L. Jones, \emph{On the {M}arkov chain central limit theorem}, Probab. Surv. \textbf{1} (2004), 299--320. \MR{2068475}

\bibitem[JR02]{MR1890063}
S\o ren~F. Jarner and Gareth~O. Roberts, \emph{Polynomial convergence rates of {M}arkov chains}, Ann. Appl. Probab. \textbf{12} (2002), no.~1, 224--247. \MR{1890063}

\bibitem[JR07]{Roberts07}
\bysame, \emph{Convergence of heavy-tailed {M}onte {C}arlo {M}arkov chain algorithms}, Scand. J. Statist. \textbf{34} (2007), no.~4, 781--815. \MR{2396939}

\bibitem[JT03]{MR1996270}
S\o ren~F. Jarner and Richard~L. Tweedie, \emph{Necessary conditions for geometric and polynomial ergodicity of random-walk-type {M}arkov chains}, Bernoulli \textbf{9} (2003), no.~4, 559--578. \MR{1996270}

\bibitem[KM12]{MR2981426}
I.~Kontoyiannis and S.~P. Meyn, \emph{Geometric ergodicity and the spectral gap of non-reversible {M}arkov chains}, Probab. Theory Related Fields \textbf{154} (2012), no.~1-2, 327--339. \MR{2981426}

\bibitem[KV86]{MR834478}
C.~Kipnis and S.~R.~S. Varadhan, \emph{Central limit theorem for additive functionals of reversible {M}arkov processes and applications to simple exclusions}, Comm. Math. Phys. \textbf{104} (1986), no.~1, 1--19. \MR{834478}

\bibitem[Maj17]{Majka17}
Mateusz~B. Majka, \emph{Coupling and exponential ergodicity for stochastic differential equations driven by {L}\'{e}vy processes}, Stochastic Process. Appl. \textbf{127} (2017), no.~12, 4083--4125. \MR{3718107}

\bibitem[MMW21]{Menshikov21}
Mikhail~V. Menshikov, Aleksandar Mijatovi\'{c}, and Andrew~R. Wade, \emph{Reflecting random walks in curvilinear wedges}, In and out of equilibrium 3. {C}elebrating {V}ladas {S}idoravicius, Progr. Probab., vol.~77, Birkh\"{a}user/Springer, Cham, [2021] \copyright 2021, pp.~637--675. \MR{4237286}

\bibitem[MT93]{meynadntweedie}
S.~P. Meyn and R.~L. Tweedie, \emph{Markov chains and stochastic stability}, Communications and Control Engineering Series, Springer-Verlag London, Ltd., London, 1993. \MR{1287609}

\bibitem[MT94]{meyn1994computable}
Sean~P Meyn and Robert~L Tweedie, \emph{Computable bounds for geometric convergence rates of markov chains}, The Annals of Applied Probability (1994), 981--1011.

\bibitem[MT96]{Mengersen96}
K.~L. Mengersen and R.~L. Tweedie, \emph{Rates of convergence of the {H}astings and {M}etropolis algorithms}, Ann. Statist. \textbf{24} (1996), no.~1, 101--121. \MR{1389882}

\bibitem[MW00]{MR1782272}
Michael Maxwell and Michael Woodroofe, \emph{Central limit theorems for additive functionals of {M}arkov chains}, Ann. Probab. \textbf{28} (2000), no.~2, 713--724. \MR{1782272}

\bibitem[NSR19]{nguyen2019non}
Than~Huy Nguyen, Umut Simsekli, and Ga{\"e}l Richard, \emph{Non-asymptotic analysis of fractional langevin monte carlo for non-convex optimization}, International Conference on Machine Learning, PMLR, 2019, pp.~4810--4819.

\bibitem[Num84]{MR776608}
Esa Nummelin, \emph{General irreducible {M}arkov chains and nonnegative operators}, Cambridge Tracts in Mathematics, vol.~83, Cambridge University Press, Cambridge, 1984. \MR{776608}

\bibitem[Ros95]{MR1340509}
Jeffrey~S. Rosenthal, \emph{Minorization conditions and convergence rates for {M}arkov chain {M}onte {C}arlo}, J. Amer. Statist. Assoc. \textbf{90} (1995), no.~430, 558--566. \MR{1340509}

\bibitem[RT96]{MR1440273}
Gareth~O. Roberts and Richard~L. Tweedie, \emph{Exponential convergence of {L}angevin distributions and their discrete approximations}, Bernoulli \textbf{2} (1996), no.~4, 341--363. \MR{1440273}

\bibitem[RW01]{MR1856277}
Michael R\"{o}ckner and Feng-Yu Wang, \emph{Weak {P}oincar\'{e} inequalities and {$L^2$}-convergence rates of {M}arkov semigroups}, J. Funct. Anal. \textbf{185} (2001), no.~2, 564--603. \MR{1856277}

\bibitem[SAP22]{MR4429313}
Nikola Sandri\'{c}, Ari Arapostathis, and Guodong Pang, \emph{Subexponential upper and lower bounds in {W}asserstein distance for {M}arkov processes}, Appl. Math. Optim. \textbf{85} (2022), no.~3, Paper No. 24, 45. \MR{4429313}

\bibitem[Sat13]{MR3185174}
Ken-iti Sato, \emph{L\'{e}vy processes and infinitely divisible distributions}, Cambridge Studies in Advanced Mathematics, vol.~68, Cambridge University Press, Cambridge, 2013, Translated from the 1990 Japanese original, Revised edition of the 1999 English translation. \MR{3185174}

\bibitem[SZTG20]{simsekli2020fractional}
Umut Simsekli, Lingjiong Zhu, Yee~Whye Teh, and Mert Gurbuzbalaban, \emph{Fractional underdamped langevin dynamics: Retargeting sgd with momentum under heavy-tailed gradient noise}, International conference on machine learning, PMLR, 2020, pp.~8970--8980.

\bibitem[TT94]{MR1285459}
Pekka Tuominen and Richard~L. Tweedie, \emph{Subgeometric rates of convergence of {$f$}-ergodic {M}arkov chains}, Adv. in Appl. Probab. \textbf{26} (1994), no.~3, 775--798. \MR{1285459}

\bibitem[Vil09]{Villani09}
C\'{e}dric Villani, \emph{Optimal transport}, Grundlehren der mathematischen Wissenschaften [Fundamental Principles of Mathematical Sciences], vol. 338, Springer-Verlag, Berlin, 2009, Old and new. \MR{2459454}

\bibitem[VSG{\etalchar{+}}24]{MR4749108}
Aki Vehtari, Daniel Simpson, Andrew Gelman, Yuling Yao, and Jonah Gabry, \emph{Pareto smoothed importance sampling}, J. Mach. Learn. Res. \textbf{25} (2024), Paper No. [72], 58. \MR{4749108}

\bibitem[ZZ23]{MR4580902}
Xiaolong Zhang and Xicheng Zhang, \emph{Ergodicity of supercritical {SDE}s driven by {$\alpha$}-stable processes and heavy-tailed sampling}, Bernoulli \textbf{29} (2023), no.~3, 1933--1958. \MR{4580902}

\end{thebibliography}
\bibliographystyle{amsalpha}
\end{document}